\documentclass[10pt,twoside,english,reqno,a4paper,scrartcl]{amsart}

\usepackage{listings,graphicx,amsmath,varioref,amscd,amssymb,color,bm,stmaryrd,amsthm,amsfonts,graphics,geometry,latexsym,pgf,pst-all, cancel} 
\theoremstyle{plain}

\usepackage{amsthm}
\theoremstyle{plain}
\newtheorem{theorem}{Theorem}[section]
\newtheorem{proposition}[theorem]{Proposition}
\newtheorem{Proposition}[theorem]{Proposition}
\newtheorem{lemma}[theorem]{Lemma}

 \DeclareOldFontCommand{\rm}{\normalfont\rmfamily}{\mathrm}
\usepackage{multirow}

\usepackage{graphics}
\usepackage{pgfplots}
\pgfplotsset{width=7cm,compat=newest} 
\usepackage{graphics}
\usepackage{tikz}
\usetikzlibrary{shapes.geometric}
\usepackage{tikz-cd}
\usepackage[labelsep=endash]{caption}
\usepackage{float}

\font\manual=manfnt
\newcommand\xqed[1]{%
  \leavevmode\unskip\penalty9999 \hbox{}\nobreak\hfill
  \quad\hbox{#1}}
\newcommand\triang{\xqed{\manual\char'170}}

\usepackage{geometry}
\usepackage{fouriernc}
\usepackage[T1]{fontenc}

\usepackage[colorlinks=false]{hyperref}

\theoremstyle{definition}
\newtheorem{defin}[theorem]{Definition}

\newtheorem{remark}[theorem]{Remark}
\newtheorem{example}[theorem]{Example}

\theoremstyle{remark}

\def\luo{L^{1}(\Omega)}

\def\dys{\displaystyle}
\numberwithin{equation}{section}
\def\huz{H^{1}_{0}(\Omega)}
\def\dis{\displaystyle}
\def\supp{\text{\text{supp}}}
\DeclareMathOperator{\diver}{div}

\DeclareMathOperator{\R}{\mathbb{R}}

\DeclareMathOperator{\io}{\int_\Omega}

\newcommand{\car}[1]{\raise1pt\hbox{$\chi$}_{#1}}

\def\rn{\mathbb{R}^N}

\def\XXint#1#2#3{{\setbox0=\hbox{$#1{#2#3}{\int}$ }
		\vcenter{\hbox{$#2#3$ }}\kern-.6\wd0}}

\begin{document}
\title[Singular Elliptic PDEs]{Singular Elliptic PDEs: an extensive overview}

\address[Francescantonio Oliva]{francescantonio.oliva@uniroma1.it}
\author[F. Oliva]{Francescantonio Oliva}
\address[Francesco Petitta]{francesco.petitta@uniroma1.it}

\author[F. Petitta]{Francesco Petitta}
\address{Dipartimento di Scienze di Base e Applicate per l' Ingegneria, ``Sapienza" Universit\`a di Roma, Via Scarpa 16, 00161 Rome, Italy}

\makeatletter
\@namedef{subjclassname@2020}{%
  \textup{2020} Mathematics Subject Classification}
\makeatother


\begin{abstract}
In this survey we provide an overview of nonlinear elliptic homogeneous  boundary value problems  featuring singular zero-order terms with respect to the unknown variable whose prototype equation is 
$$
-\Delta u = {u^{-\gamma}} \ \text{in}\ \Omega 
$$
where  $\Omega$  is  a bounded subset of $\mathbb{R}^N$ ($N\geq 2$), and $\gamma>0$.  

\smallskip 

We start by outlining the basic  concepts and the mathematical framework needed for setting  the problem. Both old and new   key existence and uniqueness results are presented, alongside regularity issues depending on the regularity of the data.   
The presentation aims to be modern,  self-contained and consistent.  Some examples and open problems are also discussed. 
  
\end{abstract}

\maketitle

\tableofcontents

\section{Introduction}
This survey aims to represent a  rundown of results and techniques in the study of  problems whose prototype is 
\begin{equation}
\begin{cases}
\displaystyle - \Delta u= \frac{f(x)}{u^\gamma} &  \text{in}\ \Omega, \\
u\geq 0&  \text{in}\ \Omega, \\
u=0 & \text{on}\ \partial \Omega,
\label{pbintro}
\end{cases}
\end{equation}
that presents a singularity  when the solutions $u$ reaches the value zero. Here     $\Omega\subset \rn$ is  an open bounded set ($N\geq 2$), $f$ is a nonnegative function, and $\gamma>0$.   
In last decades these problems have gained growing interest both for purely theoretical tricky issues and for his intimate connections with applications.

\medskip
Physical motivations for studying problems like \eqref{pbintro} arise in various fields. For instance, in the study of thermo-conductivity, where ${u^{\gamma}}$ represents the resistivity of the material, in signal transmissions,  in  boundary layer models,  and in the theory of non-Newtonian pseudoplastic fluids (see Section \ref{sec2} below). 

\medskip
From a purely theoretical standpoint, after the pioneering existence and uniqueness results presented in \cite{fulks}, a systematic treatment of problems like \eqref{pbintro} began with \cite{stuart, crt}.

 If $f$ is smooth enough (say Hölder continuous) and bounded away from zero on $\Omega$, then the existence and uniqueness of a classical solution to \eqref{pbintro} is established by desingularizing the problem and applying a suitable sub- and super-solution method.  Significant refinements are then  provided in \cite{lm}; specifically, the authors demonstrate that $u \not \in C^{1}(\overline{\Omega})$ if $\gamma > 1$, and $u$ has finite energy, i.e. $u \in \huz$, if and only if $\gamma < 3$ (see also \cite{ghl,dp} for further insights).
In case of a more general datum $f\in L^{m}(\Omega)$   the previous result has been extended in \cite{OP} proving that a solution  of problem \eqref{pbintro} does exist for any such data $f$ if and only if $\gamma<3-\frac{2}{m}$.

The classical theory for equations like \eqref{pbintro}, also known as singular Lane–Emden–Fowler equations, has been extended to cases where the term ${s^{-\gamma}}$ is replaced by a $C^{1}$ non-increasing nonlinearity $h(s)$ that blows up at zero at a specific rate (see \cite{stuart, crt, ls, zach}).

\medskip

  Now let us turn the attention  to the case in which $f$ is a nonnegative function in some $L^m(\Omega)$ ($m \geq 1$) or possibly even a measure.  If $f \in \luo$, \cite{bo} proves the existence of a distributional solution $u$ to \eqref{pbintro}. Specifically, the authors  show  that a locally strictly positive function $u$ exists such that the equation in \eqref{pbintro} is satisfied in the distributional sense: moreover, $u \in W^{1,1}_0(\Omega)$ if $\gamma < 1$, $u \in \huz$ if $\gamma = 1$, and $u \in H^1_{\rm loc}(\Omega)$ if $\gamma > 1$  where, in the latter case, the boundary data are assumed only in a weaker sense than the usual trace sense, i.e. $u^{\frac{\gamma+1}{2}} \in H^1_{0}(\Omega)$. Note that if $\gamma > 1$, solutions with infinite energy do exist, even for smooth data \cite{lm}. Additionally, in \cite{gmm, GMM, OP},   existence and uniqueness of finite energy solutions are considered for $f \in L^{\frac{2N}{N+2}}(\Omega)$, even in the case of a continuous nonlinearity $h(s)$ that mimics ${s^{-\gamma}}$.

When $f$ is a measure, the situation becomes  remarkably  different. Nonexistence of solutions to problem \eqref{pbintro} is proven, at least in the sense of approximating sequences in \cite{bo} if the measure is too concentrated. Conversely, sharp existence results are obtained in \cite{orpe} if the measure is diffuse; here, concentration and diffusion are understood in terms of capacity. For general, possibly singular, measure data, the existence of a distributional/renormalized solution is considered in \cite{do, ddo}, even for a more general (not necessarily monotone) nonlinearity;  the renormalized solution is also shown to be unique in presence of a non-increasing nonlinearity and of measures which are purely diffuse.

\medskip
The aim of the present work is to provide a comprehensive and self-contained survey on  the most noteworthy previous results we mentioned and to present them in a more unified and modern fashion.  Most of the proofs we perform, also those of  classical results, are revised  in their form and some of them are simplified. The prerequisites for this survey are basic knowledge in differential calculus, Lebesgue integration theory, and Sobolev spaces. 

\medskip
{\bf Here is what is included in this survey.} The  plan of this paper is the following:

\medskip 
In Section \ref{11}    
we introduce  the necessary background information, including key preliminaries and the notation used throughout the paper.

\medskip 
Section \ref{sec2}
is devoted to offer a brief description of some of the background mathematical models from which singular problems arise and to  discuss their practical applications. 

\medskip 
In Section \ref{sec:classical} we   present the classical theory of singular elliptic equations, and in particular, existence and uniqueness of classical solutions (Section \ref{sec:lm}) and their regularity (Section \ref{regLM}).

\medskip 
Section \ref{sec:integrable}   shifts focus to solutions in the sense of distributions of semilinear problems as in  \eqref{pbintro} , especially when dealing with data  belonging to some Lebesgue space. 

 Existence of distributional solutions for a merely integrable data is discussed in Section  \ref{exidis}, while Section \ref{gennon} is devoted to extend our results to cases involving a more general nonlinearity as lower order term.

 We then establish  that there is only one finite energy distributional solution of problems as in \eqref{pbintro} (this is done  in Section \ref{sec:uniquenessweak})

\medskip

Section \ref{sec:reg} addresses the relevant question of the weak regularity of the solution. In particular we are interested in discussing the interplay between    the singular nonlinearity and the datum in order to expect finite energy solutions to problems as in \eqref{pbintro}.

\medskip 
The previous results and considerations are pushed on in Section \ref{sec:measure} where one  deals with more general principal operators (e.g. $p$-Laplacian, Leray-Lions type monotone operators) and measure data.  Here we introduce the notion of renormalized solutions for such kind of problems providing a brief overview of the truncations methods in a very classical context.   Then we prove existence of such solutions  under suitable diffusion properties of the measure datum. 

In Section \ref{sec:nonex} instead we present a nonexistence result for measure  data  too concentrated (with respect to the capacity) provided the lower order term vanishes at infinity (e.g. $h(s)=s^{-\gamma}$).  We also prove (Section \ref{um1}) that such renormalized solutions are unique provided    the singular nonlinearity is monotone non-increasing and not to explosive at $s=0$ (say $h(s)=s^{-\gamma}, \gamma\leq 1$).

\medskip 
Weak regularity property are also considered (Section \ref{regup}) in the  general framework of monotone leading operators depending on both the behaviour of the singular nonlinear term and on the data. 

\medskip 
Finally, Section \ref{unilinear} faces the question of uniqueness for general monotone non-decreasing singular nonlinear terms and measure data. For the sake of exposition here  we give the results restricted to the semilinear case.   
 
 \medskip 
In the appendices we provide useful information and details on  the theory of nonlinear capacity and, for the sake of completeness,  and due to its relevance in our arguments,    we present   the 
 Hopf Lemma in its weakest form.

\medskip
{\bf What we will not discuss. } 
In this survey,  despite the extensive and rapidly evolving literature on the subject, for the sake of exposition we will not face some relevant problems related to singular equations.  Among the others, we will not discuss here  singular equations with first  order terms with natural growth  in the gradient (see for instance \cite{B,a6,ABLP, dura, fms2,Ol}), equations with regularizing absorption terms (\cite{ol2,cmmt}), equations with singular type reaction (\cite{DMO,OPS}), singular concave-convex problems (\cite{CoPa, gisa, gss}), and, in broad terms, the possible variational approaches to singular problems (\cite{cade}). Moreover, we will not discuss singular nonlocal equations (\cite{CoCo,ags, bdmp}), singular equations for double phase, anisotropic operators and   variable exponent problems (\cite{cma,gapa, parazh, bpz, csv}), singular parabolic problems (\cite{OP2}), singular Monge-Ampère problems (\cite{mo,LeSa}),  existence and classification of solutions in unbounded domains (\cite{mms,mms2}),  and problems arising from Lagrangians with growth 1 (\cite{DGOP,lops}). This latter topic will be addressed in a subsequent work. As homogeneous Dirichlet boundary conditions represent the critical case of singular problems, different type of boundary  data (such as Neumann, Robin, mixed and periodic) will not be discussed here either. 
 Without aiming to be exhaustive, we also refer the reader to the following papers and the references cited therein, which explores various relevant extensions, refinements, and related topics  \cite{t, gow,  hema, ghra,CoCo2, dhr, c,  bougia,  gps2,gps1, bct,sz, S, Ty,CST, edr,CaMoScSq, EsSc,  DCA, SoTe, GoGu,  SaGhCh, GuMaMo, GuMa, Ga, dugi, FeMeSe, Ha, DeOlSe, prr}.

\medskip \medskip \medskip 

\subsection{Basics and notations}\label{11}

For the entire paper $\Omega$ will be  an open bounded subset of $\mathbb{R}^N$, $N\geq 2$; when some regularity on its boundary is required it will be mentioned.   If $\Omega$ is smooth, with a little abuse of notation, in order to avoid technicalities and
without loosing generality, we will  refer to $d(x)$ as a suitable positive smooth (say $C^1$) modification of the distance function which agrees with it in a neighbourhood of $\partial\Omega$. We also  need to define the following $\varepsilon$-neighborhood of $\partial\Omega$:
\begin{equation}\label{not:omegaeps}\Omega_{\varepsilon}=\{x\in\Omega : d(x)<\varepsilon\},\end{equation}
that we will always assume to be smooth (up to the choice of a suitable small $\varepsilon$). Observe that, with this agreement, one has that $|\Omega_\varepsilon|\sim \varepsilon$.

\begin{figure}[htbp]\centering
\includegraphics[width=3in]{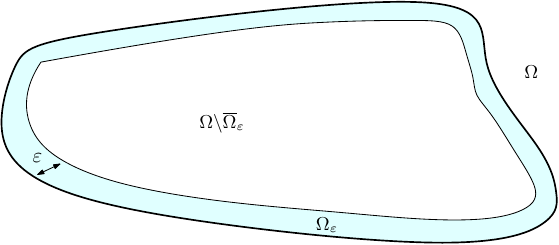}
\caption{The $\varepsilon$-neighborhood of $\partial\Omega$}\label{oe}
\end{figure}

Moreover  $B_R(x_0)$  will always represent   the ball of  radius $R>0$ in $\rn$ centered at $x_0$. 

 We denote by $C^1_0({\Omega})$ (resp. $C^\infty_0({\Omega})$) the set of functions in $C^{1}(\Omega)$ (resp. $C^{\infty}(\Omega)$) with compact support in $\Omega$, while  $C^1_0(\overline{\Omega})$ will represent the set of functions $\varphi\in C^1(\Omega)$ such that $\varphi = 0$ on $\partial \Omega$.

\medskip 
We recall that, for a given $p > 1$, the Hölder conjugate $ p' $ is defined as follows:
\[
\frac{1}{p} + \frac{1}{p'} = 1 \quad \text{or equivalently} \quad p' = \frac{p}{p-1}.
\]   
 For any  positive $\varepsilon$ we have:
$$
ab\leq \varepsilon^{p}\frac{a^p }{p} +\frac{1}{\varepsilon^{p'}}\frac{b^{p'}}{p'}, \ \ \ \ \forall a,b>0,
$$
which is  the  \underline{Generalized Young inequality.}

We will make use of the following simplified  notation for integrals: 
$$\int_{\Omega}f(x)\ dx := \int_{\Omega}f,$$
when no ambiguity is possible. 

\medskip 
We systematically use most basic tools in Lebesgue theory. Among them: 
\begin{enumerate}

\item \underline{H\"{o}lder's inequality:} for $1<p<\infty$, $p'=\frac{p}{p-1}$, we have, for every $f\in L^{p}(\Omega)$ and every  $g\in L^{p'}(\Omega)$:
$$
\dys\int_{\Omega} |fg|\leq \left(\int_{\Omega} |f|^p \right)^{\frac{1}{p}}  \left(\int_{\Omega} |g|^{p'} \right)^{\frac{1}{p'}}\, .
$$
\item  Let $1<p<\infty$,  $p'=\frac{p}{p-1}$, $f_n\in  L^p (\Omega)$, $g_n \in L^{p'} (\Omega)$ be such that $f_n$ strongly converges to $f$ in $ L^p (\Omega)$ and $g_n $ weakly converges to $g$ in $L^{p'}(\Omega)$. Then
$$
\dys\lim_{n\to \infty}\int_{\Omega} f_n \, g_n =\int_{\Omega} f g\,.
$$
The same conclusion holds true if $p=1$, $p'=\infty$ and the weak convergence of $g_n$ is replaced by the $\ast$-weak convergence in $ L^\infty (\Omega)$.  

\item  Let $f_n $ converge to $f$ in measure and suppose that:
$$
\exists C>0, \ \ q>1:\ \ \|f_n \|_{L^{q}(\Omega)}\leq C, \ \ \forall  n. 
$$
Then
$$f_n \longrightarrow f\ \ \ \text{strongly in } L^{s}(\Omega), \ \text{for every } 1\leq s<q.
$$
\item \underline{Fatou's Lemma:}   let $f_n  \geq 0$ be a sequence of measurable functions such that $f_n\to f $ a.e. in $\Omega$, then
$$
\int_{\Omega} f\leq\liminf_{n\to \infty}\int_{\Omega} f_n.
$$
\item   \underline{Generalized dominated convergence Lebesgue Theorem:}  Let $1\leq p<\infty$, and let $ f_n\in L^{p}(\Omega)$ be a sequence such that $f_n\to f $ a.e. in $\Omega$ and $|f_n |\leq g_n $ with $g_n$ strongly convergent in  $L^{p}(\Omega)$, then $f\in L^{p}(\Omega)$ and $f_n $ strongly converges to $f$ in $L^{p}(\Omega)$.

\end{enumerate}

\medskip 
We use standard notation for the Sobolev spaces  $W^{n,p}(\Omega)$ of functions in $L^{p}(\Omega)$, $p\geq 1$ with distributional derivatives in $L^{p}(\Omega)$  up to order $n$ (particular instances being $W^{0,p}(\Omega)=L^{p}(\Omega)$ and $W^{1,2}(\Omega)=H^{1}_0 (\Omega)$).  

\medskip

The Sobolev inequality for $ 1 \leq p < N $ states that there exists a constant $ \mathcal{S}_p $ such that for all functions $ u \in W^{1,p}(\mathbb{R}^N) $,
\[
\left( \int_{\mathbb{R}^N} |u|^{p^*} \right)^{\frac{1}{p^*}} \leq \mathcal{S}_p \left( \int_{\mathbb{R}^N} |\nabla u|^p \right)^{\frac{1}{p}},
\]
where $ p^* = \frac{Np}{N-p} $ is the Sobolev conjugate of $ p $.

\medskip 
For a fixed $k>0$, we define the truncation functions $T_{k}:\R\to\R$ and $G_{k}:\R\to\R$ as follows 
\begin{align*}
T_k(s):=&\max (-k,\min (s,k)),\\
G_k(s):=&(|s|-k)^+ \operatorname{sign}(s).
\end{align*}

\begin{figure}[htbp]\centering
\includegraphics[width=2in]{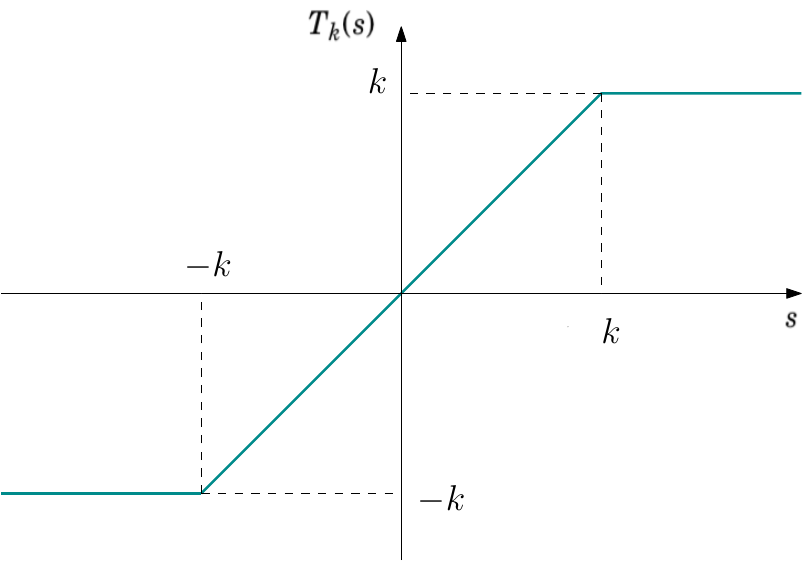}\ \ \ \includegraphics[width=2in]{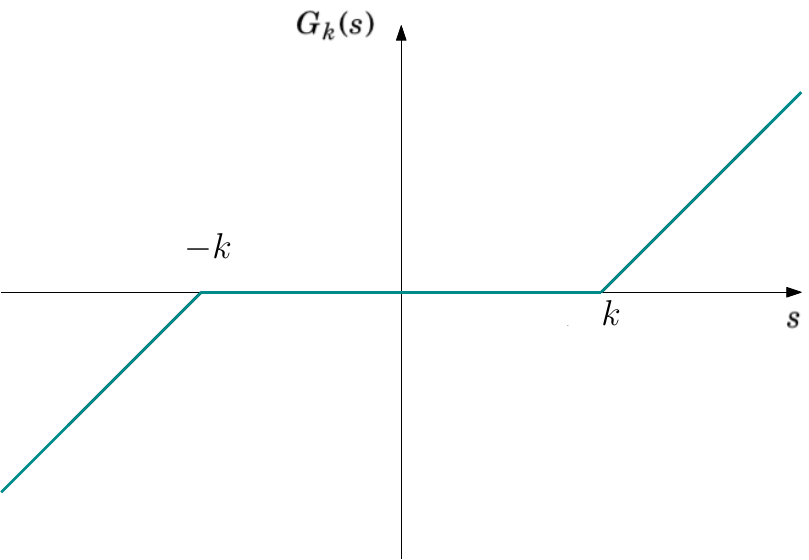}
\caption{The truncation functions $T_k(s)$ and $G_k(s)$}\label{tkgk}
\end{figure}

We will also use the following auxiliary  functions
\begin{align}\label{not:Vdelta}
\displaystyle
V_{\delta}(s):=
\begin{cases}
1 \ \ &s\le \delta, \\
\displaystyle\frac{2\delta-s}{\delta} \ \ &\delta <s< 2\delta, \\
0 \ \ &s\ge 2\delta,
\end{cases}
\end{align}
and 
\begin{equation}\label{not:pi}\pi_\delta(s)=1-V_\delta(s),\end{equation}
\begin{figure}[htbp]\centering
\includegraphics[width=2in]{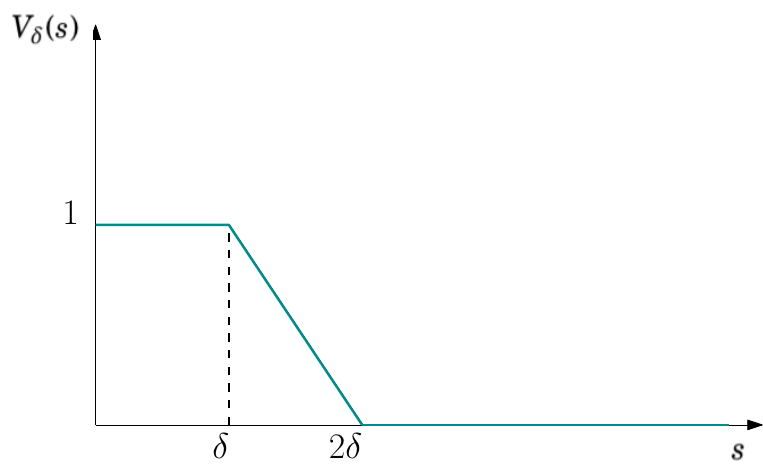} \ \ \includegraphics[width=2in]{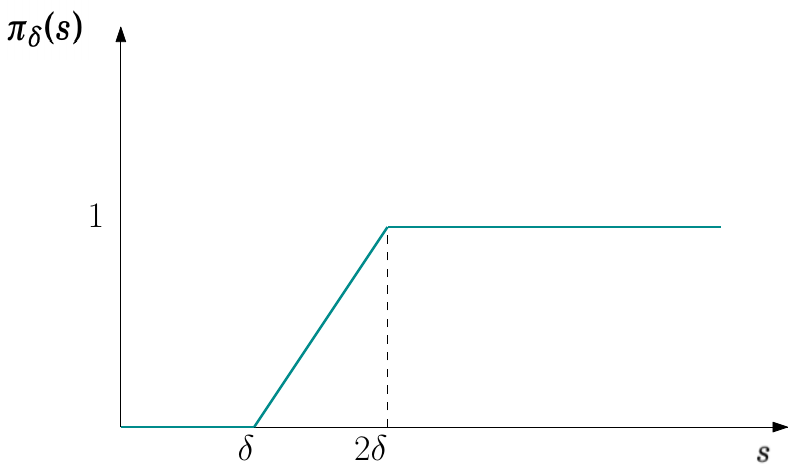}
\caption{The auxiliary functions $V_\delta (s)$ and $\pi_\delta (s)$}\label{vdpd}
\end{figure}

\medskip 

In the following we denote by $\varphi_{1,A}\in W^{2,q}(\Omega)\cap H^1_0(\Omega)$ for any $1<q<\infty$ the positive solution to the  problem  
\begin{equation}
\begin{cases}
\displaystyle -\operatorname{div}(A(x) \nabla \varphi_{1,A})= \lambda_1 \varphi_{1,A} &  \text{in}\ \Omega, \\
\varphi_{1,A}=0 & \text{on}\ \partial \Omega,
\end{cases}\label{not:phi1}
\end{equation}
where $A$ is  a symmetric, elliptic and bounded matrix with coefficients $a_{ij}\in C^{0,1}(\overline{\Omega})$. Finally we denote by $\varphi_{1,p}\in C^1(\overline{\Omega})$ the first positive eigenfunction associated to the $p$-Laplacian, i.e. $\varphi_{1,p}$ solves
\begin{equation}
	\begin{cases}
		\displaystyle -\operatorname{div}(|\nabla \varphi_{1,p}|^{p-2}\nabla \varphi_{1,p})= \lambda_1 \varphi_{1,p}^{p-1} &  \text{in}\ \Omega, \\
		\varphi_{1,p}=0 & \text{on}\ \partial \Omega.
	\end{cases}\label{not:phi1p}
\end{equation}
In the sequel $\varphi_1 := \varphi_{1,2} \in C^{\infty}_0(\overline{\Omega})$ is the first eigenfunction of the Laplacian.

\medskip

We refer to the space of Radon measures as $\mathcal{M}(\Omega)$ and by  $\mathcal{M}_{\rm loc}(\Omega)$ the space of Radon measure which are locally finite in $\Omega$. It is worth mentioning that we will make the following abuse of notation in presence of an integral with respect to a general measure
$$\int_{\Omega}\varphi \mu := \int_{\Omega}\varphi d\mu,$$
when no ambiguity is possible.

\medskip

We refer to the Lebesgue space with respect to a measure $\mu$ as $L^q(\Omega,\mu)$.
Moreover we denote by $\mathcal{M}(\Omega,d)$ as the subspace of the $\mu\in\mathcal{M}_{\rm{loc}}(\Omega)$ satisfying
$$\int_{\Omega} d(x) \mu <\infty.$$

By $\operatorname{cap}_p$ we denote the standard $p$-capacity which is an exterior measure.  

 Here   we only fixed the notations. For the sake of completeness, in  Appendix \ref{app:radon} below we present a more detailed    presentation of these tools.

\medskip 

For any $0<r<\infty$, by $M^r(\Omega)$ we denote the Marcinkiewicz   (or weak Lebesgue, or Lorentz)  space,   which is the space of functions $f$
such that $|\{|f|>t\}|\le C t^{-r}$, for any $t>0$. If $|\Omega|<\infty$, then $L^{r}(\Omega)\subset M^r(\Omega) \subset L^{r-\varepsilon}(\Omega)$, for any $0<\varepsilon \leq  r-1$. We refer to \cite{hunt} for more details  on Lorentz   spaces.

\medskip 

Finally we explicitly remark that, if no otherwise specified, we will denote by $C,c$ several positive constants whose value may change from line to line and, sometimes, on the same line. These values will only depend on the data but they will never depend on the indexes of the sequences we will introduce.

\section{Derivation of some models and applications}\label{sec2}
Up to our knowledge, physical motivations for studying \eqref{pbintro} comes from $'60$s when singular problems explicitly appear for the first time in literature. Precisely, in \cite{fulks}, the authors prove an existence result for a singular equation of the following type
\begin{equation}\label{modello}
cu_t -k\Delta u= \frac{E^2(x,t)}{q(u)} \ \ \text{in} \ \ \Omega \times (0,T),
\end{equation}
assuming some initial and boundary data, where  $T$ is a positive constant and $\Omega$ is a three dimensional region occupied by an electrical conductor. Equation \eqref{modello} is derived from the Fourier's law;   as the current is passing through the body, we have a source of heat in each point of $\Omega$. Here $u(x,t)$ is the temperature at the time $t$ in the position $x\in \Omega$ and $c, k$ are respectively the specific heat capacity of the conductor (i.e. a measure of the ability to preserve the heat) and thermal conductivity of $\Omega$ (i.e. a measure of the capability of the conductor to transfer heat). Moreover, $q(u)$ is the resistivity while the term $E(x, t)$ describes the local voltage drop in $\Omega$ given as a function of position and time. In this setting   $\frac{E^2(x,t)}{q(u)}$ can be seen as the rate of generation of heat.
A realistic example of $q(u)$ is the function $C u$, for some $C>0$,  so that equation \eqref{pbintro} (with $\gamma=1$) can be seen as a stationary case of \eqref{modello} when equipped with an homogeneous boundary condition (i.e. a reference zero temperature outside the body). 

\medskip 
Equations as in \eqref{pbintro}
that stand, in general, for a  generalized version of the so-called Lane-Emden-Fowler equation  arise also in the  theory of  gaseous dynamics in
astrophysics (\cite{fo}), in signal transmissions (\cite{nowo}), relativistic mechanics, nuclear physics and in the study of chemical reactions, glacial advance, in transport of coal slurries down conveyor belts and in several other geophysical and industrial applications (see  \cite{wo,gr}, and references therein).

\medskip

One of the most relevant example of application of singular equations as \eqref{pbintro} appear, among many other models of boundary layers, in the theory of non-Newtonian fluids and in particular the  pseudo-plastic fluids (\cite{Sch, asp, nc, vsm}). 

Prandtl introduced boundary layer theory in 1904 to understand the flow of a slightly viscous fluid near the boundary of a solid body. This theory  was  used by Blasius in \cite{Bla} in order to investigate the case of a flow of a viscous incompressible fluid past a plate region (or a smooth one). Later on  in \cite{nc,nt}  the study flow of a slightly viscous fluid over a moving conveyor belt was proposed. Many other application of the boundary layer theory appear through the years as in the study of laminar and turbulent boundary layer model for aerodynamics. 

\begin{figure}[htbp]\centering
\includegraphics[width=3in]{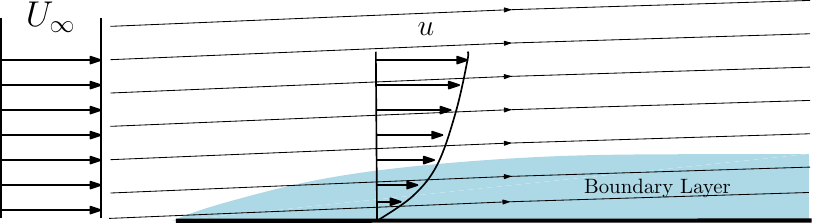}
\caption{Blasius model of a fluid on a plate in the boundary layer regime}\label{conef}
\end{figure}

 \medskip 
For the sake of exposition let us briefly summarize the derivation of the Blasius model.

Let $x$ represent the coordinate of the fluid  motion  past a plate and $y$ the transverse coordinate. Under the boundary layer approximation, if  the pressure of the fluid in the direction $x$ is assumed to be constant,  then from the stationary Navier-Stokes equation for an incompressible fluid one formally derives 
\begin{equation} \label{ins0}
u u_x+v u_y =  \frac{\tau_y}{\rho}, 
\end{equation}
coupled with the continuity equation 
\begin{equation} \label{ce}
 u_x+ v_y = 0,
\end{equation}
which is the so-called boundary layer equations for steady flow over a semi-infinite flat plate. 
Here $v$ is the velocity of the fluid in the orthogonal component to the plate  and $u$ is the parallel one,  $\rho$ is the density of the fluid while $\tau$ is the shear stress. 
Natural boundary conditions are also assumed 
\begin{equation*} \label{bci}
u(x,0)=v(x,0)=0, \ \ \ \text{and} \ \ \ u(x,\infty)=U_\infty\,, 
\end{equation*}
where $U_\infty$ is the free stream  velocity, i.e.  the velocity of the fluid  at the moment the fluid touches the plate, which remains uniform far from the boundary layer.

In case of Newtonian fluids, the shear stress $\tau$, coplanar with the  plate surface, is standardly assumed  to be proportional to  the strain rate $u_y$, i.e. $\tau=c u_y$, for some constant $c>0$.

A general law to link  the shear stress at the boundary layer and the strain rate is much more complicated in case of so-called  non-Newtonian fluids; experimental evidences (see for instance \cite{Met}) show that, in many cases, a  \underline{power law} is suitable to do the job,  for  instance 
 \begin{equation} \label{pl}
\tau = k |u_y|^{n-1}u_y\,, \ \ n>0,  
\end{equation}
where $k>0$ is the flow consistency index, i.e. $k |u_y|^{n-1}$ represents the viscosity (here, for historical reasons we use the outdated symbol $n$ to indicate the real positive power appearing in \eqref{pl}).  These type of nonlinear power laws, also called  Ostwald-de Waele relation, are the ones that, in the purely \underline{free stream}  case,  lead to the study of classical $p$-Laplacian problems (with $p=n+1$).  If $p=2$ (i.e. $n=1$) this coincide with the Newtonian  fluids  case (e.g.  water and air). 
On the other hand, Non-Newtonian fluids generally  fall into two main categories depending on  $n$: \underline{Pseudo-plastic}  ($n< 1$) and \underline{Dilatant} ($n> 1$). In terms of  $p$-Laplace models those correspond respectively to the singular case  ($p<2$) in which the viscosity is inversely proportional to the shear rate ("shear thinning fluids") and to the degenerate  ($p>2$) in which the viscosity increases as the shear rate increases ("shear thickening fluids"). 

\medskip 
Examples of dilatant  fluids are are so-called oobleck (a mixture of cornstarch and water)  and  wet sand which hardens upon application of high enough forces. 
Examples of pseudo-plastic fluids, which are the most relevant from the point of view of the applications,  are  latex paint,  blood, molten plastic, emulsions,    quicksand (try to go over it) and also sticky fluids as the  Ketchup (squeezing  decreases its viscosity). For these reasons, pseudo-plastic fluids enter in a variety of applications as in the study of glacial advance, transport of coal slurries,  or biofluid dynamics. Of a certain interest in applications are the Bingham plastic fluids in which a sort of "waiting time" appears in order for the shear stress to start the flow. In this case the related  power low reads  as $(\tau -\tau_0)^+=   k |u_y|^{n-1} u_y$, an example of this fluid being toothpaste. 

\begin{figure}[htbp]\centering
\includegraphics[width=2.5in]{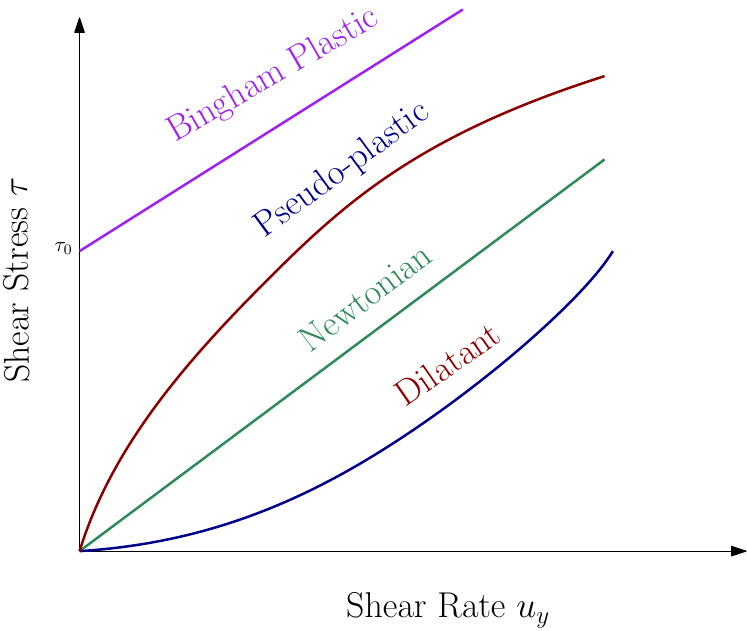}
\caption{Non-Newtonian fluid classification}\label{nonn}
\end{figure}

\medskip

Coming back to \eqref{pl} let us show how the laminar flows regime \eqref{ins0} and \eqref{ce} lead to an elliptic problem that may be singular with  respect to its unknown.  
Of course,   in the boundary layer approximation, one may think at a positive shear  rate $u_y$; that is 
 \begin{equation*} \label{pl1}
\tau =k u_y^{n}\,, \ \ n>0.  
\end{equation*}

Hence, \eqref{ins0} reads as 
\begin{equation} \label{ins0p}
u u_x+v u_y = \frac{k}{\rho}\frac{d}{dy}\left({ u_y^{n}}\right). 
\end{equation}

In order to encode the conservation of mass, i.e.  \eqref{ce},  we introduce the so-called stream function $\psi$ defined by 
$$
u=\psi_y\ \ \ \text{and}\ \ \ v=-\psi_x. 
$$

The main idea of Blasius then consists in performing a suitable change of variable that aims to normalize  both  the thickness of the boundary layer  and the velocity in a fixed box and to derive, by \eqref{ins0p}, a third  order ODE for the so-called non-dimensional stream function. 

Let us introduce the similarity parameter $\eta$ as 
\begin{equation*} \label{eta}
\eta:=y\left({\frac{U_\infty}{kx}}\right)^{\frac{1}{n+1}}, 
\end{equation*}
and the normalized stream function $f$ by 
\begin{equation*} \label{f}
f(\eta):=\frac{\psi(x,y)}{({U_\infty kx })^{\frac{1}{n+1}}}. 
\end{equation*}

We now formally rewrite all terms appearing in \eqref{ins0p} in terms of $\eta$ and $f$. 

We have 
\begin{equation} \label{iu}
u=\frac{d\psi}{dy}= \frac{d\psi}{d\eta}\frac{d\eta}{d y }= f'(\eta) ({U_\infty kx }  )^{\frac{1}{n+1}}\left({\frac{U_\infty}{kx}}\right)^{\frac{1}{n+1}}=U_{\infty}^{\frac{2}{n+1}}f'(\eta),
\end{equation}
and 
\begin{align} \label{vi}
& \nonumber v=-\frac{d\psi}{dx}=-\left(\frac{d f}{dx} ({U_{\infty} k x})^{\frac{1}{n+1}} + \frac{1}{n+1} \left({\frac{U_{\infty} k}{x^n}}\right)^{\frac{1}{n+1}} f(\eta) \right)\\ &=-\left(f'(\eta)\left(-\frac{1}{(n+1)x}\eta\right)({U_{\infty} k x})^{\frac{1}{n+1}}+ \frac{1}{n+1} \left({\frac{U_{\infty} k}{x^n}}\right)^{\frac{1}{n+1}}f(\eta) \right) = \frac{1}{n+1}  \left({\frac{U_{\infty} k}{x^n}}\right)^{\frac{1}{n+1}}(\eta f'(\eta)-f(\eta)). 
\end{align}

Moreover, 
\begin{equation} \label{ux}
u_x= U_{\infty}^{\frac{2}{n+1}} f''(\eta) \frac{d\eta}{dx}= -\frac{1}{(n+1)x}U_{\infty}^{\frac{2}{n+1}} \eta f''(\eta), 
\end{equation}
and  
\begin{equation} \label{uy}
u_y= U_{\infty}^{\frac{2}{n+1}} f''(\eta)\frac{d\eta}{dy}= U_{\infty}^{\frac{2}{n+1}} f''(\eta)\left({\frac{U_\infty}{kx}}\right)^{\frac{1}{n+1}}.
\end{equation}

\medskip

Now we gather \eqref{iu}--\eqref{uy} into \eqref{ins0p} in order to get
$$
-\frac{U^{\frac{4}{n+1}}_{\infty}f'(\eta)f''(\eta)\eta}{(n+1)x} +  \frac{U^{\frac{4}{n+1}}_{\infty}}{(n+1)x}f''(\eta) (\eta f'(\eta) -  f(\eta))=  \frac{k}{\rho} \frac{U^{\frac{3n}{n+1}}_{\infty}}{ (kx)^{\frac{n}{n+1}}}n   f''(\eta)^{n-1}f'''(\eta) \left(\frac{U_{\infty} }{kx}\right)^{\frac{1}{n+1}}, 
$$
that, by simplifying,  gives the so-called Blasius equation for non-Newtonian fluids
\begin{equation} \label{eblasius}
f''(\eta)^{n-2} f'''(\eta) +\overline{c} f(\eta)=0,
\end{equation}
where
\begin{equation*} \label{cblasius}
\overline{c}:= \frac{\rho U_\infty^{\frac{3(1-n)}{n+1}}}{n(n+1)}.
\end{equation*}
Equation \eqref{eblasius} is well posed equipped with the boundary conditions
$$
f(0)=0,\ \ \ f'(0)=0,\ \ \ f'(\infty)=U_\infty, 
$$
that physically  represent, resp.,  no flow at the surface,  no slip condition, and  that one has  free stream velocity far from the plate.

\medskip
Assuming that a solution $f$ of \eqref{eblasius} is convex then  the idea, that dates back to L. Crocco (\cite{crocco}),  is to consider the non-dimensional stream velocity as a new variable, i.e. to consider 
\begin{equation*} \label{croc}
s=f'(\eta)
\end{equation*}
and to look for an equation solved by the shear stress $g(s)$ defined by 
\begin{equation*} \label{croc2}
g(s)= f''(\eta)^{n}
\end{equation*}
One then differentiates $g(f')= (f'')^{n}$ by obtaining
$$
g'(f') f'' = n(f'')^{n-1}f''', 
$$
that, by \eqref{eblasius}, gives 
$$
g'(f')  = - n\overline{c}f(\eta). 
$$
We differentiate again and we get 
$$
g''(f')f''  = - n\overline{c} f'(\eta), 
$$
that, in the new coordinates, gives the Crocco type equation 
$$
g''(s)g(s)^{\frac{1}{n}}=-n\overline{c}s
$$
or, in other words, the singular elliptic problem 
$$
-g''(s)=\frac{n\overline{c}s}{g(s)^{\frac{1}{n}}}
$$
altogether with the boundary conditions that translates into 
$$
g(U_\infty)=0,\ \ \ \text{and}\ \ \ g'(0)=0. 
$$

\section{Classical Solutions}
\label{sec:classical}

 In this section our aim is to discuss existence, uniqueness and regularity of the solution to \eqref{pbintro} in presence of a regular source term $f$.  
Here $\Omega$ has regular boundary, $\gamma > 0$ and $f>0$ belongs to $C^{\eta}(\overline{\Omega})$ for some $0<\eta<1$. 
In particular, for this section, we suppose that for any $x\in \Omega$ 
\begin{equation}\label{condf}
\min_{x\in\overline\Omega}f(x)=:m \le f(x) \le M:=\max_{x\in\overline\Omega}f(x).
\end{equation}
We deal with \underline{classical solutions}, i.e. positive functions $u\in C^2(\Omega)\cap C(\overline{\Omega})$ satisfying 
\begin{equation}
\begin{cases}
\displaystyle - \Delta u= \frac{f(x)}{u^\gamma} &  \text{in} \ \, \Omega, \\
u=0 & \text{on}\ \partial \Omega.
\label{pb}
\end{cases}
\end{equation}
Let us underline that, since the right-hand of the equation in \eqref{pb} is not   continuous at the boundary, in general one cannot expect the solution to be of class $C^2(\overline{\Omega})$.
Among the first  pioneering  works  dealing with \eqref{pb}  one has to mention \cite{crt}. Here the authors mainly treat the case where the dependence on $x$ is $C^1$ and the nonlinear term can be slightly more general; moreover, when the source is less regular they first provide a notion of so-called  generalized solutions of  \eqref{pb}. 
Later, in \cite{lm}, the authors give a renewal attention to this kind of problems, providing also a simplification of the proof of the existence and uniqueness presented in \cite{crt}.  The most interesting result of \cite{lm} concerns both weak and classical regularity of the solution depending on the behavior of the nonlinearity (say on $\gamma$) (see Section \ref{regLM} below). Some of these results were lately extended in \cite{ghl}. 

Finally we also recall the work \cite{stuart} in which the author proves existence along with the uniqueness of the solution to the same problem in a slightly more general setting.

\subsection{Existence and uniqueness of classical solutions}
\label{sec:lm}

We start  by showing  existence of a solution to \eqref{pb}; as we already mentioned, the result is already contained in \cite{crt} and in \cite{stuart} but we prefer to present the simplest proof of \cite{lm}. Other than simplicity, the proof contained in Theorem $1.1$ of \cite{lm} has implications on the regularity of the solution, which is the main novelty of the paper.
\begin{theorem}
	Let $0<\eta<1$ and let  $0<f \in C^\eta(\overline{\Omega})$ satisfy \eqref{condf}.  Then there exists a unique solution $u\in C^{2,\eta}(\Omega)\cap C(\overline{\Omega})$ to problem \eqref{pb}.  In particular there exists $a,b>0$ such that
	\begin{equation}\label{controllo}
		a \varphi_1 \le u(x) \le b \varphi_1^t \ \ \ \forall x \in \Omega,
	\end{equation}	 
	where $t= \frac{2}{1+\gamma}$ if $\gamma>1$ and $t\in (0,1)$ if $\gamma\le 1$. 
	\label{teoexlm}
\end{theorem}
\begin{proof}
We look for a solution to problem \eqref{pb} via the following scheme of approximation
	\begin{equation}
	\begin{cases}
	\displaystyle - \Delta u_n= \frac{f}{(u_n+\frac{1}{n})^\gamma} &  \text{in}\ \Omega, \\
	u_n=0 & \text{on}\ \partial \Omega;
	\label{pbapprox}
	\end{cases}
	\end{equation}

	\medskip
	
the first move consists  in showing the existence of the solution $u_n$ through the method of sub and super-solutions (see for instance \cite{sat}).
	The sub and the super-solutions found are of the type $c\varphi_1^t$, $c,t>0$ ($\varphi_1$ defined in \eqref{not:phi1} where $A(x)=I$); then we will pass to the limit \eqref{pbapprox} as $n\to\infty$ in order to conclude. We split the proof into a few steps. 
	
	\medskip
		
		{\bf Step 1. Sub-solution.} The sub-solution is $z=a\varphi_1$ for $a>0$ sufficiently small. Indeed it is sufficient to require
	$$ a<\frac{\displaystyle m}{\lambda_1||\varphi_1+1||^{\gamma+1}_{L^\infty(\Omega)}}$$
	in order to deduce that
	$$-\Delta (a \varphi_1) = \lambda_1 a \varphi_1<\frac{f}{(\varphi_1+\frac{1}{n})^\gamma}.$$
	
	\medskip
	
	{\bf Step 2. Super-solution.} For a reason which will be clear in a moment, we need to distinguish between cases $\gamma\le1$ and $\gamma>1$. 
	
	\medskip
	
	If $\gamma>1$ the super-solution is $w=b\varphi_1^t$, for some $b,t>0$.
	We need
	\begin{equation}\label{lmsuper1}
	-\Delta (b\varphi_1^t)\ge \frac{f}{(b\varphi_1^t+ \frac{1}{n})^\gamma},
	\end{equation}
	that is implied by 
$$
\frac{1}{b^\gamma\varphi_1^{t\gamma}}(b^{1+\gamma}t(1-t)|\nabla \varphi_1|^2\varphi_1^{t(1+\gamma)-2}  + b^{1+\gamma}t\lambda_1\varphi_1^{t(1+\gamma)})\ge \frac{f}{b^\gamma \varphi_1^{t\gamma}}, 
$$
	that is 
		\begin{equation}\label{lmsuper}
b^{1+\gamma}(t(1-t)|\nabla \varphi_1|^2 +t\lambda_1\varphi_1^{2})\ge {f}, 
	\end{equation}
	provided we fixed $t:=\frac{2}{1+\gamma}<1$. 
	In order to check \eqref{lmsuper} we reason as follows: by the Hopf Lemma (see Lemma \ref{hopf}) we can pick $\varepsilon$ small enough such  that $|\nabla \varphi_1| \not= 0$ on $\Omega_\varepsilon$ ($\Omega_\varepsilon$ is defined in \eqref{not:omegaeps}).
	If  $x\in \Omega_\varepsilon$ then \eqref{lmsuper} holds once one requires 
	\begin{equation}\label{lm1}
	b\ge \left(\frac{\displaystyle\max_{x\in\overline{\Omega}_\varepsilon} f(x)}{t(1-t)\displaystyle \min_{x\in \overline{\Omega}_\varepsilon}|\nabla\varphi_1|^2}\right)^\frac{1}{\gamma+1}.
	\end{equation}
	While if $x\in \Omega\setminus\Omega_\varepsilon$ we ask $b$ to satisfy
	\begin{equation}\label{lm2}
	b\ge \left(\frac{\displaystyle\max_{x\in\Omega\setminus\Omega_\varepsilon} f(x)}{t\lambda_1\displaystyle \min_{x\in\Omega\setminus\Omega_\varepsilon}\varphi_1^{2}}\right)^{\frac{1}{\gamma+1}}.
	\end{equation}
	Choosing $b$ as   the maximum among the quantities in the right-hand of  \eqref{lm1} and \eqref{lm2} then \eqref{lmsuper} holds, and so \eqref{lmsuper1}.\\ 
	If $0<\gamma\le 1$ we reason similarly but we need to look for a super-solution $w= b\varphi_
	1^s$ where $0<s<1$; this guarantees that both terms in \eqref{lmsuper} are nonnegative.
	Now observe that, since $s<1$, then $s<\frac{2}{1+\gamma}$ and this allows us to reason as before as the term  $\varphi_1^{s(1+\gamma)-2}$  can be easily be bounded from below on $\Omega_\varepsilon$.

	\medskip 
	
	{\bf Step 3. Ordering $z$ and $w$.} In  order to apply the method we also  need
	\begin{equation*}
	z(x) \le w(x), \ \ \ \forall x \in \Omega, 
	\end{equation*}
	that is true if $a$ also satisfies
$$
	a \le \frac{b}{||\varphi_1||^{1-t}_{L^\infty(\Omega)}},
$$
if $\gamma>1$; obviously, if $0<\gamma\leq 1$,  we require 
	$$a \le \frac{b}{||\varphi_1||^{1-s}_{L^\infty(\Omega)}}.$$

{\bf Step 4. Sub and Super-solution applied and monotonicity of $u_n$.} We are now in the position to  apply the method of sub and super-solutions to conclude the existence of a smooth  solution $u_n$ of \eqref{pbapprox} such that for any $n\in \mathbb{N}$ 
	\begin{equation}\label{orde} z(x) \le u_n(x) \le w(x), \ \ \ \forall x \in \Omega.\end{equation}
	
	We show that the sequence $u_n$ is non-decreasing with respect to $n$, proving it by contradiction. Indeed let us suppose that $E=\{x\in \Omega: u_n(x)-u_{n+1}(x)>0\}$ is not empty. 
	Then it is easy to check that for $x\in E$ 
	\begin{equation*}
	-\Delta(u_n(x)- u_{{n+1}}(x))\le 0,
	\label{contrlm}
	\end{equation*}
	and, since
	$$u_n(x)=u_{n+1}(x), \ \ \ \forall x \in \partial E,$$
	then by the maximum principle we have a contradiction; this shows that $E$ is empty.
	
	\medskip
	
	{\bf Step 5. Strong convergence in $C^{2,\eta}(\Omega)$ and proof completed.} Since 
	$$z(x)\le u_n(x) \le u_{n+1}(x) \le w(x), \ \ \forall x \in \overline{\Omega}, \forall n \in \mathbb{N},$$
	then there exists a bounded function $u$ such that $\displaystyle \lim_{n\to \infty} u_n(x)=u(x)$ for all $x \in \overline{\Omega}$. Moreover \eqref{controllo} holds. 
	We need to prove that $u$ satisfies \eqref{pb}. To show this fact, let us prove that $u_n$ converges (at least locally) to $u$ in $C^{2,\eta}(\Omega)$. This part of the proof is quite standard and we  only give a brief sketch.
	Let $x_0 \in \Omega$ and let $r>0$ such that $\overline{B_r(x_0)}\subset \Omega$. Let $\psi\in C^\infty(\Omega)$ that is equal to $1$ on $\overline{B_\frac{r}{2}(x_0)}$ and equal to $0$ outside $B_r(x_0)$. We clearly have
	\begin{equation}
	\Delta (\psi u_n) = 2\nabla \psi \cdot \nabla u_n + \psi \Delta u_n + u_n \Delta \psi,
	\label{pblocal}
	\end{equation}
	where the second and the third term on the right-hand of \eqref{pblocal} are bounded uniformly with respect to $n$. Indeed, using \eqref{orde},    one has  resp.,  that  $-\Delta u_n \le M (a \varphi_1)^{-\gamma}$ and  $u_n \le w$.
	
	Then, multiplying by $\psi u_n$ equation \eqref{pblocal} and integrating on $B_r(x_0)$ one gets that 
	$$\psi  |\nabla u_n| \text{  is bounded in  } L^2(B_r(x_0)),$$
which gives
	\begin{equation}\label{3.2} |\nabla u_n| \text{  is bounded in  } L^2(B_\frac{r}{2}(x_0)).\end{equation}
	Now we proceed by iteration taking $\psi_1\in C^\infty(\Omega)$ that is equal to $1$ on $\overline{B_\frac{r}{4}(x_0)}$ and equal to $0$ outside $B_\frac{r}{2}(x_0)$. We renew    \eqref{pblocal} for $\psi_1 u_n$ and, in view of \eqref{3.2},  we apply   standard elliptic regularity to get that  $\psi_1 u_n$ is bounded in $W^{2,2}(B_\frac{r}{2}(x_0))$ so that $u_n$ is bounded in $W^{2,2}(B_\frac{r}{4}(x_0))$. Moreover by Sobolev embedding we have that $|\nabla u_n|$ is bounded in $L^{q_1}(B_\frac{r}{4}(x_0))$ with $q_1=2^*$.
	We keep reasoning as above, we take $\psi_2$ that is equal to $1$ on $\overline{B_\frac{r}{8}(x_0)}$ and equal to $0$ outside $B_\frac{r}{4}(x_0)$; this  allows us to deduce  that $u_n$ is bounded in $L^{q_2}(B_\frac{r}{8}(x_0))$ with $q_2=q_1^*$. With a finite number of steps $i$ this bootstrap argument  shows  that $u_n$ is bounded in $W^{2,q_i}(B_\frac{r}{4i}(x_0))$ such that $q_i>\frac{N}{1-\eta}$.
	Thus the Morrey Theorem guarantees the existence of a subsequence of $u_n$ that converges in $C^{1,\eta}(\overline{B_\frac{r}{2i}(x_0)})$.
	From equation \eqref{pblocal} where $\psi$ is equal to $1$ in $\overline{B_\frac{r}{4i}(x_0)}$ and $0$ outside $B_\frac{r}{4i}(x_0)$ it follows from Schauder theory that a subsequence of $u_n$ converges in $C^{2,\eta}(\overline{B_\frac{r}{4i}(x_0)})$.
	It follows from the arbitrariness of $x_0$ that $u$ belongs to $C^{2,\eta}(\Omega)$, which also implies that  it solves problem \eqref{pb}. Moreover,  $u \in C(\overline{\Omega})$ as a consequence of  \eqref{controllo}. 
	Uniqueness follows by the  maximum principle similarly to what we did in the previous step in order to get monotonicity. 
\end{proof}

\begin{remark}\label{remexlm}
Let us stress that, thanks to \eqref{controllo} (see also  \eqref{hopfvarphi2}),   the solution found in Theorem \ref{teoexlm} satisfies 
\begin{equation}\label{gM1}
c_1 d (x)\le u(x)\le c_2 d (x)^{\frac{2}{\gamma+1}},\ \ \text{if} \ \ \gamma\geq 1
\end{equation}
and, for any $s\in(0,1)$
\begin{equation}\label{gm1}
c_1 d (x)\le u(x)\le c_2 d (x)^{s},\ \,\ \ \text{if}\ \  \ \gamma <1. 
\end{equation}

Let us point out  that the previous  estimates can be improved, with a refinement of the above arguments,  by showing the existence of positive constants $c_1$, $c_2$, and $C$ such that 
\begin{equation}\label{gM11}
c_1 d (x)^{\frac{2}{\gamma+1}}\le u(x)\le c_2 d (x)^{\frac{2}{\gamma+1}},\ \ \text{if} \ \ \gamma> 1,
\end{equation}
\begin{equation}\label{gm1i}
c_1 d (x)(C-\ln d (x))^{\frac{1}{2}}\le u(x)\le c_2 d (x)(C-\ln d (x))^{\frac{1}{2}},\ \,\ \ \text{if}\ \  \ \gamma =1, 
\end{equation}
and 
\begin{equation}\label{gm1i2}
c_1 d (x)\le u(x)\le c_2 d (x),\ \,\ \ \text{if}\ \  \ \gamma <1, 
\end{equation} 
as it has been proven in  \cite[Theorem 2.1]{ghl}. 

The aforementioned  estimates give some further insight on the boundary regularity of the solutions of  \eqref{pb} that will be detailed in next section. Let us also explicitly observe that, once that the previous estimates work, the singular right-hand $fu^{-\gamma}$ is not integrable if $\gamma\ge 1$ while it always lies in $L^1(\Omega)$ if $0<\gamma<1$.
\triang \end{remark}

\subsection{Regularity of the solution}\label{regLM}
Here we analyze  the regularity of the solution of problem  \eqref{pb}. It is interesting to point out the following:  by the proofs of the next two global regularity results will be clear how  the dependence on the space dimension $N$ plays essentially  no roles. Indeed, as the problem is, roughly speaking, smooth inside, bad things can only happen, depending on the parameter $\gamma$,  at the boundary of the domain in which only one-dimensional like estimates will be involved. 

\medskip  
Let us  state and prove  the next  regularity results which are due to  \cite{lm} and \cite{ghl}:
\begin{theorem}\label{teogui}
	Let $f$ be an H\"older continuous function satisfying \eqref{condf}. Then the solution $u$ of problem \eqref{pb} satisfies:
	\begin{itemize}
		\item[i)] if $\gamma<1$ then $u\in C^{1,1-\gamma}(\overline{\Omega})$;

		\item[ii)] if $\gamma=1$ then $u\in C^{t}(\overline{\Omega})$ for all $0<t<1$; 
		\item[iii)] 	if $\gamma>1$ then $u\in C^{\frac{2}{\gamma+1}}(\overline{\Omega})$ and $u\not\in C^1(\overline{\Omega})$.  		
	\end{itemize}
\end{theorem}

\begin{proof}[\textbf{ \rm \bf Proof of i)}]

We aim to prove that there exists $c>0$ such that for any $x_1,x_2 \in \Omega$ one has 
\begin{equation}\label{gulin1}
|\nabla u(x_1)- \nabla u(x_2)| \le c |x_1- x_2|^{1-\gamma}.
\end{equation}
By classical Green's formula one has both
	$$u(x)=\int_\Omega G(x,y)f(y)u^{-\gamma}(y)dy,$$
and
	$$\nabla u(x)=\int_\Omega \nabla_x G(x,y)f(y)u^{-\gamma}(y)dy,$$
	where $G$ is the Green function relative to $\Omega$.  Observe that, since $\gamma<1$ the previous are well defined.   Recall, see for instance \cite[Proposition 3.1]{ghl},  that it is always possible to find a constant $c>0$ (that only depends on $\Omega$) such that, for $x\neq y$, one has 
	\begin{equation}\label{greene}
	|\nabla_x G(x,y)|\leq \min\left\{\frac{c}{|x-y|^{N-1}},\frac{c d(y)}{|x-y|^N}\right\}\ \ \ \ \ \  \text{and}\ \ \ \ |D_x^2 G(x,y)|\leq \min\left\{\frac{c}{|x-y|^{N}},\frac{c d (y)}{|x-y|^{N+1}}\right\}.
	\end{equation}

We pick $x_1,x_2 \in \Omega$ such that $|x_1-x_2|<\varepsilon$ for some $\varepsilon>0$. One has  
	\begin{equation}\label{gui3}
	\begin{aligned}
	|\nabla u(x_1)- \nabla u(x_2)| &\le \int_{\Omega\setminus B_R(x_1)} |\nabla_x G (x_1,y)-\nabla_x G(x_2,y)|f(y)u^{-\gamma}(y)dy 
	\\
	&+ \int_{B_R(x_1)}|\nabla_x G(x_1,y)-\nabla_x G(x_2,y)|f(y)u^{-\gamma}(y)dy=I+II,
	\end{aligned}
	\end{equation}
	where $R=2|x_1-x_2|$. Now consider the path  $\xi:[0,1]\to\Omega$ given by $\xi (t)=x_1+t(x_2-x_1)$, with $t\in [0,1]$,  that connects $x_1$ and $x_2$ (assume without loosing generality that $\varepsilon$ is small enough). Observe that, if  $y\in \Omega\setminus B_R(x_1)$, one has both
	\begin{equation}\label{disu}
	|x_1-y|\leq |x_1-\xi(t)|+|\xi(t)-y|\leq |x_1-x_2| + |\xi(t)-y| \leq 2|\xi(t)-y|\,,
	\end{equation}
	and
	\begin{equation}\label{disu2}
	 |\xi(t)-y|\leq |\xi(t)-x_1| + |x_1 -y| \leq 2|x_1-y|\,,
	\end{equation}
	where we also  used that, by definition of $R$, it holds $|x_1-x_2|,  |\xi(t)-x_1|\leq |\xi(t)-y|$. 
	So that, thanks to \eqref{disu} and \eqref{disu2},  one has 
	 	\begin{equation}\label{disuglob}
	\underline{c}|x_1-y| \leq |\xi(t)-y|\leq \overline{c}|x_1-y|\,, 
	\end{equation}
	for some positive constants any $\underline{c},  \overline{c}$ and for any $t\in (0,1)$.
	
	\begin{figure}[htbp]\centering
\includegraphics[width=3in]{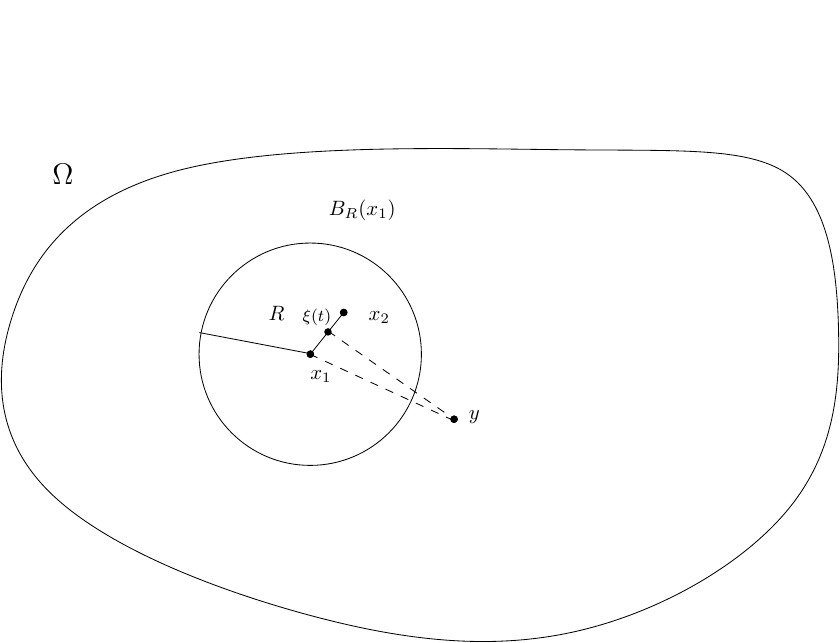}
\caption{Visualizing \eqref{disuglob}}\label{faster}
\end{figure}

	Now,  one  has
	\begin{equation*}
	\begin{aligned}
	&I \le \int_{\Omega\setminus B_R(x_1)} \left(\int_{0}^1 |D_{x}^{2} G(\xi(t),y)||\xi'(t)|dt \right)f(y)u^{-\gamma}(y)dy
	\\
	& \stackrel{\eqref{greene}}{\le} C|x_1-x_2|\int_{\Omega\setminus B_R(x_1)} \left(\int_{0}^1 \frac{\min(c|\xi(t)-y|,c d (y))}{|\xi(t)-y|^{N+1}}dt \right)f(y)u^{-\gamma}(y)dy  
	\\
	&\stackrel{\eqref{disuglob}}{\le} C|x_1-x_2|\int_{\Omega\setminus B_R(x_1)} \frac{\min(|x_1-y|, d (y))}{|x_1-y|^{N+1}} f(y)u^{-\gamma}(y)dy\,.
	\end{aligned}
	\end{equation*}

	\medskip 
	
	Then, it follows from \eqref{condf} and \eqref{gm1}, that
	\begin{equation*}
	\begin{aligned}
	&\int_{\Omega\setminus B_R(x_1)} |\nabla_x G(x_1,y)-\nabla_x G(x_2,y)|f(y)u^{-\gamma}(y)dy \le C|x_1-x_2|\int_{\Omega\setminus B_R(x_1)} \frac{\min(|x_1-y|, d (y))}{|x_1-y|^{N+1}}  d^{-\gamma}(y)dy
	\\
	&\le C|x_1-x_2|\int_{\Omega\setminus B_R(x_1)} \frac{1}{|x_1-y|^{N+\gamma}}dy\le  C|x_1-x_2| \int_R^\infty r^{-1-\gamma}dr \le C|x_1-x_2|R^{-\gamma}=C|x_1-x_2|^{1-\gamma}.
	\end{aligned}
	\end{equation*}
	For the second term of \eqref{gui3},  using that   $u(y)\ge {c_1} d (y)$ (recall \eqref{gm1}),  one obtains 
	\begin{equation*}
	\begin{aligned}
	&II=\int_{B_R(x_1)}|\nabla_x G(x_1,y)-\nabla_x G(x_2,y)|f(y)u^{-\gamma}(y)dy \le \int_{B_R(x_1)}| \nabla_x G(x_1,y)|f(y)u^{-\gamma}(y)dy  \\& +\int_{B_{R+ \frac{R}{2}}(x_2)}|\nabla_x G(x_2,y)|f(y)u^{-\gamma}(y)dy
	\stackrel{\eqref{greene}}{\le} C \int_{B_R(x_1)}\frac{\min(|x_1-y|, d (y))}{|x_1-y|^{N}} d ^{-\gamma}(y)dy \\ &+C \int_{B_{R+ \frac{R}{2}}(x_2)}\frac{\min(|x_2-y|, d (y))}{|x_2-y|^{N}} d^{-\gamma}(y)dy  \le C \int_{B_R(x_1)}\frac{dy}{|x_1-y|^{N-1+\gamma}}  +C \int_{B_{R+ \frac{R}{2}}(x_2)}\frac{dy}{|x_2-y|^{N-1+\gamma}}. 
	\end{aligned}	
	\end{equation*}
	 
	 Hence, observing  that $|x_1-y|, |x_2-y|\geq |x_1-x_2|=\frac{R}{2}$, this yields to 
	$$\int_{B_R(x_1)}|\nabla_x G(x_1,y)- \nabla_x G(x_2,y)|f(y)u^{-\gamma}(y) \le C |x_1-x_2|^{1-\gamma},$$
which, together with \eqref{gui3},  implies \eqref{gulin1},  and this concludes the proof of i). 
	\medskip 
	
	\textbf{Proof of ii).}
	Assume first $x_1, x_2\in \Omega$ such that $|x_1-x_2|\le \frac{d}{4}$ where $d:=d(x_1) \ge d(x_2)$.		
	
	Then it follows from classical elliptic regularity estimates (see for instance  \cite[Proposition $2.18$]{fernandezros} or \cite[Theorem 6.2]{gt})  that it holds
	\begin{equation}\label{stimagradiente}
		\frac{d}{2}|\nabla u(y)| \le c(||u||_{L^\infty(B_{\frac{d}{2}}(x_1))} + \frac{d^2}{4}||fu^{-1}||_{L^\infty(B_{\frac{d}{2}}(x_1))},\end{equation}
		for any $y\in B_{\frac{d}{4}}(x_1)$ and for some positive constant $c$.  
	Hence, using \eqref{gm1i} in \eqref{stimagradiente}, one obtains
		\begin{equation*}
		|\nabla u(y)| \le c\left(\left(C-\ln\left(\frac{d}{2}\right)\right)^{\frac{1}{2}} + \max_{x\in B_{\frac{d}{2}}(x_1)}f(x) \left(C-\ln\left(\frac{d}{2}\right)\right)^{-\frac{1}{2}}\right),
	\end{equation*}	
	for any $y\in B_{\frac{d}{4}}(x_1)$.
	\\
	Now as we assumed $|x_1-x_2|\le \frac{d}{4}$, it follows, by possibly changing the value of $c$,  that 
	\begin{equation}
		\label{stimagradiente2}
		|u(x_2)-u(x_1)|\le c \left(C-\ln\left(|x_1-x_2|\right)\right)^{\frac{1}{2}} |x_1-x_2|.
	\end{equation}
If $|x_1-x_2| > \frac{d}{4}$ instead, then, recalling that $d(x_1) \ge d(x_2)$,   the validity of \eqref{stimagradiente2}  follows from \eqref{gm1i}.  This concludes the proof of ii) as \eqref{stimagradiente2} implies the desired regularity.

	\medskip 

	\textbf{Proof of iii).} The proof of the H\"older regularity strictly follows the one of ii) with straightforward modifications and we omit it.

In order to conclude we want to show that $u\not\in C^1(\overline{\Omega})$. 
	To this aim, let $x_0\in \partial \Omega$ and let us denote by $\vec{n}$ the inner normal to $\partial\Omega$ at $x_0$.
	As $\gamma>1$ we can use \eqref{gM11}; we have    $u(x)\ge c_1 d(x)^\frac{2}{\gamma+1}$. 
	Thus
	$$\frac{u(x_0+t\vec{n})-u(x_0)}{t}\ge c_1 d(x_0+t\vec{n})^{\frac{1-\gamma}{1+\gamma}}\frac{d(x_0+t\vec{n})}{t}.$$
	Since
	$$\dis \lim_{t\to 0^+}\frac{d(x_0+t\vec{n})}{t} = \lim_{t\to 0^+}\frac{d(x_0+t\vec{n})-d(x_0)}{t}= \nabla d(x_0) \cdot \vec{n}>0$$ then
	$$\displaystyle \lim_{t\to 0^+}\frac{u(x_0+t\vec{n})-u(x_0)}{t}=\infty,$$
	that is what we need to prove.
	This concludes the proof of Theorem \ref{teogui}.	
\end{proof}

\medskip 
\begin{remark}
Observe that, if  $\gamma\to 3^-$ then ${\frac{2}{\gamma +1}}$ tends to $\frac12$ from above;  this fact is completely reasonable and it fits with the remarkable result we will prove next (Theorem \ref{teoreg} below).  As shown in \cite[Theorem 1.2]{ghl}, for  merely nonnegative data (i.e. $f\geq 0$) i) and ii) in Theorem \ref{teogui} remain true. On the contrary, for technical reasons,  iii) transforms into $u\in C^{\frac{2}{\gamma^2 +1}}(\overline{\Omega})$ that, for what we saw before,  seems to be not sharp. Another interesting open problem consists in showing whether, for $\gamma=1$, $u$ belongs to $C^1$ up to the boundary of $\Omega$.  
\triang \end{remark}
\medskip 
The following (at first glance) surprising  result shows that the solution to \eqref{pb} ceases to be even in $H^1_0(\Omega)$ when $\gamma\ge 3$.  
Since  uniqueness holds then, without loosing generality, we will use that the solution is constructed through the scheme of approximation \eqref{pbapprox}.

\begin{theorem}
	The solution $u$ to problem \eqref{pb} found  in Theorem \ref{teoexlm} belongs to $H^1_0(\Omega)$ if and only if $\gamma<3$.
	\label{teoreg}
\end{theorem}  
\begin{proof}
	First let $\gamma<3$ and let us show that the limit  $u$  of $u_n$, solution to \eqref{pbapprox}, belongs to $H^1_0(\Omega)$; it suffices to check that  
	$$\displaystyle \int_{\Omega}|\nabla u_n|^2\leq C,$$
	for some constant $C$ not depending on $n$.
	
	 \noindent We  take $u_n$ as  test function in \eqref{pbapprox} obtaining
	\begin{equation*}
	\displaystyle \displaystyle \int_{\Omega}|\nabla u_n|^2 = \displaystyle \int_{\Omega}f u_n \left(u_n+\frac{1}{n}\right)^{-\gamma}\le \int_{\Omega}M u_n^{1-\gamma}.
	\label{stimalm1}
	\end{equation*}
	If $\gamma=1$ the conclusion holds. Assume that $0<\gamma< 1$ and use the H\"older and the Sobolev inequalities to obtain  	\begin{equation*}
	\displaystyle \displaystyle \int_{\Omega}|\nabla u_n|^2 \le M\left (\int_\Omega u_n^{2^*}\right)^{\frac{1-\gamma}{2^*}}|\Omega|^{\frac{2^* +\gamma-1}{2^*}}\leq M|\Omega|^{\frac{2^* +\gamma-1}{2^*}} \mathcal{S}_{2}^{1-\gamma} \left (\int_\Omega |\nabla u_n|^2 \right)^{\frac{1-\gamma}{2}},
	\label{stimalm2}
	\end{equation*}
	which implies that $u\in H^1_0(\Omega).$

	 Now if $1<\gamma<3$ one has $u_n \ge a\varphi_1^t$ with $t=\frac{2}{\gamma+1}$. In fact, as shown in \cite[Theorem 2]{lm}, an  estimate as in  \eqref{gM11} holds for the approximating sequence $u_n$. Thus,   by taking $u_n$ to test \eqref{pbapprox}, one obtains 
\begin{equation}
\displaystyle \displaystyle \int_{\Omega}|\nabla u_n|^2 \le \int_{\Omega}M a^{1-\gamma}\varphi_1^{t(1-\gamma)},
\label{stimalm3}
\end{equation} 
whose right-hand of \eqref{stimalm3} is finite since $\gamma<3$.
Now we need to prove that if $\gamma\ge 3$ then $u$ is not in $H^1_0(\Omega)$; we show it by contradiction.
First let observe that, using \eqref{gM11}, one gets 
\begin{equation}
\displaystyle \int_{\Omega} \frac{f}{u^{\gamma-1}}\ge \int_{\Omega} \frac{m}{c_2^{\gamma-1}d^{t(\gamma-1)}}=\infty.
\label{stimalm4}
\end{equation} 
If $u$ belongs to $H^1_0(\Omega)$ then there exists a nonnegative sequence $w_n\in C^1_c(\Omega)$ converging to $u$ in $H^1_0(\Omega)$ as $n\to\infty$.
It follows from \eqref{stimalm4} and from the Fatou Lemma that 
$$ \displaystyle \liminf_{n\to\infty}\int_{\Omega}w_nfu^{-\gamma}\ge \int_{\Omega}\frac{f}{u^{\gamma-1}}=\infty.$$
On the other hand one  can take $w_n$ as a test function in \eqref{pb} obtaining
$$\displaystyle \int_{\Omega}\nabla u\cdot \nabla w_n= \int_{\Omega} w_n f u^{-\gamma},$$
and taking the limit in $n$, one gets
$$\displaystyle\int_{\Omega}|\nabla u|^2=\infty,$$
which contradicts that $u$ belongs to $H^1_0(\Omega)$. This concludes the proof.	

\end{proof}

We summarize the regularity properties of $u$ in the following table:
\begin{center}
	\begin{table}[H]
		\setlength{\tabcolsep}{8pt}
		\renewcommand{\arraystretch}{2.4} 
		\begin{tabular}{| c | c | c |}  
			\cline{1-3}
			$\gamma$ & Strong regularity  & Weak regularity
			\\
			\hline
			$\gamma<1$
			&
			$u\in C^{1,1-\gamma}(\overline{\Omega})$ & 
			\\ 			\cline{1-2}	
			$\gamma=1$
			&
			${u\in C^{t}(\overline{\Omega})}$ for any $0<t<1$ & ${u\in H^1_0(\Omega)}$
			\\  			\cline{1-2}	
			$1<\gamma<3$
			&
			${u\in C^{\frac{2}{\gamma+1}}(\overline{\Omega})}$, ${u\notin C^1(\overline{\Omega})}$, & 
			\\ 			\cline{1-3}	
			$\gamma\ge3$
			&
			$u\in C^{\frac{2}{\gamma+1}}(\overline{\Omega})$, $u\notin C^1(\overline{\Omega})$ & ${u\notin H^1_0(\Omega)}$		
			\\
			\cline{1-3}	
		\end{tabular}
		\vspace*{3mm}
		\caption{Strong and weak regularity \\ for  smooth  $f$'s depending on $\gamma$} 
	\end{table}
\end{center}

\begin{figure}[htbp]\centering
\includegraphics[width=3in]{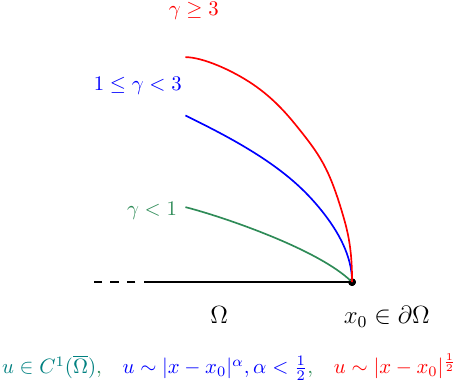}
\caption{Behaviour of the solution $u$ around some boundary point as $\gamma$ grows}\label{faster}
\end{figure}

\section{Integrable data and distributional solution}
\label{sec:integrable}

Natural questions arise about existence of solutions to problems as \eqref{pbintro} when the source term $f$ is less regular.  
This section is devoted to present the  weaker framework  that covers this case  in which, in general,  a  classical solution is not expected to exist. 

\medskip

We first prove existence of \underline{distributional solutions} for semilinear elliptic problems when $f$ belongs to a suitable Lebesgue space. Then the results will be extended to the case  of a general and not necessarily  monotone nonlinearity $h(s)$ in place of the model $s^{-\gamma}$. In the last part of the section we show that uniqueness of distributional solutions belonging to $H^1_0(\Omega)$ holds.

\medskip

Concerning  the model case of a power nonlinearity,  we will mainly present the strategy given by \cite{bo} that   takes advantage   of the strong maximum principle in order to deduce that the sequence of non-decreasing approximating solutions are far away from zero inside the domain. For generic  non-monotone nonlinearities, this technique does not apply in general and we will make use of a suitable  truncation approach,  as  in \cite{gmm},  in order to control the singularity.  Uniqueness of finite energy solutions, as in  \cite{boca},   is obtained through an extension  of the admissible test functions.  

\medskip

We deal with
	\begin{equation}
\begin{cases}
\displaystyle - \operatorname{div}(A(x)\nabla u)= \frac{f}{u^\gamma} &  \text{in }\, \Omega, \\
u=0 & \text{on}\ \partial \Omega,
\label{pbbo}
\end{cases}
\end{equation}
where $A$ is a bounded elliptic matrix, namely
\begin{equation}\label{Abounded}
	\exists \alpha,\beta>0, \ A(x)\xi\cdot\xi \ge \alpha|\xi|^2, \ |A(x)|\le \beta,
\end{equation}
for almost every $x \in \Omega$ and for every $\xi\in \mathbb{R}^N$. The nonnegative (and not identically zero) datum $f$ belongs to a suitable Lebesgue space we specify later and, finally, $\gamma>0$. Let us also specify that in this section it is not assumed any regularity on $\partial\Omega$.

\medskip

First we set the concept of distributional solutions to \eqref{pbbo}. For the sake of exposition we fix the following notation:
\begin{equation}\label{sigma}
	\sigma:= \max(1,\gamma)\,.
\end{equation}

\begin{defin}\label{defbo}
	A nonnegative function $u\in W^{1,1}_{\rm{loc}}(\Omega)$ is a {\it distributional solution} to problem \eqref{pbbo} if	
	\begin{equation}
	fu^{-\gamma}\in L^1_{\rm{loc}}(\Omega),
	\label{lweakdef1}
	\end{equation}
	\begin{equation}
	u^{\frac{\sigma+1}{2}}\in W^{1,1}_0(\Omega), 
	\label{lweakdef2}
	\end{equation}
	and
	\begin{equation} 
	\displaystyle \int_{\Omega}A(x)\nabla u \cdot \nabla \varphi =\int_{\Omega} \frac{f\varphi}{u^\gamma},\ \ \ \forall \varphi \in C^1_c(\Omega).\label{lweakdef3}
	\end{equation}	
	\label{lweakdef}
\end{defin}
\begin{remark}
Some comments are in order. Let us observe that condition \eqref{lweakdef1} implies that the right-hand of the weak formulation \eqref{lweakdef3} is well defined, while \eqref{lweakdef2} represents a weak way to recover the Dirichlet datum. In case $\gamma\le1$,  the zero boundary datum is prescribed in the usual Sobolev trace sense; on the other hand, if $\gamma >1$,  this will not anymore be the case  and, in general,  only a suitable power of the solution is asked to attain the zero value in the usual sense of the Sobolev traces. 

We also underline that, from here on, we call \underline{finite energy solution} or \underline{weak solution} a distributional solution that belongs to $H^1_0(\Omega)$. As we will show later (see Lemma \ref{lemextest} in Section \ref{sec:uniquenessweak}) a weak solution takes to an equivalent formulation that reads as:
	$$
	\int_{\Omega} A(x)\nabla u\cdot\nabla v = \int_{\Omega}fu^{-\gamma}v,
$$	 
	for every $v\in H^1_0(\Omega)$. 
\triang \end{remark}

\subsection{Existence of distributional solutions for Lebesgue data}\label{exidis}

Here we  state and prove existence and Sobolev regularity of a distributional solution to problem \eqref{pbbo}.  
\begin{theorem}
	\label{boesistenza}
	Let $A$ satisfy \eqref{Abounded} and let $0\le f \in L^m(\Omega)$ with $m\ge1$. Then there exists a distributional solution $u$ of \eqref{pbbo} such that
	\begin{itemize}
		\item[i)] if $\gamma=1$ then $u \in H^1_0(\Omega)$;
		\item[ii)] if $\gamma<1$ and $m\ge \left(\frac{2^*}{1-\gamma}\right)'$ then $u \in H^1_0(\Omega)$, otherwise if $1 \le m <\left(\frac{2^*}{1-\gamma}\right)'$ then $u \in W^{1,\frac{Nm(\gamma+1)}{N-m(1-\gamma)}}_0(\Omega)$;
		\item[iii)] if $\gamma>1$ then $u \in H^1_{\rm{loc}}(\Omega)$.
	\end{itemize}
\end{theorem}
\begin{proof}
	\noindent We reason by approximation by considering
	\begin{equation}
	\begin{cases}
	\displaystyle - \operatorname{div}(A(x)\nabla u_n)= \frac{f_n}{(u_n+\frac{1}{n})^\gamma} &  \text{in }\, \Omega, \\
	u_n=0 & \text{on}\ \partial \Omega,
	\label{pbn}
	\end{cases}
	\end{equation}
	where $f_n:=T_n(f)$. It follows from standard theory (\cite{ll}) that there exists a nonnegative $u_n\in H^1_0(\Omega)\cap L^\infty(\Omega)$ such that
	\begin{equation}\label{boapprox}
		\int_\Omega A(x)\nabla u_n\cdot \nabla \varphi = \int_\Omega \frac{f_n\varphi}{(u_n+\frac{1}{n})^\gamma}, \ \forall \varphi \in H^1_0(\Omega).
	\end{equation}
	Moreover the sequence $u_n$ is non-decreasing with respect to $n$.
	\\Indeed, we can take $(u_n-u_{n+1})^+$ as a test function in the weak formulation of $\eqref{pbn}_{n}-\eqref{pbn}_{n+1}$ getting 
	\begin{align*}\label{monotoniaapprox}
	\displaystyle \alpha\int_{\Omega} |\nabla (u_n-u_{n+1})^+|^2 &\le \int_{\Omega} \left(\frac{f_n}{(u_n+\frac{1}{n})^\gamma} - \frac{f_{n+1}}{(u_{n+1}+\frac{1}{n+1})^\gamma}\right) (u_n-u_{n+1})^+  \\ \nonumber
	&\le \int_\Omega\left(\frac{f_{n+1}}{(u_n+\frac{1}{n+1})^\gamma} - \frac{f_{n+1}}{(u_{n+1}+\frac{1}{n+1})^\gamma}\right)(u_n-u_{n+1})^+\le 0,
	\end{align*}
	that implies $u_{n+1}\ge u_n$ a.e.  in $\Omega$.
	Moreover, since the right-hand of the equation solved by $u_1$ is not identically zero, it follows from the strong maximum principle that for all $n \in \mathbb{N}$  
	\begin{equation}\label{stimadb}\forall\;\omega\subset\subset\Omega\; \ \ \exists\; c_{\omega}>0: \ u_n\geq u_1\ge  c_{\omega}\;\text{ in }\;\omega.\end{equation}
	
	\noindent {\bf{Proof of i)}.} We take $u_n$ as a test function in the weak formulation of \eqref{pbn} obtaining
	$$\displaystyle \alpha\int_{\Omega} |\nabla u_n|^2 \le \int_{\Omega} \frac{f_n u_n}{u_n+\frac{1}{n}} \le \int_{\Omega} f,$$
	so that $u_n$ is bounded in $H^1_0(\Omega)$.
	 We call $u\in H^1_0(\Omega)$ the weak (and a.e.) limit, up to subsequences, as $n\to \infty$ of the sequence $u_n$. We first  observe that for any $\varphi\in C^{1}_{c}(\Omega)$	 
	 	\begin{equation}\label{4.4}\frac{f_n\varphi}{u_n+\frac{1}{n}}\le \frac{f\varphi}{c_\omega} \in L^1(\Omega)\,,\end{equation}
	so that \eqref{lweakdef1} holds by an application of the Fatou Lemma. Finally we can pass to the limit in  \eqref{boapprox}. In particular one can simply use the weak convergence of $u_n$ on the left-hand and, by \eqref{4.4},  the Lebesgue dominated convergence Theorem on the right one. 

	This concludes the proof   that $u\in H^1_0(\Omega)$ is a solution of \eqref{pbbo} in the case $\gamma=1$. 
	
	\medskip 
	
	\noindent {\bf{Proof of ii)}.}  We take $u_n$ as a test function in \eqref{pbn} and after an application of the H\"older inequality, one obtains
	\begin{equation}
	\displaystyle \alpha\int_{\Omega} |\nabla u_n|^2 \le \int_{\Omega} f_n u_n^{1-\gamma} \le ||f||_{L^{\left(\frac{2^*}{1-\gamma}\right)'}(\Omega)}\left( \int_{\Omega} u_n^{2^*}\right)^{\frac{1-\gamma}{2^*}},
	\label{reviewg<1}
	\end{equation}
	and by  Sobolev inequality 
	\begin{equation*}
	\displaystyle \frac{\alpha}{\mathcal{S}^2_2}\left(\int_{\Omega} u_n^{2^*}\right)^{\frac{2}{2^*}} \le ||f||_{L^{\left(\frac{2^*}{1-\gamma}\right)'}(\Omega)}\left( \int_{\Omega} u_n^{2^*}\right)^{\frac{1-\gamma}{2^*}}.
	\end{equation*}		
	Thus, since ${2}>{1-\gamma}$, we have
	\begin{equation}
	\displaystyle \|u_n\|_{L^{2^*}(\Omega)}\le C.
	\label{reviewg<1stima}
	\end{equation}
	We plug \eqref{reviewg<1stima} into \eqref{reviewg<1} in order to deduce the uniform boundedness of $u_n$ in $H^1_0(\Omega)$ and  then  the same regularity for its weak limit.

	\medskip 
	 Now we turn our attention to case $f\in L^m(\Omega)$ with $1\le m<\left(\frac{2^*}{1-\gamma}\right)'$. For fixed $n$,  let $0<\varepsilon<\frac{1}{n}$ and $\gamma\le\eta<1$. Consider $(u_n+\varepsilon)^{\eta} - \varepsilon^{\eta}$ as a test function in the weak formulation of \eqref{pbn} in order to get
	\begin{equation}
	\displaystyle \eta\alpha\int_{\Omega} |\nabla u_n|^2(u_n+\varepsilon)^{\eta-1} \le \int_{\Omega} f_n (u_n+\varepsilon)^{\eta-\gamma}.
	\label{aprioriminore1}
	\end{equation}
	The left-hand of \eqref{aprioriminore1} is estimated as follows:
	\begin{equation*}
	\begin{aligned}
	\displaystyle \int_{\Omega} |\nabla u_n|^2(u_n+\varepsilon)^{\eta-1} &= \frac{4}{(\eta+1)^2}\int_{\Omega} |\nabla (u_n+\varepsilon)^\frac{\eta+1}{2}-\varepsilon^\frac{\eta+1}{2}|^2
	\\
	&\ge \frac{4}{(\eta+1)^2\mathcal{S}^2_2}\left(\int_{\Omega} |(u_n+\varepsilon)^\frac{\eta+1}{2}-\varepsilon^\frac{\eta+1}{2}|^{2^*}\right)^{\frac{2}{2^*}}.
	\label{aprioriminore11}
	\end{aligned}
	\end{equation*}
	Thus, letting $\varepsilon$ tend to zero,  one gets
	\begin{equation}
\displaystyle \frac{4\eta\alpha}{(\eta+1)^2\mathcal{S}^2_2}\left(\int_{\Omega} u_n^\frac{2^*(\eta+1)}{2}\right)^{\frac{2}{2^*}} \le \int_{\Omega} f_n u_n^{\eta-\gamma}.
\label{aprioriminore111}
\end{equation}	
	If $m=1$ then we take $\eta=\gamma$ in  the previous obtaining  that $u_n$ is bounded in $L^{\frac{N(\gamma+1)}{N-2}}(\Omega)$. 
	
	 Otherwise, in case $m>1$, we apply the H\"older  inequality in \eqref{aprioriminore111}, yielding to  
	\begin{equation*}
	\left(\int_{\Omega} u_n^\frac{2^*(\eta+1)}{2}\right)^{\frac{2}{2^*}} \le C\int_{\Omega} f_n u_n^{\eta-\gamma} \le C||f||_{L^m(\Omega)} \left(\int_{\Omega} u_n^{(\eta-\gamma)m'}\right)^{\frac{1}{m'}}. 
	\end{equation*}
	We choose $\eta$ such that $\frac{2^*(\eta+1)}{2} = (\eta-\gamma)m'$ which takes to  $\eta= \frac{N(m-1)+\gamma m (N-2)}{N-2m}>\gamma$. This implies that $u_n$ is bounded in $L^{\frac{Nm(\gamma+1)}{N-2m}}(\Omega)$. 
	\\ Thus, we deduce
	\begin{align*}
	\int_{\Omega} |\nabla u_n|^q = \int_{\Omega} \frac{|\nabla 
		u_n|^q}{(u_n+\varepsilon)^{\frac{q(1-\eta)}{2}}}(u_n+\varepsilon)^{\frac{q(1-\eta)}{2}} &\le \left(\int_{\Omega} \frac{|\nabla u_n|^2}{(u_n+\varepsilon)^{(1-\eta)}}\right)^\frac{q}{2} \left( \int_{\Omega} (u_n+\varepsilon)^{\frac{(1-\eta)q}{2-q}}\right)^{\frac{2-q}{2}} 
	\\
	& \le C\left( \int_{\Omega} (u_n+\varepsilon)^{\frac{(1-\eta)q}{2-q}}\right)^{\frac{2-q}{2}},
	\label{aprioriminore12}
	\end{align*}
	where if $q=\frac{Nm(\gamma+1)}{N-m(1-\gamma)}$ we have that, by the  choice of $\eta$, $\frac{(1-\eta)q}{2-q}= \frac{Nm(\gamma+1)}{N-2m}$. This implies that $u_n$ is bounded in $W^{1,q}_0(\Omega)$.
	
	Reasoning similarly to the case $\gamma=1$, we pass to the limit in \eqref{pbn} and we  prove the existence of a solution to \eqref{pbbo} satisfying ii). 
	
	\medskip 
	
	{\bf{Proof of iii)}.} We take $u_n^\gamma$ as a test function in the weak formulation of \eqref{pbn} and, recalling \eqref{stimadb}, we obtain that
	$$\displaystyle c_\omega^{\gamma-1}\alpha\int_{\omega} |\nabla u_n|^2 \le \alpha\int_{\Omega} |\nabla u_n|^2u_n^{\gamma-1} \le \int_{\Omega} f_n \le \int_{\Omega} f,$$
	where $\omega$ is any open subset of $\Omega$ such that $\omega\subset\subset \Omega$; this implies that $u_n$ is locally bounded in $H^1(\Omega)$.
	Moreover we have also shown that 
	$$
	\frac{4}{(\gamma+1)^2}\int_{\Omega} |\nabla u_n^{\frac{\gamma+1}{2}}|^2 =\int_{\Omega} |\nabla u_n|^2u_n^{\gamma-1} \le C,
	$$
	from which we deduce that $u_n^{\frac{\gamma+1}{2}}$ is bounded in $H^1_0(\Omega)$.  With the local estimate in mind the proof of the existence of a solution to \eqref{pbbo} satisfying $u \in H^1_{\rm{loc}}(\Omega)$ strictly follows the reasoning of the previous cases.
\end{proof}

\begin{remark} 
We stress once more  the strong regularization effect given by the singular nonlinearity;  this should  be compared with the case $m=1$ and $\gamma=0$. For instance, if $0<\gamma<1$ then 
$$
\left(\frac{2^*}{1-\gamma}\right)'<(2^*)'; 
$$
i.e. finite energy solutions do exists for a larger class of data. Also notice that, in continuity with ii), if $\gamma=1$, solutions are always in $H^1_0(\Omega)$ even if $f$ is merely in $L^1(\Omega)$. Finally, a more subtle regularizing effect appears formally sending $\gamma$ to $0$. In fact, if $f\in L^1(\Omega)$ then $u \in W^{1,\frac{N(\gamma+1)}{N-1+\gamma}}_0(\Omega)$, while for $\gamma=0$, in general,  one only expects solutions to belong to $W^{1,q}_0(\Omega)$ for any $q<\frac{N }{N-1}$.  \triang 
\end{remark}

\subsection{Existence in case of a general nonlinearity}\label{gennon}

Here we extend the result of the previous section to the case of a  general, possibly non-monotone nonlinearity as zero order term;  we look for solutions to the following problem
\begin{equation}
\begin{cases}
\displaystyle - \operatorname{div}(A(x)\nabla u)= h(u)f &  \text{in }\, \Omega, \\
u=0 & \text{on}\ \partial \Omega,
\label{pbgeneralh}
\end{cases}
\end{equation}
where $A$ satisfies \eqref{Abounded}, $f \in L^1(\Omega)$ is nonnegative and $h:[0,\infty)\to [0,\infty]$ is continuous and finite outside the origin such that 
\begin{equation}
\exists\;\gamma \ge 0, \ c_1,s_1>0: \  h(s) \le \frac{c_1}{s^\gamma}\;\text{ if }\;s<s_1,\ \ h(0)\not=0,
\label{h1}
\end{equation} 
and 
\begin{equation}
\displaystyle {\limsup_{s\to \infty}h(s)<\infty.}
\label{h2}
\end{equation}
 Due to the generality of our assumptions we only face the case of a nonnegative datum $f\in L^1(\Omega)$; regularity results in the spirit of Theorem  \ref{teoreg} and Theorem \ref{boesistenza} are presented in Section \ref{sec:reg} provided a control at infinity is also assumed on the nonlinear term and depending on the summability of the datum. 
\begin{remark}
	Let us spend some words on \eqref{h1} and \eqref{h2};	first observe that  \eqref{h1} covers also the case of bounded  function $h$. Due to the lack of conditions at infinity for the function $h$  one can  not expect, in general, regularity as in  \eqref{lweakdef2} to hold; e.g. considering $h\equiv 1$ in \eqref{pbgeneralh} leads to classic elliptic equations with $L^1$-data for which  infinite energy solutions  are known to appear. Then one has to keep in mind that the regularity effect established in the previous section is basically due to the behavior at infinity of the nonlinearity. As we can not expect better regularity than the classical case in which $h\equiv1$, we have to look, as in that case,  to kind of  truncations  of the solution to  satisfy suitable estimates in the energy space.   This in particular affects the way the boundary datum is achieved. \triang \end{remark}

\medskip

 We set the natural extension of the notion of distributional solution for this case:
\begin{defin}\label{defhu}
	A nonnegative function $u\in W^{1,1}_{\rm{loc}}(\Omega)$ is a {\it distributional solution} to problem \eqref{pbgeneralh} if
	\begin{equation}
	h(u)f\in L^1_{\rm{loc}}(\Omega),
	\label{lweakdef21}
	\end{equation}
	\begin{equation}
	T_k^{\frac{\sigma+1}{2}}(u)\in W^{1,1}_0(\Omega), 	\ \forall k>0, 
	\label{lweakdef22}
	\end{equation}
	and
	\begin{equation} 
	\displaystyle \int_{\Omega}A(x)\nabla u \cdot \nabla \varphi =\int_{\Omega} h(u)f\varphi,\ \ \ \forall \varphi \in C^1_c(\Omega).\label{lweakdef23}
	\end{equation}	
\end{defin}

\begin{theorem}
	\label{esistenzah}
	Let $A$ satisfy \eqref{Abounded}, let $0\le f \in L^1(\Omega)$ and let $h$ satisfy \eqref{h1} and \eqref{h2}. Then there exists a distributional solution $u$ of \eqref{pbgeneralh}. Moreover, $u$ is such that $G_k(u)\in W^{1,q}_0(\Omega)$ for every $q<\frac{N}{N-1}$ and $T_k(u) \in H^1_{\rm loc}(\Omega)$ for every $k>0$.
\end{theorem}
\begin{proof}
Let us consider  nonnegative solutions 
\begin{equation}
\begin{cases}
\displaystyle - \operatorname{div}(A(x)\nabla u_n)= h_n(u_n)f_n &  \text{in }\, \Omega, \\
u_n=0 & \text{on}\ \partial \Omega,
\label{pbngeneral}
\end{cases}
\end{equation}
where $h_n(s):=T_n(h(s))$ and $f_n:=T_n(f)$. The existence of such  nonnegative $u_n\in H^1_0(\Omega)\cap L^\infty(\Omega)$ follows by standard theory (see again \cite{ll}). We take $T_j (G_k(u_n))$ ($j,k>0$) as test function in the weak formulation of  \eqref{pbngeneral}, yielding to
$$
\begin{aligned}\label{do_tg}
\displaystyle \alpha\int_{\Omega}|\nabla T_j (G_k(u_n))|^2 &\le \int_{\Omega}  h_n(u_n)f_nT_j (G_k(u_n))\\
&\le   j\displaystyle \sup_{s\in (k,\infty)}h(s) ||f||_{L^1(\Omega)}.
\end{aligned}	 
$$
The previous estimate allows to apply Lemmas $4.1$ and $4.2$ of \cite{b6} in order to deduce that $G_k(u_n)$ is bounded in   $W^{1,q}_{0}(\Omega)$ for every $q<\frac{N}{N-1}$ and for every $k>0$. 

\medskip 
 Now let $\varphi\in C^1_c(\Omega)$ and let take $(T_k(u_n)-k)\varphi^2$ as a test function in  the weak formulation of \eqref{pbngeneral}, obtaining 
\begin{equation*}
	\alpha\int_\Omega |\nabla T_k(u_n)|^2\varphi^2 + 2\int_\Omega A(x)\nabla u_n \cdot \nabla \varphi \varphi (T_k(u_n)-k) \le 0, 
\end{equation*} 
from which, applying the Young inequality, one gets
\begin{equation*}
\alpha\int_\Omega |\nabla T_k(u_n)|^2\varphi^2 \le  2k \beta\varepsilon \int_\Omega |\nabla T_k(u_n)|^2\varphi^2 +  \frac{2k\beta}{{C_\varepsilon}} \int_{\Omega}|\nabla \varphi|^2, 
\end{equation*} 
and taking $\varepsilon$ small enough one obtains that
\begin{equation}\label{reg:tk}
\int_\Omega |\nabla T_k(u_n)|^2\varphi^2 \le C. 
\end{equation}
Hence, from the estimate on $G_k(u_n)$ in $W^{1,q}_{0}(\Omega)$ for every $q<\frac{N}{N-1}$ and from \eqref{reg:tk}, one deduces that $u_n$ is bounded in $W^{1,q}_{\rm{loc}}(\Omega)$ for every $q<\frac{N}{N-1}$; then there exists a function $u\in W^{1,q}_{\rm{loc}}(\Omega)$  to which, up to a subsequence, $u_n$ converges almost everywhere in $\Omega$ and weakly in $W^{1,q}(\omega)$ for any open subset $\omega\subset\subset\Omega$. 
The last estimate we need is a global one which will lead us to check the homogeneous  boundary condition.
In order to show \eqref{lweakdef22}, we take $T_k^\sigma(u_n)$ as a test function in the weak formulation of \eqref{pbngeneral}, deducing 
\begin{equation*}
	\frac{4\alpha \sigma}{(\sigma+1)^2}\int_\Omega |\nabla T_k^{\frac{\sigma+1}{2}}(u_n)|^2 \le c_1\int_{\{u_n \le s_1\}}f + k^\sigma\sup_{s\in (s_1,\infty)}h(s)\int_{\{u_n>s_1\}}f\le C\,;
\end{equation*}
in particular, up to subsequences,  $T_k^{\frac{\sigma+1}{2}}(u_n)$ converges to $T_k^{\frac{\sigma+1}{2}}(u)$ weakly in  $H^1_0(\Omega)$ and this shows \eqref{lweakdef22}.

Now we want to pass to the limit, with respect to  $n$, in  the weak formulation of  \eqref{pbngeneral}.  By weak convergence we easily pass to the limit in  the term involving the principal part.  For the term concerning $h$, in contrast with the power-like model case, we cannot reproduce the  argument of the  previous section (see also Remark \ref{upositiva} below) as no monotonicity is assumed and we need to carefully handle the zones where $u_n$ are small. 
 
 Let us underline that, from here on, we assume that $h(0)=\infty$; indeed, if $h(0)<\infty$, the passage to the limit is an immediate consequence of the Lebesgue Theorem. 

  Now let $\varphi\in C^1_c(\Omega)$ be nonnegative and, as $u_n$ is locally bounded in $W^{1,q}(\Omega)$ for $q<\frac{N}{N-1}$ with respect to $n$, one obtains 
\begin{equation*}
	\int_\Omega h_n(u_n)f_n \varphi = \int_\Omega A(x)\nabla u_n \cdot \nabla \varphi \le \beta||\nabla \varphi||_{L^\infty(\Omega)} \int_{\{\supp \varphi\}} |\nabla u_n|\le C,
\end{equation*}
and an application of the Fatou Lemma  gives   
\begin{equation}\label{lemstimal1fatou}
\int_{\Omega}h(u)f\varphi\le C,
\end{equation}
namely \eqref{lweakdef21}. 
Moreover, recalling that here $h(0)=\infty$, from \eqref{lemstimal1fatou} one deduces
\begin{equation}\label{upos}
\{u = 0\} \subset \{f = 0\}, 
\end{equation}
up to a set of zero Lebesgue measure. Now,  take $\delta>0$ such that $\delta \not\in \{\eta: |u=\eta|>0\}$,  which is at most a countable set, and write 
\begin{equation}\label{rhs}
\int_{\Omega}h_n(u_n)f_n\varphi = \int_{\{u_n\le \delta\}}h_n(u_n)f_n\varphi + \int_{\{u_n> \delta\}}h_n(u_n)f_n\varphi.
\end{equation}
We want to pass to the limit first as $n\to \infty$ and then as $\delta\to 0^+$.
For the second term of \eqref{rhs} we observe 
$$h_n(u_n)f_n\varphi\chi_{\{u_n> \delta\}} \le \sup_{s\in [\delta,\infty)}[h(s)]\ f\varphi \in L^1(\Omega),$$
which permits to apply the Lebesgue Theorem. Hence one has
\begin{equation*}\label{rhs2}
\lim_{n\to \infty}\int_{\{u_n> \delta\}}h_n(u_n)f_n\varphi= \int_{\{u> \delta\}}h(u)f\varphi.
\end{equation*}
Moreover from \eqref{lemstimal1fatou} one gets that 
$$h(u)f\varphi\chi_{\{u> \delta\}} \le h(u)f\varphi \in L^1(\Omega),$$
and then, once again by the Lebesgue Theorem
\begin{equation}\label{rhs21}
\lim_{\delta\to 0^+}\lim_{n\to \infty}\int_{\{u_n> \delta\}}h_n(u_n)f_n\varphi= \int_{\{u> 0\}}h(u)f\varphi.
\end{equation}

\medskip 
What is left is to show that first term of the right-hand of \eqref{rhs} vanishes in $n$ and then in $\delta$; we take $V_{\delta}(u_n)\varphi$ ($V_{\delta}$ is defined in \eqref{not:Vdelta}) as test function in the weak formulation of \eqref{pbngeneral}, obtaining (recall $V'_{\delta}(s)\le 0$ for $s\ge 0$)
\begin{equation*}\label{limn1}
\begin{aligned}
\int_{\{u_n\le \delta\}}h_n(u_n)f_n\varphi\le \int_{\Omega}h_n(u_n)f_nV_{\delta}(u_n)\varphi \le \int_{\Omega}A(x)\nabla u_n\cdot \nabla \varphi V_{\delta}(u_n),
\end{aligned}
\end{equation*}
and then, as $n\to\infty$, by the weak convergence and by the Lebesgue Theorem, one gets
\begin{equation*}\label{limn2}
\begin{aligned}
\limsup_{n\to\infty}\int_{\{u_n\le \delta\}}h_n(u_n)f_n\varphi\le   \beta\int_{\Omega}|\nabla u||\nabla \varphi | V_{\delta}(u).
\end{aligned}
\end{equation*}
We can now take the limit as $\delta\to 0^+$
\begin{equation}\label{limn3}
\begin{aligned}
\lim_{\delta\to 0^+}\limsup_{n\to\infty}\int_{\{u_n\le \delta\}}h_n(u_n)f_n\varphi\le \beta\int_{\{u=0\}}|\nabla u|| \nabla \varphi|= 0.
\end{aligned}
\end{equation}
  Finally observe that \eqref{rhs21}, \eqref{limn3} and \eqref{upos} imply 
\begin{equation*}
\lim_{n\to \infty}\int_{\Omega}h_n(u_n)f_n\varphi = \int_{\{u>0\}}h(u)f\varphi \stackrel{\eqref{upos}
}{=} \int_{\Omega}h(u)f\varphi,
\end{equation*}
which proves \eqref{lweakdef23} and it concludes the proof. 
\end{proof}

\begin{remark}\label{upositiva}
	  As already mentioned,  in the proof of Theorem \ref{esistenzah}, we avoided the use of the maximum principle by employing the function $V_\delta$ defined in \eqref{not:Vdelta}.  It is worth pointing out that one can reason as in \cite{do,ddo} via  a suitable comparison with an  auxiliary problem  which allows to  deduce that for any $\omega\subset\subset\Omega$
	$$\exists n_0, c_\omega>0 : u_n \ge c_\omega >0 \text{ in } \Omega, \ \forall n\ge n_0.$$
Here  the strict positivity of $h$ is required. This is a  technical assumption which is only needed in case of a measure datum (see Section \ref{sec:nonex}). For merely Lebesgue data one can show  that a solution of \eqref{pbgeneralh} exists and survives up to the first degeneration point of $h$, i.e. the smaller $\overline{s}$ such that $h(\overline{s})=0$, that is $u\leq \overline{s}$. Indeed, if this is the case and just to have a rough idea, one can consider the following approximation problems in place of \eqref{pbngeneral}:
\begin{equation}
	\begin{cases}
		\displaystyle - \operatorname{div}(A(x)\nabla u_n)= \tilde{h}_n(u_n)f_n &  \text{in }\, \Omega, \\
		u_n=0 & \text{on}\ \partial \Omega,
		\label{pbngeneraldeg}
	\end{cases}
\end{equation}
where $\tilde{h}(s) = h(s)$ if $s\le \overline{s}$ and $\tilde{h}(s) = 0$ if $s>\overline{s}$.
\\
Then if one takes $G_{\overline{s}}(u_n)$ as test function in the weak formulation of \eqref{pbngeneraldeg}, after some calculations one yields to $$\int_\Omega |\nabla G_{\overline{s}}(u_n)|^2 = 0,$$
which means that
\begin{equation}\label{hdegen}
	u_n\le \overline{s} \ \ \ \text{a.e. in } \Omega,
\end{equation}
and so $u\le \overline{s}$ almost everywhere in $\Omega$. 
By definition of $\tilde{h}$, \eqref{hdegen}  implies that $u_n$ is  a solution of \eqref{pbngeneral} and, as proven in Theorem \ref{esistenzah},  its almost everywhere limit $u$ is a solution to \eqref{pbgeneralh}.   
\triang \end{remark}

We summarize the (weak) regularity results we have: 

\begin{center}
	\begin{table}[H]
		\setlength{\tabcolsep}{8pt}
		\renewcommand{\arraystretch}{2.4} 
		\begin{tabular}{ | c | c | c | }  
			\cline{1-3}
			$h(s)$ & $f\in L^m(\Omega)$ & Regularity of $u$
			\\
			\hline
			\multirow{2}{*}{$s^{-\gamma}$, $\gamma<1$}
			&$m\ge \left(\frac{2^*}{1-\gamma}\right)'$
			& $u\in H^1_0(\Omega)$
			\\\cline{2-3}
			&
			$1 \le m <\left(\frac{2^*}{1-\gamma}\right)'$
 			& $u\in W^{1,\frac{Nm(\gamma+1)}{N-m(1-\gamma)}}_0(\Omega)$
			\\\cline{1-3}
			$s^{-\gamma}$, $\gamma=1$
			&$m = 1 $
			& $u\in H^1_0(\Omega)$
			\\\cline{1-3}
			$s^{-\gamma}$, $\gamma>1$
			&$m = 1 $
			& $u\in H^1_{\rm loc}(\Omega)$\,,\ \   $u^{\frac{\gamma+1}{2}}\in H^1_{0}(\Omega)$
			\\\cline{1-3}
		    \eqref{h1} with $\gamma\le 1$, \eqref{h2}
			&$m = 1 $
			& $u\in W^{1,q}_0(\Omega)$  $\forall q<\frac{N}{N-1}$
			\\\cline{1-3}
			\multirow{2}{*}{\eqref{h1} with $\gamma > 1$, \eqref{h2}}
			&\multirow{2}{*}{$m = 1$}
			& $G_k(u)\in W^{1,q}_0(\Omega)$ $\forall q<\frac{N}{N-1},$ $\forall k>0$ \\
			&
			&
			 $T_k(u)\in H^1_{\rm loc}(\Omega)$ $\forall k>0$
			\\\cline{1-3}			
		\end{tabular}
		\vspace*{3mm}
		\caption{Regularity of solutions to \eqref{pbgeneralh}} 
	\end{table}
\end{center}

\subsection{Uniqueness of a solution}	
\label{sec:uniquenessweak}

In this section we show that, provided $h$ is non-increasing,  there exists at most one weak solution to \eqref{pbgeneralh}, i.e. a distributional solution in $H^1_0(\Omega)$.  More general uniqueness instances will be illustrated  later in Section \ref{sec:measureex} (renormalized solutions) and Section \ref{op2_uniqueness_distributional} (infinite energy distributional solutions).

\medskip

Here the strategy relies on the possibility to extend the set of test functions in formulation \eqref{lweakdef23} aiming to consider  any  solution  as test function in order to  provide a comparison principle among two different solutions. 

\medskip 

This is the content of the following: 

\begin{lemma}\label{lemextest}
 Let $0\le f\in L^1(\Omega)$ and let  $u$ be a weak solution  to \eqref{pbgeneralh} in the sense of Definition \ref{defhu}. Then  
\begin{equation}\label{uni5}
\displaystyle\int_{\Omega} A(x)\nabla u \cdot \nabla v = \int_{\Omega} h(u)fv , \ \ \ \forall v \in H^1_0(\Omega).
\end{equation}
\end{lemma}

\begin{proof}
	 Let  $u	\in H^1_0(\Omega)$ satisfy 
\begin{equation}
\displaystyle\int_{\Omega} A(x)\nabla u \cdot \nabla \varphi = \int_{\Omega} h(u)f\varphi , \ \ \ \forall \varphi \in C^1_c(\Omega).
\label{distrsol}
\end{equation}	
We consider a nonnegative $v\in H^1_0(\Omega)$ and also a sequence of nonnegative functions $v_{\eta,n}\in C^1_c(\Omega)$ such that  (we call $v_n$ the almost everywhere limit of $v_{\eta,n}$ as $\eta \to 0^+$)
	\begin{equation}\label{propapprox}
	\begin{cases}
	v_{\eta,n} \stackrel{\eta  \to 0^+}{\to} v_{n}  \ \ \ \text{in } H^1_0(\Omega) \text{ and } \text{$*$-weakly in } L^\infty(\Omega), 
	\\
	v_{n} \stackrel{n  \to \infty}{\to} v \ \ \ \ \ \text{in } H^1_0(\Omega),
	\\
	\supp v_n\subset \subset \Omega: 0\le v_n\le v \ \ \ \text{for all } n\in \mathbb{N}.
	\end{cases}
	\end{equation}
	An example of $v_{\eta,n}$ is $\rho_\eta \ast (v \wedge \phi_n)$ ($v \wedge \phi_n:= \inf (v,\phi_n)$) where $\rho_\eta$ is a smooth mollifier while $\phi_n$ is a sequence of nonnegative functions in $C^1_c(\Omega)$ which converges to $v$ in $H^1_0(\Omega)$.
	Hence we can take $v_{\eta,n}$ as a test function in \eqref{distrsol} deducing that 
	\begin{equation*}\label{uni1}
	\int_{\Omega} A(x)\nabla u\cdot\nabla 	v_{\eta,n}  = \int_{\Omega}h(u)f v_{\eta,n}.
	\end{equation*}
	We want to take first $\eta \to 0^+$ and then $n\to \infty$. Since  $v_{\eta,n}$ converges to $v_n$ in $H^1_0(\Omega)$ we can pass to the limit  on the left-hand. For the right-hand one has that
	 $h(u)f \in L^1_{\textrm{loc}}(\Omega)$ and we  pass here to the limit since $v_{\eta,n}$ converges $*$-weakly in $L^\infty(\Omega)$ to $v_n$ which has compact support in $\Omega$. Hence one gets
	\begin{equation}\label{uni2}
	\int_{\Omega} A(x)\nabla u\cdot\nabla v_n  = \int_{\Omega}h(u)fv_n.
	\end{equation} 	
	Now an application of the Young inequality yields to 
	\begin{equation*}\label{uni3}
	\int_{\Omega}h(u)fv_n \le \beta\int_{\Omega} |\nabla u|^{2} + \beta\int_{\Omega}|\nabla v_n|^2,	
	\end{equation*}	
	and, by \eqref{propapprox}, the right-hand of the previous is bounded with respect to $n$. Hence one can apply the Fatou Lemma with respect to $n$, obtaining 
	\begin{equation*}\label{uni4}
	\int_{\Omega}h(u)fv \le C.	
	\end{equation*}	
	Now  we can easily pass to the limit as $n\to \infty$ in \eqref{uni2}. Indeed for the right-hand one can apply the Lebesgue Theorem since 
	$$h(u)fv_n\le h(u)fv\in L^1(\Omega).$$
	Therefore this easily implies that 
	\begin{equation*}
	\int_{\Omega} A(x)\nabla u\cdot\nabla v = \int_{\Omega}h(u)fv,
	\end{equation*} 		 
	for every $v\in H^1_0(\Omega)$.  
	\end{proof}
	
	\begin{theorem}\label{boca}
  Let $0\le f\in L^1(\Omega)$ and let $h$ be non-increasing. There is at most one weak solution to \eqref{pbgeneralh} in the sense of Definition \ref{defhu}. 
\end{theorem}
	\begin{proof}
	  Let $u_1, u_2 \in H^1_0(\Omega)$ be distributional solutions to \eqref{pbgeneralh}. Then it follows by Lemma \ref{lemextest}   that one can take $u_1-u_2$ as a test function in difference between formulation \eqref{uni5} solved respectively   by $u_1$ and by $u_2$. 
	
	  As $h$ is non-increasing, one obtains 
	\begin{equation*}\label{uni6}
	\alpha\io |\nabla u_1 - \nabla u_2|^2\le  \int_{\Omega}(h(u_1)-h(u_2))f(u_1-u_2) \le 0,
	\end{equation*} 	
	which implies that $u_1=u_2$ almost everywhere in $\Omega$. This concludes the proof.  
\end{proof}
\begin{remark}\label{remarkgeneralizzazione}
	The above result can be easily extended to the case of a more general nonlinear leading term, i.e.  to the following type of problems
	\begin{equation*}
	\begin{cases}
	\displaystyle - \operatorname{div} (a(x,u))= h(u)f &  \text{in }\, \Omega, \\
	u=0 & \text{on}\ \partial \Omega,
	\end{cases}
	\end{equation*}
	where $a(x,u)$ has Leray-Lions structure type (i.e. it satisfies \eqref{cara1}-\eqref{cara3} below) and $h$ is a non-increasing function.  
	
	We also remark that  uniqueness of solutions for the case of the $p$-Laplacian is proven in   \cite{CST} by introducing a slightly different 
notion of solution. The more relevant difference with the case we treat here relies in the boundary datum that, in \cite{CST} is assumed by asking that $(u-\varepsilon)^+$ belongs to $W^{1,p}_{0}(\Omega)$, for any $\varepsilon>0$.

	Also observe that the fact that a distributional solution of  \eqref{pbgeneralh} satisfies 
	\eqref{uni5} has a proper independent interest as it shows that a stronger formulation holds even if, in general, $h(u)f$ needs not to be in $L^1(\Omega)$. 
\triang \end{remark}
\begin{remark}
	  Theorem \ref{boca} gives a basic proof of uniqueness for weak solutions.  It is worth to point out that, in Theorem \ref{boesistenza}, we have seen that solutions are in $H^1_0(\Omega)$ if $\gamma<1$ and $f\in L^{\left(\frac{2^*}{1-\gamma}\right)'}(\Omega)$ where we read $\left(\frac{2^*}{1-\gamma}\right)'=1$ if $\gamma=1$.   Then a first interesting question relies on understanding if solutions lying outside of $H^1_0(\Omega)$ actually exist; as shown below, this will be the case since, in general, one can provide instances of infinite energy solutions.  On the other hand, it is of interest comprehend if there are  situations  where the solutions have finite energy, especially in case $\gamma>1$.  
\triang \end{remark}

\section{Weak regularity of the solution}
\label{sec:reg}

 In Section \ref{sec:uniquenessweak} we pointed out as the Sobolev  regularity of the distributional solution is strictly related to its uniqueness; therefore we aim to analyze various occurrences of problem \eqref{pbgeneralh} in which  a solution in $H^1_0(\Omega)$ does exist. The results will obviously depend on both the regularity of the datum $f$ and on the behavior of the nonlinearity $h$.  We already know that, in the model case $h(s)=s^{-\gamma}$,  finite energy solutions exist in case $\gamma<3$ for a smooth datum $f$ (Theorem \ref{teoreg}) and for any integrable $f$ if $\gamma=1$ (Theorem \ref{boesistenza}). Instances of infinite energy solutions can also be produced; besides the smooth case with $\gamma\geq 3$ some further examples will be provided in Example \ref{ex} below. In this section we push forward this analysis to the case of a generic Lebesgue datum $f$ and a merely continuous nonlinearity $h$. 

\medskip

We underline that in the sequel $\Omega$ is a regular bounded open subset of $\mathbb{R}^N$; we will specify the results where the boundary regularity is not needed. For more comments on the boundary of $\Omega$  we refer to Remark \ref{rembordo} below.

\medskip

Precisely we look for a solution $u\in H^1_0(\Omega)$ to the following
\begin{equation}
\begin{cases}
\displaystyle -\operatorname{div} (A(x)\nabla u)= h(u)f &  \text{in}\,\ \Omega, \\
u=0 & \text{on}\ \partial \Omega,
\end{cases}\label{op2_regular}
\end{equation}
where, for most of the section, $A$ is a symmetric matrix   satisfying 
\begin{equation}
A\in C^{0,1}(\overline{\Omega}): \ \exists \alpha,\beta >0  \ \ A(x)\xi \cdot \xi \ge \alpha|\xi|^2, \ |A(x)|\le \beta.\label{lip}
\end{equation}  
The nonnegative datum $f\in L^m(\Omega)$ with $m\ge 1$.  The function $h:[0,\infty)\to [0,\infty]$ is continuous, finite outside the origin,  satisfying  \eqref{h1},  and for this section we could require a  specific behavior at infinity:
\begin{equation}\label{h2bis}
\exists\;\theta\ge 0, c_2,s_2>0: \ h(s) \le \frac{c_2}{s^\theta}\; \ \text{if}\ \;s>s_2.
\end{equation}
The above request is someway natural as confirmed by Example \ref{exnoth1} below which shows that, in general, solutions have not finite energy if $h$ does not degenerate at infinity at a certain rate.

\medskip

Let us stress that for the current section by $u_n$ we mean the  nonnegative solution to 
\begin{equation}\label{pbgeneralapprox}\begin{cases}
-\operatorname{div}(A(x)\nabla u_n)= h_n(u_n)f_n & \text{in} \ \Omega,\\
u_n=0 & \text{on}\ \partial\Omega,
\end{cases}\end{equation}
where  $h_n(s):= T_n(h(s))$  and $f_n:= T_n(f)$. We observe that, in Theorem \ref{esistenzah}, we have shown that $u_n$ converges to a distributional solution $u$ to \eqref{op2_regular}. Thus, our aim becomes proving estimates on the sequence $u_n$ in $H^1_0(\Omega)$; this will be sufficient to deduce the existence of a distributional solution $u\in H^1_0(\Omega)$. We also highlight that, if the function $h$ is non-increasing and the datum $f\in L^1(\Omega)$ then, as already shown in Theorem \ref{boesistenza}, one can prove that $u_n$ is non-decreasing in $n$.
 We finally recall that in Theorem \ref{boesistenza} we have already shown that, in case $A$ is an elliptic and bounded matrix, $u$ is in $H^1_0(\Omega)$ under the following assumptions:
\begin{itemize}
	\item [i)] if $h(s)= \frac{1}{s}$ and $0\le f\in L^1(\Omega)$,
	\item [ii)] if $h(s)= \frac{1}{s^\gamma}$ with $\gamma<1$ and $0\le f\in L^{(\frac{2^*}{1-\gamma})'}(\Omega)$.	
\end{itemize}
Moreover, if  $\gamma>1$, the solutions found in Theorem \ref{boesistenza} only belong to $H^1(\Omega)$ locally, and not, in general,  to $H^1_0(\Omega)$; see also Example \ref{ex} below.

\subsection{Case $\gamma\le 1$}
\noindent First of all, we present a generalization of the result i) listed above. In some sense we look for the right behaviour of $h$ at the origin and at  infinity in order to have finite energy solutions even in presence of merely integrable data $f$'s. It is also worth to point out that the following theorem does not need any regularity on $\Omega$ nor on the matrix $A$ apart from boundedness; moreover it can be simply extended to nonlinear operators satisfying Leray-Lions structure conditions.
\begin{theorem}\label{reg_generico}
	Let $A$ satisfy \eqref{Abounded} and let $h$ satisfy \eqref{h1} and \eqref{h2bis} where $\gamma\le1$ and $\theta\ge1$. If $0\le f\in L^1(\Omega)$ then there exists a solution $u\in H^1_0(\Omega)$ to problem \eqref{op2_regular}.
\end{theorem}
\begin{proof}
	\noindent 
	 As we have already pointed out, we need some a priori estimates on the sequence of approximating solutions $u_n$ to \eqref{pbgeneralapprox} in $H^1_0(\Omega)$, so that we can proceed as in Section \ref{sec:integrable} to pass to the limit deducing the existence of the  sought solution.	
	\medskip
	
	To this aim, we take $u_n$ as a test function in \eqref{pbgeneralapprox} obtaining 
	\begin{align*}
	\alpha\int_{\Omega} |\nabla u_n|^2 &\le\int_{\Omega} h_n(u_n)f_n u_n \leq c_1\int_{\{u_n<s_1\}}u_n^{1-\gamma}f_n+ \max_{s\in[s_1,s_2]}h(s)\int_{\{s_1\le u_n \le s_2\}}f_n u_n \notag 
	\\
	&+c_2\int_{\{u_n>s_2\}}u_n^{1-\theta}f_n\le c_1s_1^{1-\gamma}\int_{\{u_n<s_1\}}f + s_2\max_{s\in[s_1,s_2]}h(s)\int_{\{s_1\le u_n \le s_2\}}f  
	\\
	&+c_2s_2^{1-\theta}\int_{\{u_n>s_2\}}f \le C.
	\label{stimavicino0}
	\end{align*}
	The proof is concluded.
\end{proof}

\noindent We now highlight that a milder control on $h$ at infinity (namely $\theta < 1$) is not enough in order to ensure the existence of finite energy solutions as the following example shows.
\begin{example}\label{exnoth1}
	For $N>2$ we fix $\theta<1$ and we choose two parameters $m$ and $q$ such that 
	$$1\le m< q<\frac{2N}{N+2} \text{   and   } \theta\le\frac{N(q-m)}{m(N-2q)}.$$
	Let us consider the following
	\begin{equation*}
	\begin{cases}
	\displaystyle -\Delta u= f &  \text{in}\ \Omega, \\
	u=0 & \text{on}\ \partial \Omega,
	\end{cases}
	\end{equation*}
	where $0\le f\in L^q(\Omega)$. By classical regularity results, $u$ is known to be in $W^{1,q^*}_0(\Omega)$ ($q^*$ is the Sobolev embedding exponent) but not, in general,  in $H^1_0(\Omega)$; we fix an $f$ doing this job. 
	
	\medskip 
	
	Now consider a continuous and positive $h$ on which we assume the following control from below
	$$
	\exists\;\theta <1, \ \ c_2,s_2>0: \ h(s) \ge \frac{c_2}{s^\theta}\; \ \text{if}\ \;s>s_2.
	$$ 
	Then we have that
	\begin{equation*}
	\begin{cases}
	\displaystyle -\Delta u= h(u)g &  \text{in}\ \Omega, \\
	u=0 & \text{on}\ \partial \Omega,
	\end{cases}
	\end{equation*}
	where $g:= fh(u)^{-1}$.
	Observe that  $g\in L^m(\Omega)$: indeed, since $m<q$, we have using the H\"older inequality
	\begin{align*}
	\displaystyle \int_\Omega g^m &\le  \sup_{s\in [0,s_2)} h(s)^{-m} \int_{\{u\leq s_2\}} f^m + \frac{1}{c_2^m}\int_{\{u>s_2\}} f^mu^{\theta m}
	\\
	&\le \sup_{s\in [0,s_2)} h(s)^{-m} \int_{\{u\leq s_2\}} f^m + C\left( \int_{\Omega} u ^{\frac{\theta m q}{q-m}}\right)^{\frac{q-m}{q}}<\infty,
	\end{align*}
	since $\frac{\theta m q}{q-m} \le \frac{q N}{N- 2q}=q^{**}$. Thus, we have found a solution to our singular problem not belonging to $H^1_0(\Omega)$ in case $\theta < 1$. Let us also observe that the behaviour at zero plays no roles. 
\end{example}	

\medskip
	
The assumption on $\theta$ can be relaxed by assuming some further requests on $f$, namely more regularity inside and a control near the boundary of the following type (recall $\Omega_\varepsilon$ is defined in
 \eqref{not:omegaeps})
\begin{equation}\label{condfaltoL1}
f(x)\le \frac{C}{d(x)} \text{  a.e.  in }\Omega_\varepsilon, 
\end{equation}
where $\varepsilon$ is small enough in order to guarantee that $\Omega\setminus\overline{\Omega}_\varepsilon$  is smooth and compactly contained in $\Omega$.
\begin{theorem}
	Let $A$ be a symmetric matrix satisfying \eqref{lip} and $0\le f\in L^1(\Omega) \cap L^m(\Omega\setminus\overline{\Omega}_\varepsilon)$ with $m > \frac{N}{2}$ satisfying \eqref{condfaltoL1}.
	 If $h$ is a   function satisfying \eqref{h1} and \eqref{h2bis} with $\gamma<1$ and $\theta>0$. Then there exists a solution $u$ to \eqref{op2_regular} belonging to $H^1_0(\Omega)$.
	\label{gamma<1}
\end{theorem}
\begin{proof}
	Without loss of generality we assume $\theta < 1$ otherwise we can apply Theorem \ref{reg_generico} to conclude. From a classical De Giorgi-Stampacchia  regularity results one has that  $u_n$ is in $C(\Omega\setminus\Omega_\varepsilon)$ and its norm is independent from $n$. 
	We take $u_n$ as a test function in the weak formulation of \eqref{pbgeneralapprox} obtaining
	\begin{equation}
	\begin{aligned}	\label{stimaungammaminore}
	\displaystyle \alpha\int_{\Omega} |\nabla u_n|^2 & \le \int_{\Omega} h_n(u_n)f_nu_n \le c_1\int_{\Omega_\varepsilon\cap \{u_n< s_1\}}u_n^{1-\gamma}f_n + \int_{\Omega_\varepsilon\cap \{s_1 \le u_n \le s_2 \}}h_n(u_n)f_nu_n
	\\ 
	&+ c_2\int_{\Omega_\varepsilon\cap \{u_n> s_2\}}u_n^{1-\theta}f_n + 	\int_{\Omega\setminus\overline{\Omega}_\varepsilon}  h_n(u_n)f_nu_n \le  c_1s_1^{1-\gamma}\int_{\Omega_\varepsilon\cap \{u_n< s_1\}}f_n
	\\ 
	&+		s_2 \max_{s\in [s_1,s_2]} h(s) \int_{\Omega_\varepsilon\cap \{s_1 \le u_n \le s_2 \}}f +  c_2\int_{\Omega_\varepsilon\cap \{u_n> s_2\}}f_nu_n^{1-\theta}  
	\\ 	
	&+ \max_{s\in [c_{\Omega\setminus\overline{\Omega}_\varepsilon}, \infty)}h(s) ||u_n||_{L^\infty(\Omega\setminus\overline{\Omega}_\varepsilon)}\int_{\Omega\setminus\overline{\Omega}_\varepsilon}  f,
	\end{aligned} 
	\end{equation}
	where we have also employed Remark \ref{upositiva}.
	Inequality \eqref{stimaungammaminore} implies that in order to provide some boundedness on the sequence $u_n$ in $H^1_0(\Omega)$ we need to 
	control 
	$$\int_{\Omega_\varepsilon\cap \{u_n> s_2\}}f_n u_n^{1-\theta}.$$
	Let us consider the smooth  nonnegative solution to the following problem
	\begin{equation}
	\label{aux}
	\begin{cases}
	-\operatorname{div}(A(x)\nabla w_n)= \overline{h}_n(w_n)f_n & \text{in}\ \Omega,\\
	w_n=0 & \text{on}\ \partial\Omega,
	\end{cases}
	\end{equation}		
	 with $\overline{h}_n(s):=T_n(\overline{h}(s))$ where the  function  $\overline{h}:[0,\infty)\to [0,\infty]$
	is both non-increasing and such that $\overline{h}(s)\ge h(s)$ for any $s>0$. The existence of such an $h$ is easy (see for instance \cite{do, lops}).

		\begin{figure}[htbp]\centering
\includegraphics[width=3in]{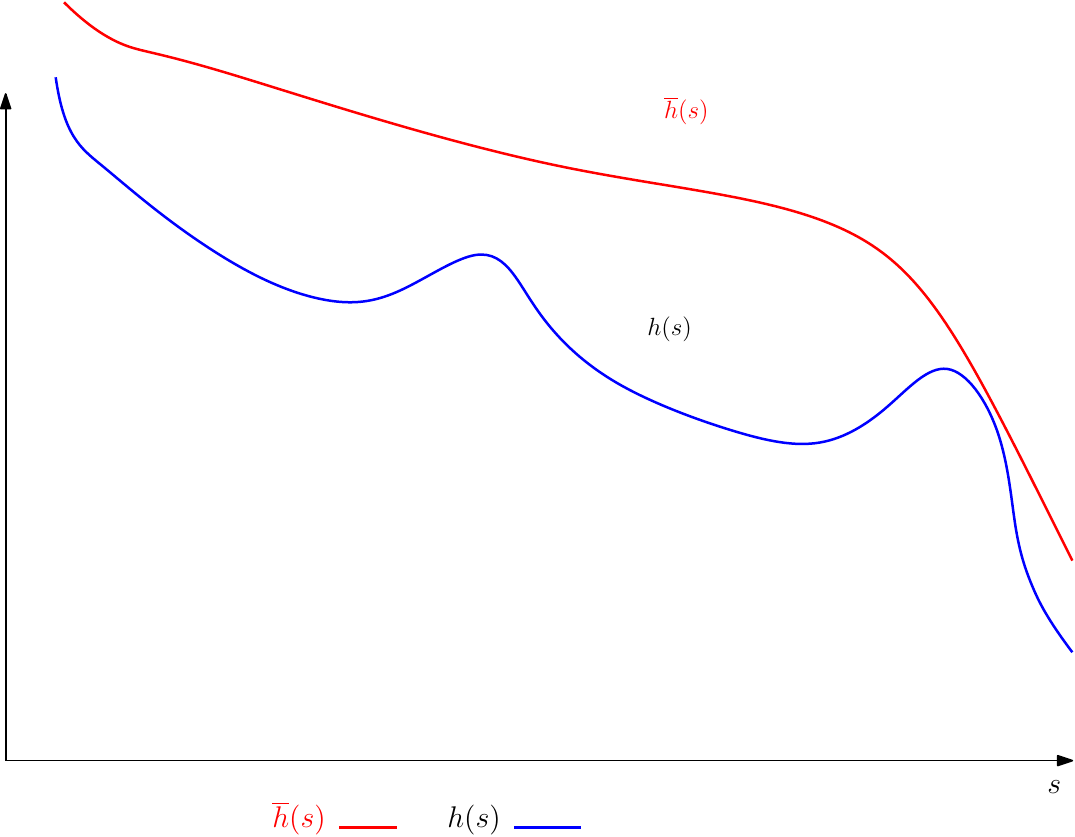}
\caption{A monotone non-increasing $\overline{h}$ above $h$}\label{gover}
\end{figure}

	By comparison $u_n$ is a sub-solution to problem \eqref{aux} and so $u_n\le w_n$ in $\Omega$; indeed  it is sufficient to test   with $(w_n - u_n)^-$.
	
	Now we look for a super-solution to \eqref{aux} in $\Omega_\varepsilon$   in the form of a positive power of $\varphi_{1,A}$ (as defined in \eqref{not:phi1})  in order to prove that for some $C, t>0$, then
	\begin{equation}
	C\varphi_{1,A}^t\ge w_n\ge  u_n, \ \ \ \text{in } \Omega_\varepsilon.
	\label{sopragamma<1}
	\end{equation}
	We fix $t=\frac{1}{\gamma+1}$ and we observe that if $\gamma<1$ then $\frac{1}{2}<t<1$ and it follows that $\varphi_{1,A}^t\in H^1(\Omega_\varepsilon)$.
	We need to prove the first inequality in 
	\begin{align*}			\label{soprasolgammaminore}
	-\operatorname{div}(A(x)\nabla (C\varphi_{1,A}^t))&=\overline{h}(C\varphi_{1,A}^t)\left(Ct(1-t)\frac{\varphi_{1,A}^{t-2}}{\overline{h}(C\varphi_{1,A}^t)}A(x)\nabla \varphi_{1,A}\cdot\nabla \varphi_{1,A}  + Ct\lambda_1\frac{\varphi_{1,A}^{t}}{\overline{h}(C\varphi_{1,A}^t)}\right)
	\\ \nonumber
	&\ge f\overline{h}(C\varphi_{1,A}^t)\ge f_n\overline{h}_n(C\varphi_{1,A}^t)  \text{ in } \Omega_\varepsilon.
	\end{align*}
	Dropping positive terms and recalling both Lemma \ref{hopfdiaz} and \eqref{condfaltoL1} the previous is implied by the request
	$$\alpha Ct(1-t)\frac{\varphi_{1,A}^{t-2}}{\overline{h}(C\varphi_{1,A}^t)}|\nabla \varphi_{1,A}|^2\ge \frac{c}{\varphi_{1,A}} \ge \frac{c}{d}\ge f,$$
	that, in view of Lemma \ref{hopfvarphi2}, essentially reduces (up to normalization of the constants) in proving
	that there exists a positive constant $C$ such that (recall that $t=\frac{1}{\gamma+1}$) 
	\begin{equation}\label{maxo}
	\overline{h}(C\varphi_{1,A}^t)(C\varphi_{1,A}^t)^\gamma\le C^{ 1+\gamma} \ \ \ \text{in }\Omega_\varepsilon.
	\end{equation}
	We have 
	$$\overline{h}(C\varphi_{1,A}^t)(C\varphi_{1,A}^t)^\gamma\le\max (c_1, \max_{[s_1, \infty)}\overline{h}(s)(C\varphi_{1,A}^t)^\gamma),$$
	and \eqref{maxo} is satisfied up the following choice
	$$C\ge \max (c_1^t,\max_{[s_1, \infty)}\overline{h}(s)||\varphi_{1,A}^t||_{L^\infty(\Omega_\varepsilon)}^\gamma).$$
	By possibly increasing the value of such a $C$ one can also assume
	$$C\varphi_{1,A}^t\ge w_n, \ \ \ \text{in } \partial(\Omega\setminus\overline{\Omega}_\varepsilon),$$ 
	then we can apply a comparison principle (see for example Theorem $10.7$ and pages $45-46$ of \cite{gt}) obtaining that \eqref{sopragamma<1} holds.\\
	Thus, we estimate 
	\begin{equation*}
	\displaystyle \int_{\Omega_\varepsilon\cap \{u_n>s_2\}}u_n^{1-\theta}f_n  \le \int_{\Omega_\varepsilon}\frac{c\varphi_{1,A}^{\frac{1-\theta}{\gamma+1}}}{d}\le \int_{\Omega_\varepsilon}\frac{c}{d^{\frac{\gamma+\theta}{\gamma+1}}}<\infty,
	\end{equation*}
	since $\theta<1$. This concludes the proof. 
\end{proof}

\medskip 

  Let us conclude with the case $\gamma\leq 1$ by making the following observation. As we know, in the model case, one has finite energy solutions for $\gamma=1$ and any $f\in L^1(\Omega)$;  this can be view as a limit point of the following general regularity criterion: 
\begin{Proposition}
	 Let $h(s) = s^{-\gamma}$ with $\gamma<1$, let $0\le f\in L^1(\Omega)$ and let $u$ be a distributional solution to \eqref{op2_regular}. Then $u\in H^1_0(\Omega)$ if and only if  
	\begin{equation}\label{necsuf}
	\int_{\Omega} fu^{1-\gamma}<\infty.
	\end{equation}
\end{Proposition}
\begin{proof}
	 If $u\in H^1_0(\Omega)$,   using  Lemma \ref{lemextest},  we can use $u$ as a test function  in \eqref{op2_regular}  showing that   \eqref{necsuf} holds.
	\\ 
	Now assume the validity of  \eqref{necsuf}  and consider the approximation given by \eqref{pbn}, namely
	\begin{equation*}\begin{cases}
	\dis-\operatorname{div}(A(x)\nabla u_n)= \frac{f_n}{(u_n+\frac{1}{n})^\gamma} & \text{in}\ \Omega,\\
	u_n=0 & \text{on}\ \partial\Omega.
	\end{cases}\end{equation*}
	We take $u_n$ as a test function in the previous obtaining	
	$$\dis \alpha \int_{\Omega}|\nabla u_n|^2 \le  \int_{\Omega}f u_n^{1-\gamma} \le \int_{\Omega}f u^{1-\gamma}<\infty,$$	
	where in the last step we also used that $u_n$ is non-decreasing with respect to n, and we conclude by weak lower semicontinuity.	
\end{proof}
\begin{remark} 
One could conjecture that the above kind of regularity principle extends for any $\gamma> 0$ (or, more, for any $h$); namely one should wonder whether, given a nonnegative $f\in L^1(\Omega)$ and the solution $u$ to \eqref{op2_regular}, it holds that 
$$u\in H^1_0(\Omega) \text{  if and only if  } \int_{\Omega} fh(u)u<	\infty.$$
Observe that if $h(s)s$ is non-decreasing then a simple re-adaption of the above proof works fine. Also notice that solutions in Example \ref{ex} below will satisfy (sharply) this criterion: if $\gamma> 1$ then $\int_{\Omega}fu^{1-\gamma}<\infty$ if and only if $\gamma< 3 -\frac{2}{m}$. 
\triang \end{remark}

\subsection{Case $\gamma>1$}
\noindent For merely nonnegative data we have the following.
\begin{theorem}\label{teogammamaggiore}
	Let $A$ satisfy \eqref{lip} and let $0\le f\in L^m(\Omega)$ with $m> 1$. If $h$ satisfies \eqref{h1} and \eqref{h2bis} with $\theta\ge 1$ then there exists a solution $u$ to \eqref{op2_regular} belonging to $H^1_0(\Omega)$ provided
	$$\displaystyle 1<\gamma<2 - \frac{1}{m}.$$   
\end{theorem}
\begin{proof}
	We take $u_n$ as a test function in  the weak formulation of \eqref{pbgeneralapprox}, yielding to
	\begin{align*}
	\displaystyle \alpha\int_{\Omega}|\nabla u_n|^2 &\le \int_{\Omega}h_n(u_n)f_nu_n \le c_1\int_{\{u_n< s_1\}}f_n u_n^{1-\gamma} + \int_{\{s_1\le u_n \le s_2\}}h_n(u_n) f_n u_n \nonumber \\
	&+ c_2\int_{\{u_n> s_2\}}f_n u_n^{1-\theta} \le  c_1\int_{\{u_n< s_1\}}f_n u_n^{1-\gamma}  + s_2\max_{s\in [s_1,s_2]}h(s) \int_{\{s_1\le u_n \le s_2\}} f_n  \nonumber \\
	&+ c_2s_2^{1-\theta}\int_{\{u_n> s_2\}}f_n.
	\label{stimaun1}
	\end{align*}
	\noindent Thus in order to have an estimate in $H^1$ for $u_n$ we just need to control
	$$\int_{\{u_n< s_1\}}f_n u_n^{1-\gamma}.$$
	Now let us consider a     function $\underline{h}:[0,\infty)\to [0,\infty)$ such that   $\underline{h}$ is bounded, non-increasing with respect to  $s$ and $\underline{h}(s)\leq h(s)$ for all $s\geq0$. In particular, for   $n$ large enough  we may assume  $\underline{h}(s)\le h_n(s)$   for all $s\geq 0$ (see again  \cite{do, lops} for a possible construction of such an $\underline{h}$).
	\begin{figure}[htbp]\centering
\includegraphics[width=3in]{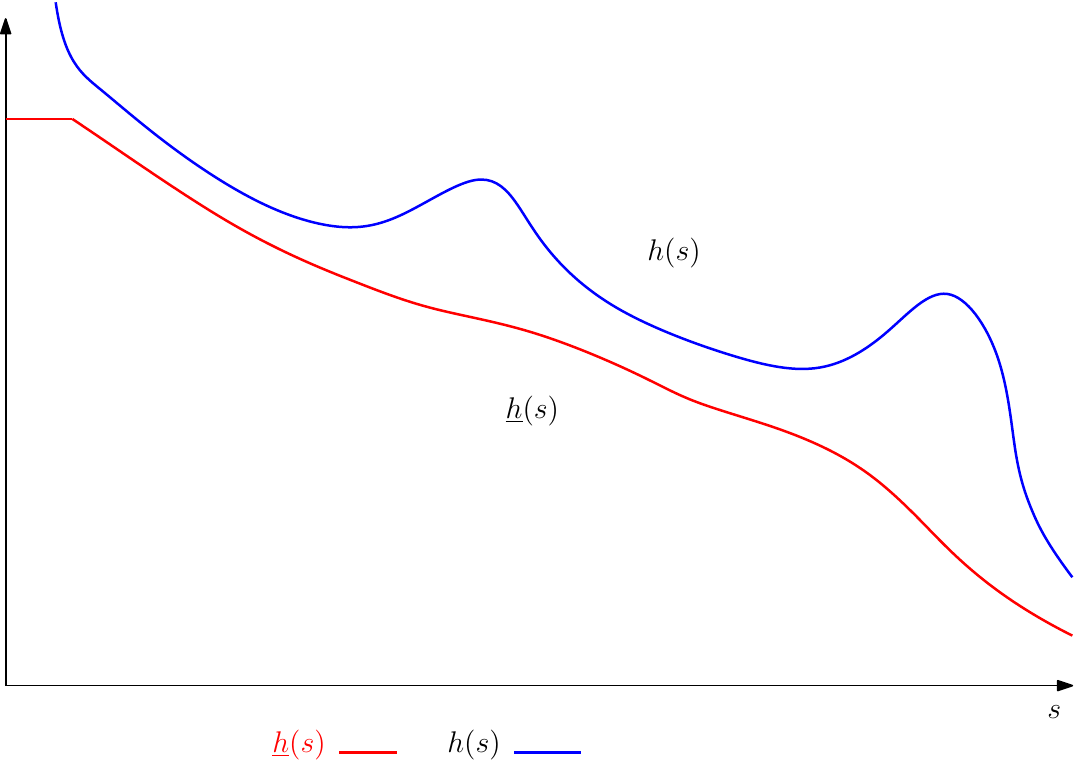}
\caption{A monotone non-increasing $\underline{h}$ below $h$}\label{g}
\end{figure}

	Let us consider $v_n\in H^1_0(\Omega)$ solution to the following problem
	\begin{equation}
	\label{auxsotto}
	\begin{cases}
	-\operatorname{div}(A(x)\nabla v_n)= \underline{h}_n(v_n)f_n & \text{in} \ \Omega,\\
	v_n=0 & \text{on} \ \partial\Omega.
	\end{cases}
	\end{equation}				
	By standard comparison  we have that $u_n\ge v_1$ in $\Omega$. As  $v_1\in C^1(\overline{\Omega})$ (it follows by standard regularity results, see for instance  Theorem $3.17$ of \cite{tro}) we can apply Hopf's Lemma \ref{hopf} to $v_1$ deducing
	$$v_1(x)\ge Cd(x) \ \  \forall x\in\Omega.$$ 
	Thus, it follows from H\"older inequality and from the previous that
	$$\int_{\{u_n< s_1\}}f_n u_n^{1-\gamma} \le \int_{\Omega}fv_1^{1-\gamma}\le ||f||_{L^m(\Omega)} \int_{\Omega} \frac{C}{d^{(\gamma-1)m'}} < \infty,$$
	since $\displaystyle \gamma<2 - \frac{1}{m}$. This concludes the proof.
\end{proof}

\medskip

Though, in this generality, it seems not  easy to improve the threshold on $\gamma$ given in the above theorem, it turns out to be not the optimal one. For instance, let $f>0$ in $L^m(\Omega)$ with $m>1$ in the model case  $h(s) = s^{-\gamma}$. Then, in \cite{am} the authors prove the existence of a solution $u\in H^1_0(\Omega)$ to \eqref{pbbo} provided $\gamma<\frac{3m-1}{m+1}$.  Moreover prototypical examples show however that finite energy solutions can be found up to $\gamma<3-\frac{2}{m}$ (see Example \ref{ex} below). It is worth to underline that 
$$3-\frac{2}{m}>\frac{3m-1}{m+1}$$
and, as $m$ tends to infinity, one formally recovers the Lazer-McKenna threshold $\gamma<3$.
Let us also remark that, on the other hand, as $m \to 1^+$  the regularity threshold tends to one in continuity with case $\gamma = 1$ for which finite energy solutions always exist for any $f\in L^1(\Omega)$. Apart from explicit examples, we also refer to \cite{zach} and \cite{bougia} in which, as suggested in \cite{lm}, these thresholds are reached in the case of the Laplacian and a smooth and bounded away from zero datum $f$ that blows up uniformly at $\partial\Omega$ at a precise rate.  Let us only mention the, in some sense, opposite case of $f$ having compact support on $\Omega$: if $\gamma>0$, indeed,  then the estimate on $u_n$ in $H^1_0(\Omega)$ is for free for any $f\in L^1(\Omega)$.

\medskip

  In the model  case $h(s)=s^{-\gamma}$ we are able to prove the following sharp result concerning  problem \eqref{pbbo}. 
 
\begin{theorem}
	Let $A$ be a symmetric matrix satisfying \eqref{lip},  and   let  $\gamma>1$.   Then there exists a solution $u \in H^1_0(\Omega)$ for any $0<f\in L^{m}(\Omega)$ ($m>1$) to \eqref{pbbo}  if and only if $\gamma<3-\frac{2}{m}$. 
	\label{teosharp}
\end{theorem}
\begin{proof}
Assume first that $\gamma <3-\frac{2}{m}$. To show that $u$ belongs to $H^1_0(\Omega)$  we employ a version of a  result in \cite{sz}; in fact,  a line by line re-adaptation to the case a bounded matrix of the proof of \cite[Theorem 1]{sz} allows us to state that,  
for $\gamma>1$ and $0<f\in L^1(\Omega)$, there exists a unique solution $u\in H^1_0(\Omega)$ to \eqref{pbbo} if and only if there exists a function $u_0\in H^1_0(\Omega)$ such that 
	\begin{equation}
	\int_{\Omega}fu_0^{1-\gamma}<\infty.
	\label{condcinese}
	\end{equation}

 We aim to apply  \eqref{condcinese} with $u_{0}=d(x)^{t}$. Indeed, by \eqref{hopfvarphi} it is possible to choose a suitable $t>\frac{1}{2}$ such that,  as $\gamma <3-\frac{2}{m}$, by H\"older inequality one has
	$$
	\int_\Omega f u_0^{1-\gamma}\le  C\left(\io d^{t(1-\gamma)m'}\right)^\frac{1}{m'}<\infty.
	$$ 
	\noindent In order to prove optimality we let $\gamma\geq 3-\frac{2}{m}$,  	\begin{equation*}
	f(x) :=  \max \left(\frac{1}{d(x)^{\frac{1}{m}}\log\left(\frac{1}{d(x)}\right)},1\right),
	\label{condizionef}
	\end{equation*}
	and $f_{n}:=T_{n}(f)$. 
	\\Then one can show the existence of a constant $C>0$ such that $C\varphi_{1,A}^{t}$ with 
	\begin{equation*}
	\displaystyle t = \frac{2}{\gamma+1} - \frac{1}{m(\gamma+1)}.
	\end{equation*}
	(observe that $0<t\leq \frac12$) is a super-solution to the approximating problems
	\begin{equation*}\begin{cases}
	\dis-\operatorname{div}(A(x)\nabla u_n)= \frac{f_n}{(u_n+\frac{1}{n})^\gamma} & \text{in} \ \Omega,\\
	u_n=0 & \text{on} \ \partial\Omega.
	\end{cases}\end{equation*}
	\\ \noindent Indeed, since 
	$$\frac{f}{C^\gamma\varphi_{1,A}^{t\gamma}} \ge \frac{f_n}{(C\varphi_{1,A}^{t}+\frac{1}{n})^\gamma},$$
	then we only need to show that 
	$$\displaystyle -\operatorname{div}(A(x)\nabla (C\varphi_{1,A}^t)) = \frac{1}{C^\gamma\varphi_{1,A}^{\gamma t}}\left(C^{1+\gamma}t(1-t)A(x)\nabla \varphi_{1,A}\cdot\nabla \varphi_{1,A} \varphi_{1,A}^{t-2+\gamma t} + C^{1+\gamma}\lambda_1t\varphi_{1,A}^{t+\gamma t}\right)\ge \frac{f}{C^\gamma\varphi_{1,A}^{\gamma t}},$$
	that is implied by
	\begin{equation}
	\alpha C^{1+\gamma}t(1-t)|\nabla \varphi_{1,A}|^2\varphi_{1,A}^{t-2+\gamma t} + C^{1+\gamma}\lambda_1t\varphi_{1,A}^{t+\gamma t}\ge f.
	\label{stimaesempio1}
	\end{equation}
	\noindent Let $\dys \varepsilon < \frac{1}{e}$ be a small enough positive number such that $\Omega\setminus\overline{\Omega}_\varepsilon$ is smooth and compactly contained in $\Omega$ and observe that both terms on the left-hand of \eqref{stimaesempio1} are nonnegative.
	Thus, if $x \in \Omega\setminus\overline{\Omega}_\varepsilon$ we have  
	\begin{equation*}
	\displaystyle C\ge \left(\frac{||f||_{L^{\infty}(\Omega\setminus\overline{\Omega}_\varepsilon)}}{\lambda_1 t\displaystyle\min_{x \in \Omega\setminus{\Omega}_\varepsilon}{\varphi_{1,A}^{t+\gamma t}}}\right)^{\frac{1}{1+\gamma}}.
	\label{stimaesempio2}
	\end{equation*}
	Otherwise, if $x \in \Omega_\varepsilon$, we require 
	\begin{align*}
	\alpha C^{1+\gamma}t(1-t)|\nabla \varphi_{1,A}|^2\varphi_{1,A}^{t-2+\gamma t}\ge \frac{1}{d^{\frac{1}{m}}}\ge \frac{1}{d^{\frac{1}{m}}\log\left(\frac{1}{d}\right)} = f,
	\label{stimaesempio3}
	\end{align*}
	recalling that  $\dys d(x)<\varepsilon<\frac{1}{e}$ for any $x\in\Omega_\varepsilon$.
	Thus we need to show that 
	\begin{align*}
	\alpha C^{1+\gamma}t(1-t)|\nabla \varphi_{1,A}|^2\varphi_{1,A}^{t-2+\gamma t}&\ge  \frac{\alpha C^{1+\gamma}\displaystyle\min_{x \in \overline{\Omega}_\varepsilon} |\nabla \varphi_{1,A}|^2 t(1-t)}{\varphi_{1,A}^{-t+2-\gamma t}} 
	\\
	&\ge \frac{\alpha C^{1+\gamma}\displaystyle\min_{x \in \overline{\Omega}_\varepsilon} |\nabla \varphi_{1,A}|^2 t(1-t)c_d}{d^{-t+2-\gamma t}} \ge \frac{1}{d^\frac{1}{m}}.
	\end{align*}
	The last inequality in the previous holds for $C$ big enough since it follows from the choice of $t$ that $2-t-\gamma t=\frac{1}{m}$. 
	\noindent As $C\varphi_{1,A}^t$ and $u_n$ are continuous up to the boundary and regular inside we can apply a comparison principle (see, once again, Theorem $10.7$ and pages 45-46 of \cite{gt}) in order to deduce $C\varphi_{1,A}^t\ge u_n$ and so \begin{equation} C\varphi_{1,A}^t\ge u. \label{sopragamma>1} \end{equation} 
	
	 It is easy to check that $u$ is the unique (Theorem \ref{boca}) solution of problem \eqref{pbbo} found in Theorem  \ref{boesistenza}.  
	
	\noindent Now suppose by contradiction that $u\in H^1_0(\Omega)$, then  we can use $u$  as test function in \eqref{pbbo} obtaining, as $\gamma \ge 3-\frac{2}{m}$ and  recalling \eqref{hopfvarphi}, that
	$$\beta\displaystyle \int_{\Omega} |\nabla u|^2 \ge \int_{\Omega}\frac{f}{u^{\gamma-1}} \stackrel{\eqref{sopragamma>1}}{\ge} \int_{\Omega}f\varphi_{1,A}^{t(1-\gamma)}=\infty,$$
	which is a contradiction. 
	
\end{proof}

We conclude with a summarizing table of the assumptions under which the solution to \eqref{op2_regular} has been shown to be $H^1_0(\Omega)$. Recall that $A$ satisfies \eqref{lip} and, in the last two lines, $A$ is also symmetric.

\begin{center}
	\begin{table}[H]
		\setlength{\tabcolsep}{8pt}
		\renewcommand{\arraystretch}{2.4} 
		\begin{tabular}{ | c | c | }  
			\cline{1-2}
			$h(s)$ & $f$
			\\\cline{1-2}
			$s^{-\gamma}$, $\gamma<1$
			&$L^{\left(\frac{2^*}{1-\gamma}\right)'}(\Omega) $
			\\\cline{1-2}
			\eqref{h1} with $\gamma\le 1$, \eqref{h2bis} with $\theta\ge 1$
			&$L^1(\Omega) $
			\\\cline{1-2}
			\eqref{h1} with $\gamma< 1$, \eqref{h2bis} with $\theta> 0$
			&$L^1(\Omega)\cap L^{m}(\Omega\setminus\Omega_{\varepsilon})$ with $m>\frac{N}{2}$
			\\\cline{1-2}
			\eqref{h1} with $1<\gamma< 2-\frac{1}{m}$, \eqref{h2bis} with $\theta> 0$
			&$L^m(\Omega)$ with $m>1$	
			\\\cline{1-2}
			$s^{-\gamma}$, $1<\gamma<3 -\frac{2}{m}$ 
			&$f>0$, $L^m(\Omega)$ with $m>1$ 	
			\\\cline{1-2}	
		\end{tabular}
		\vspace*{3mm}
		\caption{Assumptions to have at least a solution to \eqref{op2_regular} in $H^1_0(\Omega)$} 
	\end{table}
\end{center}

\subsection{Further remarks on the lower order term} 
\label{lowerorder}
We conclude this part with some comments on the regularity of the lower order term.
\medskip
We start with the following example.
\begin{example}\label{ex}
	Let  $u = (1-|x|^{2})^{\eta}$, $\eta>0$, and $\Omega=B_1(0)$. Then, if 
	\begin{equation}\label{al}\frac{1}{1+\gamma}<\eta < 1,\end{equation}  $u$ solves \eqref{pbbo}  with 
	$$
	f\sim \frac{1}{(1-|x|^{2})^{2-\eta -\eta\gamma}}\in L^{1}(\Omega). 
	$$
\end{example}
\noindent Some remarks are in order: first of all observe that, as $\eta<1$, then  $-\Delta u\notin L^{1}(\Omega)$ while $-\Delta u \in L^{1}(\Omega, d)$ for any $\eta>0$. 
\\ Moreover, using \eqref{al}, we have
\begin{itemize}
	\item [i)] if $\gamma=1$ the solution is always in $H^1_0(\Omega)$ as expected;
	\item [ii)] if $\gamma>1$ then $f$ is in $L^{m}(\Omega)$ provided 
	$$
	\eta >\frac{2-\frac{1}{m}}{\gamma+1},
	$$
	that is, $u\in H^1_0(\Omega)$ if $\gamma<3-\frac{2}{m}$. Also observe that for any $\gamma>1$ one can pick a datum $f\in L^m(\Omega)$, for suitable $m>1$,  for which the relative solution does not have finite energy;
	\item [iii)] if $\gamma<1$ then $u$ is always in $H^{1}_{0}(\Omega)$  and $f\in L^{m}(\Omega)$ for any
	$$
	m<\frac{1}{2-\eta-\eta\gamma}\,.
	$$
	We observe that 
	\begin{equation}\label{opt}
	\frac{1}{2-\eta-\eta\gamma}\nearrow\frac{1}{1-\gamma}\ \ \text{as $\eta\to1^{-}$}.
	\end{equation}
\end{itemize}

\medskip

The previous example shows that one cannot expect in general the lower order term $fu^{-\gamma}$ to be in $L^{1}(\Omega)$. The best one can expect, in general,  is the weighted summability one shall infer by  Lemma \ref{M1respectdelta} below. Nevertheless in the case $\gamma<1$ we have the following result which is optimal in view of \eqref{opt}. 

\begin{theorem}
	Let $A$ satisfy \eqref{lip}, $\gamma<1$, and $0\le f \in L^{m}(\Omega)$ with $m>\frac{1}{1-\gamma}$. Let $h$ be a function satisfying \eqref{h1} and \eqref{h2}. Then the solution $u$ to problem \eqref{op2_regular} is such that $h(u)f \in L^{1}(\Omega)$. 
	\label{gamma<1L1}
\end{theorem}
\begin{proof}
   Let $v_n$ be the solution to \eqref{auxsotto} as introduced in Theorem \ref{teogammamaggiore}; this means that for a.e.  $x\in\Omega$
	$$u(x)\ge u_n(x)\ge v_1(x)\ge Cd(x).$$ 
	Thus
	\begin{align*}
	\int_\Omega h(u) f &\le  c_1\int_{\{u< s_1\}}  fu^{-\gamma} + \sup_{s\in [s_1,\infty)} h(s) \int_{\{u \ge s_1\}}  f  
	\\
	&\le c_1\int_\Omega f v_1^{-\gamma} + \sup_{s\in [s_1,\infty)} h(s) \int_{\{u \ge s_1\}}  f  \leq C
	\end{align*}
	since 
	$$c_1\int_\Omega f v_1^{-\gamma} \le  C\left(\int_\Omega d^{-m'\gamma}\right)^\frac{1}{m'},$$
	that is finite if $m> \frac{1}{1-\gamma}$.
\end{proof}
\begin{remark}\label{rembordo}
	
	 The boundary of $\Omega$ is needed to be $C^{1,1}$. Indeed, this is the required regularity in order to apply the Calderon-Zygmund theory (see for instance  \cite[Theorem $3.17$]{tro}). The regularity on the boundary is also necessary in order to apply the classical Hopf Lemma; though for the Hopf Lemma we just need $\Omega$ to satisfy the interior ball condition in every point of its boundary (see Lemma \ref{hopf} below). 
\triang \end{remark}

\section{Renormalized solutions for nonlinear equations with measure data}
\label{sec:measure}

  So far we have focused on problems with Lebesgue data and a linear principal operator.   The aim of this section is to extend, whenever possible,  some of the previous results  in  various directions; in particular,  we shall examine the case of \underline{measure data}, as well as the presence of a \underline{nonlinear operator} in the second order leading term.

 We will address the homogeneous Dirichlet boundary value problem  related to 

\begin{equation}\label{gen}
 -\operatorname{div}(a(x,\nabla u)) = h(u)\mu  \ \  \text{in} \  \Omega\,,
\end{equation}
where $a(x,\cdot)$ is a Leray-Lions type coercive operator with $(p-1)$-growth  ($p>1$), $\mu$ is a nonnegative Radon measure,   and $h$ is a nonnegative continuous function on $[0,\infty)$.

Measure data problems in case  $h(s)=1$,   have been  a huge field of investigation for decades. Pioneering contributions  to this topic have been given  by G. Stampacchia (\cite{st}), H. Brezis (\cite{Br}), L. Boccardo and T. Gallou\"et (\cite{bg}).   The theory of existence and uniqueness for such problems was then   implemented and mostly fulfilled by cornerstone papers as \cite{dpl}, \cite{b6}, \cite{bgo},  \cite{bl}, \cite{blmu},  and \cite{dmop} with the introduction of truncation methods that, eventually, brought to the notion of entropy  and  renormalized solutions for these kinds of problems (see also Section \ref{renorm} below). 

Let us also mention that regularity of SOLA solution (i.e. Solutions Obtained as a Limit of Approximation) to these kind of problems has been also addressed (see \cite{dall}); for a  compendium on this topic see \cite{min} and references therein. 

\medskip

Coming back to our non-autonomous and  possibly \underline{singular} scenario, namely problems as in \eqref{gen}, we will  first show that, if the measure $\mu$ is too \underline{concentrated} (see Definition \ref{defconcentrated}), nonexistence of solutions is possible in case of the model nonlinearity (i.e. $h(s)=s^{-\gamma}$, $\gamma>0$); then we prove existence of a distributional solution for suitable measure data.  Finally we prove a very general existence result of  \underline{renormalized solutions} that, in turn,  will allow us to prove  uniqueness of solutions  provided $h$ is non-increasing. 

\medskip

For this part we mainly refer to \cite{bo,do,ddo,orpe}. We also underline that essential  tools concerning Radon measures and capacity are summarized in Appendix \ref{app:radon}.
\medskip

\subsection{Nonexistence for concentrated measure data}
\label{sec:nonex}

Here we try to catch a glimpse to the case of a semilinear problem in presence of a measure as datum in the involved  lower order term. A  first phenomenon to be  highlighted is that the  combination of a concentrated measure and a degenerate nonlinearity (at infinity) lead to a nonexistence result in the sense of approximating problems (aka SOLA). 

\medskip

Let us consider
\begin{equation}
\begin{cases}
\displaystyle - \Delta u= \frac{\mu}{u^\gamma} &  \text{in} \ \Omega, \\
u=0 & \text{on} \ \partial \Omega,
\label{pb:misuramodello}
\end{cases}
\end{equation}
where  $\mu\in \mathcal{M}(\Omega)$ is nonnegative.  We consider a  natural approximating sequence of problems,   i.e. let $u_n\in H^1_0(\Omega)\cap L^\infty(\Omega)$ be a nonnegative weak  solution to 
\begin{equation}
\begin{cases}
\displaystyle - \Delta u_n= \frac{\mu_n}{(u_n+\frac{1}{n})^\gamma} &  \text{in} \ \Omega, \\
u_n=0 & \text{on} \ \partial \Omega,
\label{pbn_ren}
\end{cases}
\end{equation}
where $\mu_n$ is a suitable approximation of $\mu$ as given in Lemma \ref{lem_approssimazione} below.

A typical insight  when dealing with measure data problems 
suggests that  {solutions do blow-up on the set where the measure is} {\it too concentrated}; forcing this argument one can guess that, as $s^{-\gamma}$ vanishes at infinity, then solutions $u_n$ to \eqref{pbn_ren} tends to zero beyond certain concentration threshold for $\mu$. 

\medskip

We state and prove the following result:
\begin{theorem}
	\label{bononesistenza}
	Let $ 0\le \mu \in \mathcal{M}(\Omega)$ be concentrated on a Borel set $E$ of zero $q$-capacity. Let $0\le \mu_n \in L^\infty(\Omega)$, bounded in $L^1(\Omega)$, converging to $\mu$ in the narrow topology of measures. Let $u_n$ be a solution to \eqref{pbn_ren} then  
	\begin{itemize}
		\item[i)] if $\gamma\leq 1$ and $q= \left(\frac{N(\gamma+1)}{N-1+\gamma}\right)'$ then $u_n$ weakly converges to $0$ in $W^{1,\frac{N(\gamma+1)}{N-1+\gamma}}_0(\Omega)$;
		\item[ii)] if $\gamma>1$ and $q=2$ then $u_n^{\frac{\gamma+1}{2}}$ weakly converges to $0$ in $H^1_0(\Omega)$.
	\end{itemize}
\end{theorem}
\begin{proof} 
	{\bf Proof of i).} 
	  First observe that, following the proof of Theorem \ref{boesistenza}, one has that $u_n$ is bounded in $W^{1,\frac{N(\gamma+1)}{N-1+\gamma}}_0(\Omega)$ with respect to $n$. Moreover Lemma \ref{dalmaso} gives the existence of $\Psi_\eta \in C^1_c(\Omega)$ such that 
	\begin{equation}\label{psieta}
	\displaystyle 0\le \Psi_\eta\le 1, \ \ \ 0\le \int_{\Omega} (1-\Psi_\eta) \mu \le \eta, \ \ \ \int_{\Omega} |\nabla \Psi_\eta|^{q} \le \eta.
	\end{equation}
	Hence we take $T_k(u_n)(1-\Psi_\eta)$ ($k>0$) as a test function in the weak formulation of \eqref{pbn_ren}, yielding to
	\begin{equation*}
		\displaystyle \int_{\Omega} |\nabla T_k(u_n)|^2(1-\Psi_\eta) - \int_{\Omega} \nabla u_n \cdot \nabla \Psi_\eta T_k(u_n) \le \int_{\Omega} \frac{\mu_n T_k(u_n)(1-\Psi_\eta)}{\left(u_n+\frac{1}{n}\right)^\gamma} \le k^{1-\gamma}\int_{\Omega} (1-\Psi_\eta)\mu_n.
	\end{equation*}		
	The estimate on $u_n$ and the properties of $\Psi_\eta$ allow us  to take the limit  firstly as  $n\to\infty$ and then to let $\eta\to 0^+$ in order to deduce
	\begin{equation*}
		\displaystyle \int_{\Omega} |\nabla T_k(u)|^2 \le 0, 
	\end{equation*}
	for any $k>0$, implying that $u=0$ almost everywhere in $\Omega$. \\
	{\bf Proof of ii).} Following once again the proof of Theorem \ref{boesistenza}, one has that $u_n^{\frac{\gamma+1}{2}}$ is bounded in $H^1_0(\Omega)$ and $u_n$ is locally bounded in $H^1(\Omega)$ with respect to $n$.
	In this case, we take $T_k^\gamma(u_n)(1-\Psi_\eta)$ as a test function in the weak formulation of \eqref{pbn_ren} where $\Psi_\eta$ is as in \eqref{psieta}. Then one deduces  
		\begin{equation*}
		\frac{4\gamma}{(\gamma+1)^2}\displaystyle \int_{\Omega} |\nabla T_k^{\frac{\gamma+1}{2}}(u_n)|^2(1-\Psi_\eta) - \int_{\Omega} \nabla u_n \cdot \nabla \Psi_\eta T_k^\gamma(u_n) \le \int_{\Omega} \frac{\mu_n T_k^\gamma(u_n)(1-\Psi_\eta)}{\left(u_n+\frac{1}{n}\right)^\gamma} \le \int_{\Omega} (1-\Psi_\eta)\mu_n.
	\end{equation*}		
	The estimates on $u_n$ permit to reason as for the first case, allowing to take $n\to\infty$ and $\eta\to 0^+$, in order to deduce 	
	\begin{equation*}
	\displaystyle  \int_{\Omega} |\nabla T_k^{\frac{\gamma+1}{2}}(u)|^2 \le 0,
	\end{equation*}
	for any $k>0$. This concludes the proof. 
\end{proof}
\begin{remark}
	Observe that if $\gamma=1$ then $u_n$ weakly converges to $0$ in $H^1_0 (\Omega)$.  It is also  worth to point out that, reasoning as above, if $\mu_n$ is as in Theorem \ref{bononesistenza} and if we consider problem \eqref{pbn_ren} with right-hand $(g_n+\mu_n)\left(u_n+\frac{1}{n}\right)^{-\gamma}$ with $g_n$ converging to $g$ in $L^1(\Omega)$ then $u_n$ converges to $u$ solution of  problem \eqref{pb:misuramodello}  with right-hand $gu^{-\gamma}$. Thus,  the Lebesgue part overcomes the approximation while too concentrated measures do not in general.
\triang \end{remark}

\subsection{Existence for diffuse measure.  An approach by  monotone approximation}

 The above arguments suggest   that, working by approximation,  measure data to be handle with need not to be purely singular at least for  degenerate nonlinearities $h$'s.  Therefore a major question relies on  the possibility of coupling sufficiently smooth measures with degenerate nonlinearities.

\medskip 

We also stress that  a re-adaptation of the proof of  Theorem \ref{boesistenza} is not straightforward for a general measure as datum since,  in general, it may  not be approximated monotonically. 

\medskip

Here we consider problem \eqref{pb:misuramodello} in case of a measure datum $\mu$ which is \underline{diffuse} (see Definition \ref{defdiffuse} below); we will take advantage that nonnegative diffuse measures, among other relevant features, can be approximated increasingly.

\medskip

We state and prove the following:

\begin{theorem}
	\label{orpeesistenza2}
	Let  $0\le \mu\in \mathcal{M}(\Omega)$ be diffuse with respect to the $2$-capacity. Then there exists a distributional solution $u$ to  \eqref{pb:misuramodello}.
\end{theorem} 
\begin{proof}
	 We consider the following approximation scheme 
	\begin{equation*}
	\begin{cases}
	\displaystyle - \Delta u_n= \frac{\mu_n}{(u_n+\frac{1}{n})^\gamma} &  \text{in} \ \Omega, \\
	u_n=0 & \text{on} \ \partial \Omega,
	\label{pborpe}
	\end{cases}
	\end{equation*}
	where $\mu_n\in H^{-1}(\Omega)$ is an increasing sequence converging to $\mu$ strongly in $\mathcal{M}(\Omega)$ (see Proposition \ref{approssimazionediffusecrescente}). The existence of such nonnegative $u_n \in H^1_{0}(\Omega)$ follows  as in  \cite{mupo}. 
	
	As already done in the proof of Theorem \ref{boesistenza}, one can take  $(u_n-u_{n+1})^+$ as a test function in the difference of formulations solved by $u_n$ and $u_{n+1}$ allowing to deduce that $u_n$ is    non-decreasing with respect to  $n$.
	Moreover, since $\frac{\mu_1}{(u_1+1)^\gamma}$ is not null, one obtains from the strong maximum principle that  
	\begin{equation}\label{stimaorpe1}
	\forall\omega\subset\subset\Omega \ \ \exists c_\omega>0 : \ \ \ u_n \ge u_1 \ge  c_\omega >0 \ \text{in} \ \omega.
	\end{equation}
	Now let observe that, concerning the a priori estimates, one can strictly follow the arguments of Theorem's \ref{boesistenza} proof. 	
	
	Now let first $\gamma = 1$ and take $u_n$ as a test function in the weak formulation of \eqref{pborpe} then one has
	$$\int_\Omega |\nabla u_n|^2 = \int_\Omega \frac{u_n}{u_n+\frac{1}{n}} \mu_n \le |\mu|(\Omega),$$
	which implies that $u_n$ is bounded in $H^1_0(\Omega)$ with respect to $n$.
	In case $\gamma>1$ one can take  $T_k^\gamma(u_n) \in H^1_0(\Omega)$ as a test function which implies that 
	$$\frac{4\gamma}{(\gamma+1)^2}\int_\Omega |\nabla T_k^{\frac{\gamma+1}{2}}(u_n)|^2 \le |\mu|(\Omega),$$
	and, taking $k\to \infty$, one has that $u_n^\frac{\gamma+1}{2}$is bounded in $H^1_0(\Omega)$ with respect to $n$. Moreover, by $|\nabla u_n^\frac{\gamma+1}{2}|=  |\nabla u_n|u_n^\frac{\gamma-1}{2}$ and  by \eqref{stimaorpe1} one deduces that $u_n$ is also locally bounded in $H^1(\Omega)$.
	When $\gamma<1$ one can take $(u_n+\varepsilon)^\gamma -\varepsilon^\gamma \in H^1_0(\Omega)$ ($0<\varepsilon<\frac{1}{n}$) as a test function in the weak formulation of \eqref{pborpe}. This allows to reason exactly as in the proof of Theorem \ref{boesistenza} in order to deduce that $u_n$ is bounded in 
	$W^{1,\frac{Nm(\gamma+1)}{N-m(1-\gamma)}}_0(\Omega)$. 
	
	\medskip
	
	Now we have to pass to the limit with respect to $n$ in the weak formulation of \eqref{pborpe}. Clearly, we can pass to the limit by weak convergence in the term of the left-hand. We re-write the right-hand as  
	\begin{equation}\label{stimaorpe2}
		\int_\Omega \frac{\varphi}{(u_n+\frac{1}{n})^\gamma}  \mu_n = 	\int_\Omega \frac{\varphi}{(u_n+\frac{1}{n})^\gamma} ( \mu_n -  \mu) + \int_\Omega \frac{\varphi}{(u_n+\frac{1}{n})^\gamma}  \mu.  
	\end{equation}	
	Recalling  that $\varphi \in C^1_c(\Omega)$, and using \eqref{stimaorpe1} with $\omega={\rm supp}\varphi$,  one has  	$$\lim_{n\to\infty}\int_\Omega \frac{\varphi}{(u_n+\frac{1}{n})^\gamma} (\mu_n - \mu) \le \lim_{n\to\infty} \frac{||\varphi||_{L^\infty(\Omega)}}{c^{\gamma}_{\supp \varphi}} |\mu_n-\mu|(\Omega) = 0.$$
	Let us focus on the second term in \eqref{stimaorpe2}. As  $0\le \varphi\in C^1_c(\Omega)$ we  take $-\frac{\varphi}{\left(u_n+\frac{1}{n}\right)^{2\gamma+1}}$ as a test function in \eqref{pborpe} deducing that
	\begin{equation*}
		(2\gamma +1)\int_\Omega \frac{|\nabla u_n|^2\varphi}{\left(u_n+\frac{1}{n}\right)^{2\gamma+2}} - \int_\Omega  \frac{\nabla u_n \cdot\nabla \varphi}{\left(u_n+\frac{1}{n}\right)^{2\gamma+1}} \le 0, 
	\end{equation*}
	that  implies 	\begin{equation}\label{stimaorpe}
	(2\gamma +1)\int_\Omega \frac{|\nabla 	u_n|^2\varphi}{\left(u_n+\frac{1}{n}\right)^{2\gamma+2}} \le  \int_\Omega  \frac{|\nabla u_n||\nabla \varphi|}{\left(u_n+\frac{1}{n}\right)^{2\gamma+1}} \le \frac{1}{c^{2\gamma + 1}_{\supp \varphi}} \int_\Omega |\nabla u_n||\nabla \varphi| \le C, 
	\end{equation}	
	for some constant which does not depend on $n$. Since 
		$$\nabla \left(\frac{\varphi}{(u_n+\frac{1}{n})^\gamma} \right) =  \frac{\nabla\varphi}{(u_n+\frac{1}{n})^\gamma} - \frac{\gamma\nabla u_n\varphi}{(u_n+\frac{1}{n})^{\gamma+1}}, $$
	then estimate \eqref{stimaorpe} implies that $\frac{\varphi}{(u_n+\frac{1}{n})^\gamma}$ is bounded both in $H^1_0(\Omega)$ and in  $L^\infty(\Omega)$. Then it weakly converges in both space and this is sufficient to take $n\to\infty$ in the second term on the right-hand of \eqref{stimaorpe2} since, from Theorem \ref{diffuse}, $\mu$ can be decomposed as $H^{-1}(\Omega)+L^1(\Omega)$. This concludes the proof.
	
\end{proof}

\subsection{The method  of renormalized solutions}

\label{renorm} 

As we have already observed in Theorem \ref{boca}, there is at most one distributional solution $u\in H^1_0(\Omega)$ to \eqref{pbgeneralh} when $h$ is non-increasing. By the way, along Section \ref{sec:reg},  and in particular in Example \ref{exnoth1}, we encountered explicit examples of  infinite energy solutions. 
To extend the existence and uniqueness results from the previous sections to a more general nonlinear setting, we apply  renormalization techniques.

\medskip

 For the reader's convenience, we briefly want to sum up  the renormalized approach to solve PDEs  boundary value problems. The main idea was introduced in \cite{dpl} for the study of Boltzmann equation to handle with infinite energy solutions. This technique was then adapted to the case of nonlinear elliptic equations with unbounded convection terms (\cite{bdgm1,bdgm2}) and, later, in order to study   uniqueness of solutions for nonlinear elliptic equations with measure  data (\cite{mu,dmop}).

\medskip 

Now for the sake of the exposition,  let us focus on a linear elliptic equation; in this case the theory of renormalized solutions turns out to coincide with the one of duality solutions developed by G. Stampacchia \cite{st}.

As we are concerned with uniqueness of solutions, a crucial motivation to deal with the renormalization method comes from the
celebrated paper \cite{serrin} (see also \cite{prignet}); here the author finds a nontrivial  distributional solution of  problem
\begin{equation*}\label{serrine}
\begin{cases}
	\displaystyle -\operatorname{div}(A^\varepsilon(x) \nabla u) = 0 & \text{in}\ \Omega,\\
	u=0 &\text{on}\ \partial\Omega,
	\end{cases}
\end{equation*}
  where $A^\varepsilon(x)$ is defined by 
$$a^\varepsilon_{ij} = \delta_{ij} + \left(\frac{N-1}{\varepsilon(N-2+\varepsilon)}-1\right)\frac{x_ix_j}{|x|^2}.$$

 Note that, by linearity, this example implies a strong non-uniqueness result for distributional solutions to the Dirichlet problem associated to 
$$	\displaystyle -\operatorname{div}(A^\varepsilon(x) \nabla u) = f \ \   \text{in}\ \Omega\,,$$
for any reasonable datum $f$ (e.g. $f\in L^1(\Omega)$). Let us stress that in this latter case SOLA belong to $W^{1,q}_0(\Omega)$ for any $q<\frac{N}{N-1}$,  while the \underline{pathological} ones found by Serrin are less regular, i.e. they belong to $W^{1,q}_0(\Omega)$ for any $q<\frac{N}{N-1+\varepsilon}$. One may then wonder if  uniqueness holds in the class of functions belonging to $W^{1,q}_0(\Omega)$ for every $q<\frac{N}{N-1}$   but this is not the case in $N\geq 3$ as shown in \cite{prignet}. Moreover, if $N=2$, uniqueness holds in this class due to Meyers summability  Theorem.  

\medskip 

A key observation is, in fact,   that  truncations $T_k(u)$ of the right solutions should have, in general, finite energy.  
This regularity property is quite natural;  if $u$ is a solution of the homogeneous Dirichlet problem associated to  
\begin{equation}\label{frameren}
 -\operatorname{div}(A(x) \nabla u) = f
\end{equation}
where $A$ satisfies \eqref{Abounded} and $f\in L^1(\Omega)$, formally multiply by $T_k(u)$ to obtain 
\begin{equation*}\label{tronch1}
\alpha \int_{\Omega} |\nabla T_k(u)|^2 = \int_\Omega fT_k(u) \le k||f||_{L^1(\Omega)}.
\end{equation*}
Let us remark that the previous estimate which is only formal can be make rigorous at an approximation level. 
Ultimately  any solution $u$ to the homogeneous Dirichlet problem associated to  \eqref{frameren} obtained by approximation enjoys  $$T_k(u)\in H^1_0(\Omega)\ \  \text{for any}\ \ \  k>0\,.$$  
A natural question is then if  a  right notion of solution in order to expect uniqueness is the fact that  $T_k(u)\in H^1_0(\Omega)$ for any  $k>0$. This is, in general, an open problem. Although, the idea to look at the equation solved by $T_k(u)$ is in some sense the core of the renormalization idea which, as we already said,  was firstly introduced in \cite{dpl} in the context  of Boltzmann equations (where logarithmic type truncations appeared) and then adapted to the elliptic setting with irregular  data in \cite{LiMu, mu} and, later, in \cite{dmop}.

\medskip

To be more concrete let us   multiply \eqref{frameren} by $S(u)$ where $S$ is a Lipschitz function on $\mathbb{R}$ with compact support, one has 
 \begin{equation}\label{frameren2}
 -\operatorname{div}(A(x) S(u)\nabla u)  + A(x)\nabla u\cdot \nabla u S'(u) = f S(u).
 \end{equation}
 If $T_k(u) \in H^1_0(\Omega)$ then each term in \eqref{frameren2} is well defined;  uniqueness of renormalized solutions is then related to the fact that, roughly speaking, one can extend  the set of test functions in \eqref{frameren2} to be in $H^1_0(\Omega)\cap L^\infty(\Omega)$ and  one is allowed to take $T_k(u)$ as a test function and, hopefully (and this will be the case), getting a comparison between solutions.

\medskip

Of course one also has to provide information on the zone where the solution is large which is missed in the truncated equation. To better understand that, let us fix $S=V_m$ in \eqref{frameren2} ($V_m$ is defined in \eqref{not:Vdelta}), so that 
 \begin{equation}\label{frameren3}
-\operatorname{div}(A(x) V_m(u)\nabla u)  + \frac{1}{m}A(x)\nabla u\cdot \nabla u \chi_{\{m<u<2m\}} = f V_m(u).
\end{equation}
If one requires that distributionally 
\begin{equation}\label{frameren4}
	\lim_{m\to\infty} \frac{1}{m}A(x)\nabla u\cdot \nabla u \chi_{\{m<u<2m\}} = 0, 
\end{equation}
then by taking $m\to\infty$ in  \eqref{frameren3} one obtains that  \eqref{frameren} holds in the sense of distributions. Conditions as in \eqref{frameren4} are typically  the right ones in order  to extend the notion of distributional solution and to obtain uniqueness also in the case of low regularity of the data. This condition essentially expresses the fact that,   even if we are not capable to assure that the function has finite energy,  the energy of $u$ on their superlevels does not  grow too fast.

Therefore, we can summarize as follows: a  \underline{renormalized solution} to the Dirichlet problem associated to \eqref{frameren} is a function $u$ satisfying $T_k(u) \in H^1_0(\Omega)$ for any $k>0$, \eqref{frameren2} and \eqref{frameren4}. 

\medskip

Let us finally recall that, if $f$ is diffuse with respect to the $2$-capacity, then the notion of renormalized solution coincides the one of entropy solution introduced in \cite{b6}.

\medskip
In the next section we will extend it to the case of a nonlinear operator of $p$-Laplace type, measure data, and generic, possibly singular, zero order terms.

\subsection{Existence and uniqueness of solutions in case of measure data and nonlinear principal operator}
\label{sec:measureex}

We deal with the following general problem: 

\begin{equation}
\begin{cases}
\displaystyle -\operatorname{div}(a(x,\nabla u)) = h(u)\mu &  \text{in} \  \Omega, \\
u=0 & \text{on}\ \partial \Omega,
\label{pbmain}
\end{cases}
\end{equation}
where $\displaystyle{a(x,\xi):\Omega\times\mathbb{R}^{N} \to \mathbb{R}^{N}}$ is a Carath\'eodory function satisfying the classical Leray-Lions structure conditions for $1<p<N$, namely
\begin{align}
&a(x,\xi)\cdot\xi\ge \alpha|\xi|^{p}, \ \ \ \alpha>0,
\label{cara1}\\
&|a(x,\xi)|\le \beta|\xi|^{p-1}, \ \ \ \beta>0,
\label{cara2}\\
&(a(x,\xi) - a(x,\xi^{'} )) \cdot (\xi -\xi^{'}) > 0,
\label{cara3}	
\end{align}
for every $\xi\neq\xi^{'}$ in $\mathbb{R}^N$ and for almost every $x$ in $\Omega$. The datum $0\le \mu \in \mathcal{M}(\Omega)$ and it is worth to point out (see Theorem \ref{decmeas} below) that $\mu$ admits an unique decomposition into $\mu_d+\mu_c$, where $\mu_d$ is a diffuse measure with respect to the $p$-capacity and $\mu_c$ is a measure concentrated on a set of zero $p$-capacity. Let us also assume that 
\begin{equation}\label{hmu}
	\mu_d\not\equiv0
\end{equation}
in order to avoid degeneracy phenomena of the approximation solutions as in the previous section.

The function $h:[0,\infty)\to (0,\infty]$ is continuous and finite outside the origin satisfying \eqref{h1} and \eqref{h2} which we recall here for the sake of completeness:
\begin{equation}
\exists\;\gamma \ge 0, \ c_1,s_1>0: \  h(s) \le \frac{c_1}{s^\gamma}\;\text{ if }\;s<s_1,
\label{h1ren}
\end{equation}
and 
\begin{equation}
\displaystyle \limsup_{s\to \infty} \ h(s)=:h(\infty)<\infty.
\label{h2ren}
\end{equation}
Let us observe that the strict positivity of $h$ is a technical assumption in order to handle with the concentrated part of the measure. In Section \ref{H=0}, as we will see, the case when $h$ degenerates  is simpler and bounded solutions are shown to exist even in presence of rough data. Let us explicitly mention that, for the entire Section \ref{sec:measureex} we require $\gamma\le 1$. Indeed, in this case, we exploit that $T_k(u)\in W^{1,p}(\Omega)$ for any $k>0$ which is something unexpected, in general, in case $\gamma>1$. 

In Section \ref{sec:distrmagg1} below, we state the existence of a distributional solution if $\gamma>1$ and we outline the proof's main steps.  

\medskip

 Firstly we precisely set the notion of \underline{renormalized solution} for our case:
\begin{defin}\label{renormalized}
	Let $a$ satisfy \eqref{cara2} then a nonnegative function $u$, which is almost everywhere finite on $\Omega$, is a {\it renormalized solution} to problem \eqref{pbmain} if $T_k(u) \in W^{1,p}_0(\Omega)$ for every $k>0$ and if 
	\begin{equation}\label{ren0}
		h(u)S(u)\varphi \in L^1(\Omega,\mu_d)
	\end{equation}
	and
	\begin{gather}
	\int_{\Omega} a(x,\nabla u)\cdot\nabla \varphi S(u) + \int_{\Omega}a(x,\nabla u)\cdot\nabla u S'(u)\varphi = \int_{\Omega}h(u)S(u)\varphi \mu_d \label{ren1}\\
	\forall S \in W^{1,\infty}(\mathbb{R})\;\text{ with compact support and}\;\forall\varphi \in W^{1,p}_0(\Omega)\cap L^\infty(\Omega),\notag
	\end{gather}
	\begin{equation}
	\label{ren2}
	\lim_{m\to \infty} \frac{1}{m}\int_{\{m< u< 2m\}} a(x,\nabla u)\cdot \nabla u \varphi = h(\infty)\int_{\Omega}\varphi \mu_c\quad\forall \varphi \in C_b(\Omega).
	\end{equation}
\end{defin}
  Let us now precise what we mean by distributional solution to problem \eqref{pbmain} in this more general situation. Let us also recall that $\sigma:=\max\left(1, \gamma\right)$ which was already set in \eqref{sigma}. 
\begin{defin}\label{distributional}
	A nonnegative and measurable function $u$ such that $|a(x,\nabla u)| \in L^1_{\textrm{loc}}(\Omega)$ is a {\it distributional solution} to problem \eqref{pbmain} if $h(u) \in L^1_{\textrm{loc}}(\Omega,\mu_d)$, and the following hold
	\begin{equation}\label{troncate}
	T_k^\frac{\sigma-1+p}{p}(u) \in W^{1,p}_0(\Omega)\quad\forall k>0,
	\end{equation}
	and
	\begin{equation} \displaystyle \int_{\Omega}a(x,\nabla u) \cdot \nabla \varphi =\int_{\Omega} h(u)\varphi \mu_d + h(\infty)\int_{\Omega} \varphi \mu_c \ \ \ \forall \varphi \in C^1_c(\Omega).\label{distrdef}\end{equation}
\end{defin}
The notion of renormalized solution is much more general than the distributional one. It holds the following result in case $\gamma\leq1$:
\begin{lemma}
	\label{equivrindis}
	  Let $a$ satisfy \eqref{cara1} and \eqref{cara2},  let $h$ satisfy \eqref{h1ren} with $\gamma\leq1$ and \eqref{h2ren},  and let $0\le\mu\in \mathcal{M}(\Omega)$. Then a renormalized solution to \eqref{pbmain} is also a distributional solution to \eqref{pbmain}. 
\end{lemma}
\begin{proof}
	Let $u$ be a renormalized solution to \eqref{pbmain}.
	Obviously \eqref{troncate} holds. Let us take in \eqref{ren1} $S=V_m \ (m>0)$, where $V_m$ is defined in \eqref{not:Vdelta}, and $\varphi=T_k(u)$ ($s_1<k<m$); hence we have
	\begin{equation*}
	\int_\Omega a(x,\nabla u)\cdot \nabla T_k(u) V_m(u)\leq\frac{k}{m}\int_{\{m<u<2m\}} a(x,\nabla u)\cdot \nabla u + \int_\Omega h(u)T_k(u) V_m(u)\mu_d.
	\end{equation*}
	Using \eqref{cara1} and \eqref{h1ren}, it holds
	\begin{align*}
	\alpha\int_\Omega |\nabla T_k(u)|^p &\leq \frac{k}{m}\int_{\{m< u< 2m\}} a(x,\nabla u)\cdot \nabla u +\int_{\{u<s_1\}}h(u)T_k(u)V_m(u)\mu_d
	\\
	&+\int_{\{u\geq s_1\}} h(u)T_k(u)V_m(u)\mu_d \leq \frac{k}{m}\int_{\{m<u<2m\}} a(x,\nabla u)\cdot \nabla u
	\\
	&+c_1s_1^{1-\gamma}|\mu_d|(\Omega) +k \sup_{s\in [s_1,\infty)}h(s) |\mu_d|(\Omega).
	\end{align*}
	From the previous, as $m\to\infty$, one has 
	\begin{equation}
	\label{rendis1}
	\int_\Omega |\nabla T_k(u)|^p\leq C(k+1), \quad \forall k>0,
	\end{equation}
	for some positive constant $C>0$.

	By \eqref{rendis1} and by Lemma \ref{dalmaso2}, one deduces that $u$ is cap$_p$-almost everywhere finite and cap$_p$-quasi continuous; moreover, using Lemmas $4.1$ and $4.2$ of \cite{b6}, one has that $\displaystyle u,|\nabla u|^{p-1}\in L^1(\Omega)$. Now taking $\varphi\in C_c^1(\Omega)$ and $S=V_m$ in \eqref{ren1} one gets
	\begin{equation}
	\label{rendis2}
	\int_\Omega a(x,\nabla u)\cdot \nabla\varphi V_m(u)=\frac{1}{m}\int_{\{m<u<2m\}} a(x,\nabla u)\cdot \nabla u\varphi + \int_\Omega h(u)\varphi V_m(u)\mu_d.
	\end{equation}
	By \eqref{ren0} it results $h(u)V_1(u)\varphi\in L^1{(\Omega,\mu_d)}$, and so, using Lemma \ref{dalmaso3}, one has
	\begin{align*}
	&\int_\Omega h(u)|\varphi| \mu_d=\int_{\{u<1\}}h(u)|\varphi| \mu_d+\int_{\{u\geq 1\}}h(u)|\varphi|\mu_d \\
	&\leq \int_\Omega h(u) V_1(u)|\varphi|\mu_d+\sup_{s\in [1,\infty)}h(s) ||\varphi||_{L^{\infty}(\Omega)}|\mu_d|(\Omega)<\infty,
	\end{align*}
	that implies $h(u)\in L^1_{\textrm{loc}}(\Omega,\mu_d)$. 
	\\
	This fact and \eqref{ren2} allow to take $m\to\infty$ in the right-hand of \eqref{rendis2} using also the Lebesgue Theorem for the second term. The Lebesgue Theorem also applies, as $m\to \infty$, for the left-hand observing that $a(x,\nabla u)\in L^1(\Omega)$ thanks to \eqref{cara2} since $|\nabla u|^{p-1}\in L^1(\Omega)$. This shows that $u$ is a distributional solution to \eqref{pbmain}.
\end{proof}
 
Now we are ready to state the existence of the renormalized solution to \eqref{pbmain}.
\begin{theorem}\label{teoexrinuniqueness}
	Let $a$ satisfy \eqref{cara1}-\eqref{cara3}, let $0\le \mu\in \mathcal{M}(\Omega)$ satisfy \eqref{hmu} and let $h$ satisfy \eqref{h1ren} and \eqref{h2ren} with $\gamma\le 1$. Then there exists a renormalized solution $u$ to problem \eqref{pbmain} such that:
	\begin{itemize}
		\item [i)] if $1<p\leq2-\frac{1}{N}$ then $u^{p-1}\in L^q(\Omega)$ \ $\forall\, q<\frac{N}{N-p}$ and $\quad|\nabla u|^{p-1} \in L^q(\Omega)$ $\forall\, q<\frac{N}{N-1}$;
		\item [ii)] if $p>2-\frac{1}{N}$ then $u \in W_0^{1,q}(\Omega) \ \;\forall\,q<\frac{N(p-1)}{N-1}$.
	\end{itemize}
\end{theorem}
Under a natural monotonicity request on $h$, the renormalized solution is unique as shown from the following result.

\begin{theorem}\label{teounique}
	Let $a$ satisfy \eqref{cara1}-\eqref{cara3}, and let $\mu\in \mathcal{M}(\Omega)$ such that $\mu_c\equiv 0$. Let $h$ be a non-increasing function satisfying \eqref{h1ren} with $\gamma\le 1$ then there is at most one renormalized solution to \eqref{pbmain}.
\end{theorem}

\begin{remark}\label{linda}
	Let us once again highlight that we avoid to treat the case $\mu_d\equiv0$ to elude nonexistence results (in the approximation sense) similar to the ones given in the previous section.
	\\Let also observe that, in case $\mu_d\equiv0$, our notions of solution formally led us to deal with
	$$\begin{cases}
	\displaystyle -\operatorname{div}(a(x,\nabla u)) = h(\infty)\mu_c &  \text{in} \ \Omega, \\
	u=0 & \text{on}\ \partial \Omega,
	\end{cases}$$
	which could be studied using classical tools.
\triang \end{remark}

\subsubsection{\underline{Approximation scheme and a priori estimates}}
\label{sec:maggiore0}

We work through a two step approximation scheme keeping apart the truncation on $h$ and the regularization of $\mu$. 

\medskip
 
Namely we deal with: 
\begin{equation}\begin{cases}
\displaystyle -\operatorname{div}(a(x,\nabla u_{n,m})) = h_n(u_{n,m})(\mu_d + \mu_{m}) &  \text{in}\ \Omega, \\
u_{n,m}=0 & \text{on}\ \partial \Omega,
\label{pbapproxeps}
\end{cases}\end{equation}
where $h_n (s)=T_n(h(s))$ and $\mu_{m}$ is a sequence of nonnegative functions in $L^{p'}(\Omega)$, bounded in $L^1(\Omega)$, that converges to $\mu_c$ in the narrow topology of measures. 
\\
It follows from \cite{mupo,ddo} that there exists a renormalized solution $u_{n,m}$ to \eqref{pbapproxeps}. Moreover it holds:
\begin{itemize}
	\item [-] if $1<p\leq2-\frac{1}{N}$ then $u_{n,m}^{p-1}\in L^q(\Omega)$ for every $q<\frac{N}{N-p}$ and $\quad|\nabla u_{n,m}|^{p-1} \in L^q(\Omega)$ for every $q<\frac{N}{N-1}$;
	\item [-] if $p>2-\frac{1}{N}$ then $u_{n,m} \in W_0^{1,q}(\Omega)$ for every $q<\frac{N(p-1)}{N-1}$.
\end{itemize}
Let also observe that, from Lemma \ref{equivrindis}, $u_{n,m}$ is also a distributional solution to \eqref{pbapproxeps}.
\\For the sake of simplicity, since for deriving estimates is not necessary to distinguish between $n$ and $m$, we deal with following approximation in place of \eqref{pbapproxeps} until the passage to the limit
\begin{equation}\begin{cases}
\displaystyle -\operatorname{div}(a(x,\nabla u_{n})) = h_n(u_{n})(\mu_d + \mu_n) &  \text{in}\ \Omega, \\
u_{n}=0 & \text{on}\ \partial \Omega.
\label{pbapproxDDO}
\end{cases}
\end{equation}

Let us show a priori estimates on $u_n$.

\begin{lemma}\label{lemmastimehsing}
	Let $a$ satisfy \eqref{cara1}-\eqref{cara3}, let $0\le \mu\in \mathcal{M}(\Omega)$ satisfy \eqref{hmu},  and let $h$ satisfy \eqref{h1ren} with $\gamma\le 1$ and \eqref{h2ren}.  Moreover,  let $u_n$ be a solution to \eqref{pbapproxDDO}. Then $T_k(u_n)$ is bounded in $W^{1,p}_{0}(\Omega)$ with respect to $n$ for any $k>0$. It also holds that:
	\begin{itemize}
		\item[i)] if $p>2-\frac{1}{N}$, $u_n$ is bounded in $W^{1,q}_{0}(\Omega)$ for every $q<\frac{N(p-1)}{N-1}$;
		\item[ii)] if $1<p\le 2-\frac{1}{N}$, $u_n^{p-1}$ is bounded in $L^q(\Omega)$ for every $q<\frac{N}{N-p}$ and $|\nabla u_n|^{p-1}$ is bounded in $L^q(\Omega)$ for every $q<\frac{N}{N-1}$.
	\end{itemize}
	 Moreover,  there exists a measurable function $u$ to which  $u_n$ converges, up to a subsequence,   $\mu_d$-a.e.  in $\Omega$.  Finally, $u$ is cap$_p$-almost everywhere finite and cap$_p$-quasi continuous.
\end{lemma}

\begin{proof}
	Let us take $S=V_r$ ($V_r$ defined by \eqref{not:Vdelta}) and $\varphi=T_k(u_n)$ ($r>k$) in the renormalized formulation of \eqref{pbapproxDDO}. Using \eqref{ren2} (here $\mu_c\equiv 0$) and  letting $r\to\infty$,   one yields to  
	\begin{equation}\label{stimatk}
		\begin{aligned}
			\alpha \displaystyle \int_{\Omega} |\nabla T_k(u_{n})|^p &\le c_1s_1^{1-\gamma}\int_{\{u_n<s_1\}}(\mu_d + \mu_n)+k \sup_{s\in[s_1,\infty)}h(s) \int_{\{u_n\ge s_1\}}(\mu_d+\mu_n)\\ 
			&\le C(k+1).
		\end{aligned}
	\end{equation}
	An analogous reasoning with $\varphi= T_k(G_1(u_n))$ as a test function in the renormalized formulation of  \eqref{pbapproxDDO} gives that  
	\begin{equation}\label{tkg1}
	\io |\nabla T_k(G_1(u_n))|^p\leq C k \quad\forall k>0.
	\end{equation}
	Then an application of Lemmas $4.1$ and $4.2$ of \cite{b6} give that, if $p>2-\frac{1}{N}$,  $G_1({u_n})$ is bounded in $W_{0}^{1,q}(\Omega)$ for every $q<\frac{N(p-1)}{N-1}$. This proves i) since $T_1(u_n)$ is bounded in $W^{1,p}(\Omega)$ with respect to $n$ thanks to \eqref{stimatk}. Then there exists a nonnegative function $u$ belonging to $W_{\rm{loc}}^{1,q}(\Omega)$ for every $q<\frac{N(p-1)}{N-1}$ such that $u_n$ converges to $u$ almost everywhere in $\Omega$. 
	\\On the other hand, if $1<p\leq 2-\frac{1}{N}$, Lemmas $4.1$ and $4.2$ of \cite{b6} give that $u_n$ is bounded in $M^{\frac{N(p-1)}{N-p}}(\Omega)$ and that $|\nabla u_n|$ is bounded in $M^{\frac{N(p-1)}{N-1}}(\Omega)$. In particular $u_n^{p-1}$ is bounded in $L^q(\Omega)$ for every $q<\frac{N}{N-p}$ and $|\nabla u_n|^{p-1}$ is bounded in $L^q(\Omega)$ for every $q<\frac{N}{N-1}$. Now observe that $T_k(u_n)$ is a Cauchy sequence in $L^p(\Omega)$ for all $k>0$, so that, up to subsequences, it is a Cauchy sequence in measure for each $k>0$. Then one gets that
\begin{equation}\label{divido}
\{|u_n-u_m|>l\}\subseteq\{u_n\geq k\}\cup\{u_m\geq k\}\cup\{|T_k(u_n)-T_k(u_m)|>l\}.
\end{equation}

Now, if $\varepsilon>0$ is fixed, the estimates on $u_n$ imply that there exists a $\overline{k}>0$ such that 
$$\left|\{u_n> k\}\right|<\frac{\varepsilon}{3},\;\;\left|\{u_m> k\}\right|<\frac{\varepsilon}{3}\;\forall n,m\in\mathbb{N},\;\forall k>\overline{k},$$
while, using that $T_{k}(u_{n})$ is a Cauchy sequence in measure for each $k>0$, one has that there exists $\eta_{\varepsilon}>0$ such that
$$\left|\{|T_k(u_n)-T_k(u_m)\right|>l\}|<\frac{\varepsilon}{3}\;\forall n,m>\eta_{\varepsilon},\;\forall l>0.$$
Thus, if $k>\overline{k}$, from \eqref{divido} we obtain that 
$$\left|\{|u_n-u_m|>l\}\right|<\varepsilon\quad\forall n,m\geq\eta_{\varepsilon},\;\forall l>0,$$ 
and so that $u_n$ is a Cauchy sequence in measure. This means that there exists a nonnegative measurable function $u$ to which $u_n$ converges almost everywhere in $\Omega$. Moreover the fact that $u_n^{p-1}$ is  bounded in $L^q(\Omega)$ for every $q<\frac{N}{N-p}$, imply that $u$ is almost everywhere finite.

\medskip 

Now observe that, for any $p>1$,  the weak lower semicontinuity as $n\to\infty$ in \eqref{tkg1}  implies
\begin{equation*}\label{Tk}
\io |\nabla T_k(G_1(u))|^p\leq Ck \quad\forall k>0.
\end{equation*}
Then, recalling Definition \ref{a8},  one has  
$$\text{cap}_p(\{u\ge k+1\}) \le \int_\Omega \frac{|\nabla T_k(G_1(u))|^p}{k^p} \le \frac{C}{k^{p-1}},$$
and  taking $k\to\infty$, one obtains that $u$ is cap$_p$-almost everywhere finite. Moreover, since $T_k(u) \in W^{1,p}_0(\Omega)$, then $u$ is also cap$_p$-quasi continuous. 

\medskip

Finally let us recall that, by Theorem \ref{diffuse} below,  one has that  $\mu_d \in W^{-1,p'}(\Omega) + L^1(\Omega)$.  Since  $T_k(u_n)$ converges to $T_k(u)$ $*$-weakly in $L^\infty(\Omega)$ and weakly in $W^{1,p}_0(\Omega)$ as $n\to\infty$ for any $k>0$, then one deduces that $T_k(u_n)$ converges to $T_k(u)$ in $L^1(\Omega,\mu_d)$ as $n\to\infty$ and for any $k>0$; this argument, as $u$ is also cap$_p$-almost everywhere finite in $\Omega$, implies that $u_n$ converges to $u$ $\mu_d$-a.e.  in $\Omega$ as $n\to\infty$. This concludes the proof.
\end{proof}

 Next result guarantees a control for the possibly singular term as well as the identification of the limit of $\nabla u_n$ as $n\to\infty$. It is worth to stress that, from here on, $u$ is the function found in Lemma \ref{lemmastimehsing}. 

\begin{lemma}\label{convgradienti}
	Under the assumptions of Lemma \ref{lemmastimehsing}, it holds that for any $0\le \varphi\in W^{1,p}_0(\Omega)\cap L^\infty(\Omega)$
	\begin{equation}\label{l1locddo}
		\int_\Omega h_n(u_n)\varphi (\mu_d+\mu_n) \le C,
	\end{equation}
	 where $C$ does not depend on $n$. In particular
	\begin{equation}\label{stimavdeltaddo}
	\limsup_{n\to\infty}\int_{\{u_n\le \delta\}} h_n(u_n)\varphi (\mu_d+\mu_n) \le C_\delta,
	\end{equation}	
	where $C_\delta \to 0^+$ as $\delta \to 0^+$. Finally $\nabla u_n$ converges to $\nabla u$ almost everywhere in $\Omega$ as $n\to\infty$.
\end{lemma}
\begin{proof}
	Let us take $0\le \varphi\in W^{1,p}_0(\Omega)\cap L^\infty(\Omega)$ as a test function and $S=V_\delta$ ($\delta>0$) in the renormalized  formulation of \eqref{pbapproxDDO}. 
	Then, getting rid of the term involving $V'$ and using \eqref{cara2}, one gets 
	\begin{equation}\label{stimaHsing}
		\int_{\{u_n\le \delta\}} h_n(u_n) \varphi (\mu_d+\mu_n)\le \int_\Omega h_n(u_n)V_\delta(u_n)\varphi (\mu_d+\mu_n) \le \beta\int_\Omega |\nabla u_n|^{p-1} |\nabla \varphi|  V_\delta(u_n)\le C,
	\end{equation}
which holds thanks to the estimates of Lemma \ref{lemmastimehsing}.

Obviously one also has that
	\begin{equation}\label{stimaHsing2}
	\int_{\{u_n> \delta\}} h_n(u_n) \varphi (\mu_d+\mu_n)\le \sup_{s\in[\delta,\infty)}h(s)\int_\Omega (\mu_d+\mu_n) \le C.
\end{equation}
Hence \eqref{stimaHsing} and \eqref{stimaHsing2} give \eqref{l1locddo}.

\medskip

Now we show the almost everywhere convergence of the gradients. 

We take $T_k(T_r(u_n)-T_r(u))\varphi$  ($  r,k>0   \text{ and } 0\le\varphi\in W^{1,p}_0(\Omega)\cap L^\infty(\Omega)$) as a test function and set $S=V_r$ in the renormalized formulation of \eqref{pbapproxDDO}. 
This yields to
\begin{equation*}
\begin{aligned}
	&\int_\Omega (a(x,\nabla T_r(u_n)) - a(x,\nabla T_r(u))) \cdot \nabla T_k(T_r(u_n)-T_r(u)) \varphi \\
	&= - \int_{\{r<u_n<2r\}} a(x,\nabla u_n)  \cdot \nabla T_k(T_r(u_n)-T_r(u)) \varphi V_r(u_n) \\
	&+ \frac{1}{r}\int_{\{r<u_n<2r\}} a(x,\nabla u_n)\cdot \nabla u_n T_k(T_r(u_n)-T_r(u)) \varphi \\
	& -  \int_{\Omega} a(x,\nabla u_n)  \cdot \nabla \varphi T_k(T_r(u_n)-T_r(u)) V_r(u_n) \\
	& + \int_\Omega h_n(u_n) T_k(T_r(u_n)-T_r(u))\varphi V_r(u_n)(\mu_d+\mu_n) \\
	& -\int_\Omega a(x,\nabla T_r(u)) \cdot \nabla T_k(T_r(u_n)-T_r(u)) \varphi = (A)+(B)+(C)+(D)+(E).	
\end{aligned}
\end{equation*}
Recall that $T_r(u_n)$ is bounded in $W^{1,p}_0(\Omega)$ for any $r>0$ with respect to $n$ from Lemma \ref{lemmastimehsing}; then it weakly converges in $W^{1,p}_0(\Omega)$ and strongly in $L^q(\Omega)$ for any $q<\infty$ to $T_r(u)$ for any $r$.
These convergences imply that
$$\limsup_{n\to\infty} \  (A)+(C)+(E) = 0.$$
Now observe that it follows from \eqref{stimatk} that
$$(B)\le \frac{ck}{r}\int_{\{r<u_n<2r\}} |\nabla u_n|^p \le ck.$$
Estimate \eqref{l1locddo} also gives that $(D)\le ck$. 
All the arguments above show that
$$\limsup_{n\to\infty} \int_\Omega (a(x,\nabla T_r(u_n)) - a(x,\nabla T_r(u))) \cdot \nabla T_k(T_r(u_n)-T_r(u)) \varphi\le ck.$$
Now reasoning as in the second part of Theorem $2.1$ of \cite{bm} one deduces that, up to subsequences,   $\nabla T_r(u_n)$ converges almost everywhere to $\nabla T_r(u)$ in $\Omega$ as $n\to\infty$.

\medskip

It remains to prove \eqref{stimavdeltaddo}; let recall that \eqref{stimaHsing} gives 
\begin{equation*}
\int_{\{u_n\le \delta\}} h_n(u_n) \varphi (\mu_d+\mu_n)\le \int_\Omega h_n(u_n)V_\delta(u_n)\varphi (\mu_d+\mu_n) \le \beta\int_\Omega |\nabla u_n|^{p-1} |\nabla \varphi|  V_\delta(u_n).
\end{equation*}
Now, using that $\nabla u_n$ converges almost everywhere in $\Omega$ to $\nabla u$ in $n$ and using also that $T_k(u_n)$ is bounded in $W^{1,p}_0(\Omega)$, it is simple to convince that, as $n\to\infty$, one has
\begin{equation*}
\begin{aligned}
\limsup_{n\to\infty} \int_{\{u_n\le \delta\}} h_n(u_n) \varphi (\mu_d+\mu_n) &\le \beta\int_\Omega |\nabla u|^{p-1} |\nabla \varphi|  V_\delta(u) = C_\delta. 
\end{aligned}
\end{equation*} 
Finally note that
$$ \lim_{\delta\to 0^+} C_\delta = \lim_{\delta\to 0^+} \beta\int_\Omega |\nabla u|^{p-1} |\nabla \varphi|  V_\delta(u) = \beta\int_{\{u=0\}} |\nabla u|^{p-1} |\nabla \varphi| = 0.$$
This concludes the proof.
\end{proof}

\subsubsection{\underline{Strong convergence of truncations}}

In this section we show that the truncation at any level for the sequence $u_n$ strongly converges in $W^{1,p}_0(\Omega)$ with respect to $n$.
\begin{lemma}\label{stronghsing} Under the assumptions of Lemma \ref{lemmastimehsing}, $T_{k}(u_n)$ converges to $T_{k}(u)$ in $W^{1,p}_{0}(\Omega)$ as $n\to\infty$.
\end{lemma}
\begin{proof}	 
	 The proof will be concluded by applying Lemma $5$ of \cite{bmp}
	after one shows that
	\begin{equation}\label{strong}
	\lim_{n\to\infty}\int_\Omega\big(a(x,\nabla T_{k}(u_n))-a(x,\nabla T_{k}(u))\big)\cdot\nabla(T_{k}(u_n)-T_{k}(u))=0.
	\end{equation}
	Let us take $(T_{k}(u_n)-T_{k}(u))(1-\Psi_{\nu})$ as a test function where, for $\nu>0$,  $\Psi_{\nu}$ is as in Lemma \ref{dalmaso} below and set $S=V_r$ ($r>k$) in the renormalized formulation of \eqref{pbapproxDDO}. This takes to
	\begin{equation}
	\begin{aligned}\label{1-4.1}
	&\int_\Omega a(x,\nabla T_k(u_n))\cdot\nabla(T_{k}(u_n)-T_{k}(u))(1-\Psi_{\nu})\\
	&=-\int_{\{k<u_n<2r\}}a(x,\nabla u_{n})\cdot\nabla(T_{k}(u_n)-T_{k}(u))V_r(u_n)(1-\Psi_{\nu}) \ \ \  \\
	&+\frac{1}{r}\int_{\{r<u_n<2r\}}a(x,\nabla u_{n})\cdot\nabla u_{n}(T_{k}(u_n)-T_{k}(u))(1-\Psi_{\nu}) \ \ \ \\
	&+\int_\Omega h_n(u_n)V_r(u_n)(T_{k}(u_n)-T_{k}(u))(1-\Psi_{\nu})\mu_d \ \ \ \\
	&+\int_\Omega h_n(u_n)V_r(u_n)(T_{k}(u_n)-T_{k}(u))(1-\Psi_{\nu})\mu_n \ \ \ \\
	&+\int_\Omega a(x,\nabla u_{n})\cdot\nabla\Psi_{\nu}(T_{k}(u_n)-T_{k}(u))V_r(u_n) = (A)+(B)+(C)+(D)+(E).\ \ \ 
	\end{aligned}
	\end{equation}
	Let us start estimating $(A)$ by 
	$$(A) \le \int_\Omega|a(x,\nabla u_{n})|V_r(u_n)|\nabla T_{k}(u)|\chi_{\{u_n>k\}}.$$
	The right-hand of the previous goes to zero as $n\to\infty$ since $|a(x,\nabla u_{n})|V_r(u_n)$ is bounded in $L^{p'}(\Omega)$
	with respect to $n$ and $\displaystyle |\nabla T_{k}(u)|\chi_{\{u_n>k\}}$ converges to zero in $L^p(\Omega)$.
	This shows that
	\begin{equation}\label{Aterm}
	\limsup_{n\to\infty} (A) \le 0.
	\end{equation}
	To estimate $(B)$ let us take $\pi_r(u_n)(1-\Psi_{\nu})$ ($\pi_r$ is defined into \eqref{not:pi}) as a test function in the renormalized formulation of \eqref{pbapproxDDO} and $S=V_m$, where $m>r$ (once again $\Psi_\nu$ as in Lemma \ref{dalmaso}). One has
	\begin{equation}\label{Bq}\begin{aligned}
	&\frac{1}{r}\int_{\{r<u_n<2r\}}a(x,\nabla u_{n})\cdot\nabla u_{n}V_m(u_n)(1-\Psi_{\nu})\\
	&=\frac{1}{m}\int_{\{m<u_n<2m\}}a(x,\nabla u_{n})\cdot\nabla u_{n}\pi_r(u_n)(1-\Psi_{\nu})\ \ \   \\
	&+\int_\Omega h_n(u_n)\pi_r(u_n)V_m(u_n)(1-\Psi_{\nu})\mu_d\ \ \  \\
	&+\int_\Omega h_n(u_n)\pi_r(u_n)V_m(u_n)(1-\Psi_{\nu})\mu_n\ \ \  \\
	&+\int_\Omega a(x,\nabla u_{n})\cdot\nabla\Psi_{\nu}\pi_r(u_n) V_m(u_n) = (B_1)+(B_2)+(B_3)+(B_4).\ \ \ 
	\end{aligned}\end{equation}
	Hence, let us take $m,n,r\to\infty$ and $\nu\to 0$ in the previous; for $(B_1)$ one has
	\begin{equation*}\label{Bq1}\begin{aligned}
	&\lim_{m\to\infty}\frac{1}{m}\int_{\{m<u_n<2m\}}a(x,\nabla u_{n})\cdot\nabla u_{n}\pi_r(u_n)(1-\Psi_{\nu})\\
	& \le \lim_{m\to\infty}\frac{1}{m}\int_{\{m<u_n<2m\}}a(x,\nabla u_{n})\cdot\nabla u_{n}=0, 
	\end{aligned}\end{equation*}
 	since $u_n$ is a renormalized solution and \eqref{ren2} holds.
	\\Concerning $(B_2)$, one gets that 
	$$(B_2) \le \sup_{s\in[r,\infty)}h(s) \int_\Omega \pi_r(u_n)\mu_d,$$
	and since $\pi_r(u_n)$ is bounded in $W^{1,p}_0(\Omega)\cap L^\infty(\Omega)$ with respect to $n$, one can   pass to the limit as $n\to \infty$ (recall Theorem \ref{diffuse} and Lemma \ref{lemmastimehsing}), yielding to
	$$\lim_{n\to\infty} \lim_{m\to\infty} (B_2) \le \sup_{s\in[r,\infty)}h(s) \int_\Omega \pi_r(u)\mu_d,$$
	and finally
	$$\lim_{r\to\infty} \lim_{n\to\infty} \lim_{m\to\infty} (B_2) = 0,$$
	since $u$ is cap$_p$-almost everywhere finite and one can apply the Lebesgue Theorem (with respect to $\mu_d$).

	As for $(B_3)$, one has only to recall that $\mu_n$ narrow converges to $\mu_c$ (see also Lemma \ref{dalmaso}); indeed it holds that
	$$\limsup_{n\to\infty} \lim_{m\to\infty}(B_3) \le \lim_{n\to\infty} \sup_{s\in[r,\infty)}h(s) \int_\Omega (1-\Psi_{\nu})\mu_{n}=\sup_{s\in[r,\infty)}h(s)\int_\Omega (1-\Psi_{\nu})\mu_c\le C\nu,$$
	and so
	$$\lim_{\nu\to 0^+}\limsup_{r\to\infty}\limsup_{n\to\infty} \lim_{m\to\infty}(B_3)  =  0.$$	
	
	An application of the Lebesgue Theorem for $(B_4)$ gives that
	\begin{equation*}\label{senzam}
	\lim_{m\to\infty}\int_\Omega a(x,\nabla u_{n})\cdot\nabla\Psi_\nu\pi_r(u_n)V_m(u_n)=\int_\Omega a(x,\nabla u_{n})\cdot\nabla\Psi_\nu\pi_r(u_n).
	\end{equation*}
	Since $\supp (\pi_r(t))=\{|t|\geq r\}$, $u$ is almost everywhere finite,  and $|\nabla u_n|^{p-1}$ is bounded in $L^q(\Omega)$ for any $q<\frac{N}{N-1}$, then applying the H\"older inequality one has
	\begin{align*}
	(B_4) \leq C \left(\int_{\Omega}|\nabla u_n|^{(p-1)q}\right)^{\frac{1}{q}}\left|\{u_n\geq r\}\right|^\frac{1}{q'}\leq C\,\left|\{u_n\geq r\}\right|^\frac{1}{q'},
	\end{align*}
	and $(B_4)$ goes to zero as $m,n,r \to\infty$. Hence, by taking $m,n,r\to\infty$ and $\nu \to 0^+$ in \eqref{Bq}, we have shown that 
	\begin{equation}\label{Bterm}
	\lim_{\nu\to 0^+}  \limsup_{r\to\infty}\limsup_{n\to\infty} 	(B)  =  0.
	\end{equation}
	Let us focus on $(C)$ in \eqref{1-4.1}. If $h(0)<\infty$, then one deduces that $(C)$ goes to zero as $n\to\infty$ since $T_k(u_n)$ converges to $T_k(u)$ in $L^q(\Omega,\mu_d)$ for any $q<\infty$ and $k>0$   as it is implied by Lemma \ref{lemmastimehsing}.
	
	So that, without loosing generality, we assume $h(0)=\infty$.    Lemma \ref{lemmastimehsing} ensures that $u_n$ converges to $u$ $\mu_d$-a.e.  in $\Omega$, then an application of the Fatou Lemma in \eqref{l1locddo} gives that $h(u) \in L^1_{\rm loc}(\Omega,\mu_d)$ which also implies
	\begin{equation}\label{u00}
		\mu_d(\{u=0\}) = 0.
	\end{equation}
	In particular, for some $\delta>0$,  one has
	\begin{equation}\label{C}
		(C) \le \int_{\{u_n\le \delta\}} h_n(u_n)V_r(u_n)(T_{k}(u_n)-T_{k}(u))(1-\Psi_\nu)\mu_d + \sup_{s\in [\delta,\infty)}h(s)\int_{\{u_n> \delta\}} |T_{k}(u_n)-T_{k}(u)|\mu_d,
	\end{equation}	
	and the second term on the right-hand of the previous goes to zero because  again  
      $T_k(u_n)$ converges to $T_k(u)$ in $L^q(\Omega,\mu_d)$ for any $q<\infty$ and for any $k>0$.
     Now, recalling  that $\gamma\le 1$ and $\sup_{s\in[0,\infty)}h(s)T_k(s)$ is finite, then  the first term on the right-hand of \eqref{C} can be estimated as follows 
	 	\begin{equation*}\label{C1}
	 	\int_{\{u_n\le \delta\}} h_n(u_n)V_r(u_n)(T_{k}(u_n)-T_{k}(u))(1-\Psi_\nu)\mu_d \le \sup_{s\in[0,\infty)}h(s)T_k(s) \int_{\{u_n\le \delta\}} \mu_d, \end{equation*}
	 	and the right-hand of the previous goes to zero as $n\to\infty$ and $\delta \to 0^+$ thanks to \eqref{u00} since $u_n$ converges  to $u$ $\mu_d$-a.e. in $\Omega$. 
	 	Therefore,  we have shown that 
	 	\begin{equation}\label{Cterm}
	 		\limsup_{n\to\infty} (C) \le 0.
	 	\end{equation}	
	Concerning  $(D)$, using again  that $\gamma\le 1$ and that  $\sup_{s\in(0,\infty)}h(s)T_k(s)$ is finite,  then
	$$\limsup_{n\to\infty} (D)\le \lim_{n\to\infty} \sup_{s\in(0,\infty)}h(s)T_k(s) \int_\Omega(1-\Psi_{\nu})\mu_{n}= \sup_{s\in(0,\infty)}h(s)T_k(s)\int_\Omega (1-\Psi_{\nu})\mu_c\le C\nu,$$
	which follows from the narrow convergence of $\mu_{n}$ to $\mu_c$ and from Lemma \ref{dalmaso}.  This clearly gives
	\begin{equation}\label{Dterm}
		\limsup_{\nu\to 0^+}\limsup_{r\to\infty}\limsup_{n\to\infty}(D) \le 0.
	\end{equation}
	Finally observe that
		\begin{equation}\label{Eterm}
		\limsup_{n\to\infty} (E)\le 0,
	\end{equation}
	since $\displaystyle |a(x,\nabla u_{n})||\nabla\Psi_{\nu}|V_r(u_n)$ is bounded in $L^{p'}(\Omega)$ and $T_k(u_n)$ strongly converges to $T_k(u)$ in $L^p(\Omega)$.
	
	Then, from \eqref{Aterm}, \eqref{Bterm}, \eqref{Cterm}, \eqref{Dterm}  and \eqref{Eterm}, one obtains from \eqref{1-4.1} that 
	\begin{equation}\label{convtronc}
			\limsup_{\nu\to 0^+}\limsup_{n\to\infty}\int_\Omega a(x,\nabla T_{k}(u_n)\cdot\nabla(T_{k}(u_n)-T_{k}(u))(1-\Psi_{\nu})\le 0.
	\end{equation}
	Now observe that 
	\begin{equation}
	\begin{aligned}\label{tesiq}
	&\int_\Omega \big(a(x,\nabla T_{k}(u_n)-a(x,\nabla T_{k}(u)\big)\cdot\nabla(T_{k}(u_n)-T_{k}(u))\\
	&=\int_\Omega\big(a(x,\nabla T_{k}(u_n))-a(x,\nabla T_{k}(u)\big)\cdot\nabla(T_{k}(u_n)-T_{k}(u))\Psi_{\nu}\\
	&+\int_\Omega a(x,\nabla T_{k}(u_n)\cdot\nabla(T_{k}(u_n)-T_{k}(u))(1-\Psi_{\nu})\\
	&-\int_\Omega a(x,\nabla T_{k}(u)\cdot\nabla(T_{k}(u_n)-T_{k}(u))(1-\Psi_{\nu}),\\
	\end{aligned}
	\end{equation}
	and, in order to conclude, we just need to pass to the limit the first term on the right-hand of the previous. We choose as test function $(k-u_{n})^+\Psi_{\nu}$ and $S=V_k$ in the renormalized formulation of \eqref{pbapproxDDO}, obtaining  
	\begin{align*}
	&-\int_\Omega a(x,\nabla T_k(u_{n}))\cdot\nabla T_{k}(u_{n})\Psi_{\nu} + \int_\Omega a(x,\nabla T_{k}(u_{n}))\cdot\nabla\Psi_{\nu} (k-u_{n})^+\\
	&=\int_\Omega h_n(u_{n})(k-u_{n})^+\Psi_{\nu}\mu_{d} + \int_\Omega h_n(u_{n})(k-u_{n})^+\Psi_{\nu}\mu_{n},
	\end{align*}
	which, since $\mu_{d}\geq 0$ and thanks to \eqref{cara1}, gives that
	\begin{align*}\label{u1}
	&\alpha\int_\Omega |\nabla T_{k}(u_{n})|^p\Psi_\nu + \int_\Omega h_n(u_{n})(k-u_{n})^+\Psi_{\nu}\mu_{n}\leq \int_\Omega a(x,\nabla T_{k}(u_{n}))\cdot\nabla\Psi_{\nu} (k-u_{n})^+.
	\end{align*}
	Since $T_{k}(u_n)$ is bounded in $\displaystyle W^{1,p}_0(\Omega)$, \eqref{cara2} and Lemma \ref{dalmaso} guarantee that the right-hand of the previous goes to zero as $n\to\infty$ and $\nu \to 0^+$.
	Hence we have proved that
	\begin{equation} \label{u3}
	\limsup_{\nu\to 0^+}\limsup_{n\to\infty}\int_\Omega |\nabla T_{k}(u_{n})|^p\Psi_\nu=0
	\end{equation}
	and
	\begin{equation}\label{u4}
	\limsup_{\nu\to 0^+}\limsup_{n\to\infty} \int_\Omega h_n(u_{n})(k-u_{n})^+\Psi_{\nu}\mu_{n}=0.
	\end{equation}
	Finally, by using \eqref{convtronc} and \eqref{u3} into \eqref{tesiq}, one has
	$$\limsup_{n\to\infty}\int_\Omega \big(a(x,\nabla T_{k}(u_n)-a(x,\nabla T_{k}(u)\big)\cdot\nabla(T_{k}(u_n)-T_{k}(u))=0,$$
	which is \eqref{strong} and the proof concludes.
\end{proof}

\begin{remark}\label{epsilon0} 
Let us focus a bit more on   Lemma \ref{stronghsing}. We  recall that, in Lemma \ref{lemmastimehsing}, we have shown that $u_{n}$ converges, up to subsequences, to $u$ $\mu_d$-a.e.  in $\Omega$; in particular this convergence property has been crucial in order to prove Lemma \ref{stronghsing}. By the way it is worth mentioning that, once that Lemma \ref{stronghsing} is in force, a finer convergence result holds for the sequence $u_n$. Indeed one is in position to apply \cite[Lemma $3.5$]{kkm}) in order to deduce, by a standard diagonal argument,  that $u_{n}$ converges, up to subsequences,  to $u$ cap$_p$-almost everywhere as $n\to\infty$ in $\Omega$. 
\triang \end{remark}

\begin{remark}\label{remconvddo}
Let us underline that from Lemma \ref{lemmastimehsing} and Lemma \ref{stronghsing} one has that, if $p>2-\frac{1}{N}$, $u_n$ converges to $u$ strongly in $W_{0}^{1,q}(\Omega)$ for every $q<\frac{N(p-1)}{N-1}$. If $1<p\leq2-\frac{1}{N}$, $u_n^{p-1}$ converges to $u^{p-1}$ strongly in $L^q(\Omega)$ for every $q<\frac{N}{N-p}$ and	$|\nabla u_n|^{p-1}$ converges to $|\nabla u|^{p-1}$ strongly in $L^q(\Omega)$ for every $q<\frac{N}{N-1}$. 
\triang \end{remark}

\subsubsection{\underline{Existence of a renormalized solution}}

In this section we prove Theorem \ref{teoexrinuniqueness}.
In the proof we first show that the almost everywhere limit of $u_n$ satisfies \eqref{distrdef}, which is instrumental for proving the existence of a renormalized solution.

Let underline that here it is where we take advantage of the two-parameters approximation given by \eqref{pbapproxeps}. By the way, for the sake of presentation and with a little abuse of notation, we denote by $u_n$ the almost everywhere limit of $u_{n,m}$ as $m\to\infty$.

\begin{proof}[Proof of Theorem \ref{teoexrinuniqueness}]
	Let $u_{n,m}$ be a renormalized solution to \eqref{pbapproxeps}; then it follows from Lemma \ref{lemmastimehsing} that there exists a function $u$ which is its almost everywhere limit as $m,n\to\infty$. Moreover, recalling Remark \ref{epsilon0},  $u_{n,m}$ converges, up to subsequences,  to $u$ cap$_p$-almost everywhere. Let also observe that an application of the Fatou Lemma as $n\to\infty$ in \eqref{l1locddo} gives that $h(u) \in L^1_{\rm loc}(\Omega,\mu_d)$. Moreover the regularity properties on $u$ standardly follow by Lemma \ref{lemmastimehsing}.
	 
	 \medskip
	 
	 As already mentioned we want to show \eqref{distrdef} passing to the limit first in $m$ and then in $n$ in
	\begin{equation}
	\int_\Omega a(x,\nabla u_{n,m})\cdot\nabla\varphi=\int_\Omega h_n(u_{n,m})\varphi \mu_d+\int_\Omega h_n(u_{n,m})\varphi\mu_m,\;\ \ \ \forall\varphi\in C^1_c(\Omega).
	\label{ndist}
	\end{equation}
	Lemmas \ref{lemmastimehsing} and \ref{convgradienti} (see also Remark \ref{remconvddo})  and assumption \eqref{cara2}  allow to pass to the limit in the first term on left-hand of the previous as $m,n\to\infty$. 
	
	\smallskip
	
	As for the right-hand let us firstly pass to the limit the second term, namely one has to show (let assume $n$ large enough)
	\begin{equation}\label{ddoconc}
		\lim_{m\to\infty} \int_\Omega h_n(u_{n,m})\varphi\mu_m = h(\infty)\int_\Omega \varphi\mu_c.
	\end{equation}
	Observe that one can estimate
	\begin{equation*}
	\begin{aligned}\label{dis3}
	\left|\int_\Omega h_n(u_{n,m})\varphi\mu_{m}-\int_\Omega h(\infty)\varphi \mu_{c}\right| \leq &\left|\int_\Omega (h_n(u_{n,m})-h(\infty))\varphi\mu_{m}\right| 
	\\
	&+ \left|\int_\Omega h(\infty)\varphi(\mu_{m}-\mu_{c})\right|  
	\end{aligned}
	\end{equation*}
	which, thanks to the narrow convergence of $\mu_m$, gives
	\begin{equation}
	\begin{aligned}\label{dis3bis}
	\limsup_{m\to\infty}\left|\int_\Omega h_n(u_{n,m})\varphi\mu_{m}-\int_\Omega h(\infty)\varphi \mu_{c}\right| \leq &\limsup_{m\to\infty}\left|\int_\Omega (h_n(u_{n,m})-h(\infty))\varphi\mu_{m}\right|.  
	\end{aligned}
	\end{equation}		
From now, without loss of generality, we assume $n>h(\infty)$ and  in order to estimate the right-hand of the previous observe that for every $\eta>0$ there exist $s_{\eta}>0$ and $L_{\eta}>0$ such that
	\begin{equation}
	\label{H1}
	|h_n(s)-h(\infty)|\leq\eta, \qquad \forall s \in (s_{\eta},\infty);
	\end{equation}
	and, since $h$ is positive, one also has 
	\begin{equation}\label{H2}
	\frac{|h_n(s)-h(\infty)|}{h_n(s)} \leq L_{\eta}(2s_{\eta}-s), \qquad \forall s\in[0,s_{\eta}].
	\end{equation}
	Hence \eqref{H1}, \eqref{H2}, and the properties of $\Psi_\nu$ (defined as in Lemma \ref{dalmaso}) and $\mu_m$ give that
	\begin{align*}
	\int_\Omega |h_n(u_{n,m})-h(\infty)|\mu_{m}=&\int_\Omega |h_n(u_{n,m})-h(\infty)|\Psi_\nu\mu_{m}+\int_\Omega |h_n(u_{n,m})-h(\infty)|(1-\Psi_\nu)\mu_{m}\\
	\leq &\eta\int_{\{u_{n,m}> s_{\eta}\}}\Psi_{\nu}\mu_{m}+L_{\eta}\int_{\{u_{n,m}\leq s_{\eta}\}} h_n(u_{n,m})(2s_{\eta}-u_{n,m})\Psi_\nu \mu_{m}
	\\
	&+2n\int_\Omega (1-\Psi_\nu)\mu_{m}.
	\end{align*}
	 As \eqref{u4} is in force, taking $m\to \infty$ and $\nu,\eta \to 0^+$ in the previous gives that the right-hand of \eqref{dis3bis} is zero, namely \eqref{ddoconc}. 
	
	\smallskip
	
	Let us now focus on the first term on the right-hand of \eqref{ndist}. If $h(0)<\infty$ one can simply pass to the limit by the Lebesgue Theorem for a general measure since $u_{n,m}$ converges cap${_p}$-almost everywhere in $\Omega$ to $u$ as $m,n\to\infty$. Hence, without loss of generality, from here we assume that $h(0)=\infty$.
	Obviously, once again, since $u_{n,m}$ converges cap${_p}$-almost everywhere in $\Omega$ to $u_n$ as $m\to\infty$, one can take $m\to\infty$ in the first term of the right-hand of \eqref{ndist}. Then, in order to take $n\to\infty$, we assume $\delta \not\in \{\eta: \mu_d(\{u=\eta\})>0\}$ which is at most a countable set since $u\in L^1(\Omega,\mu_d)$; then we write 
	\begin{equation*}\label{fatounodea}
	\int_\Omega h_n(u_n)\varphi \mu_d = \int_{\{u_n\le \delta\}} h_n(u_n)\varphi \mu_d + \int_{\{u_n> \delta\}} h_n(u_n)\varphi \mu_d.
	\end{equation*}	
	Recall that \eqref{stimavdeltaddo} gives that the first term goes to zero as the limsup is taken as $n\to\infty$ and then  $\delta \to 0^+$. For the second term let us apply the Lebesgue Theorem as $n\to\infty$ since $h_n(u_n)\chi_{\{u_n>\delta\}} \le  \sup_{s\in [\delta,\infty)}h(s)$ yielding to 
	\begin{equation}\label{lebesgueddo}
	\lim_{n\to\infty}\int_{\{u_n> \delta\}} h_n(u_n)\varphi \mu_d = \int_{\{u> \delta\}} h(u)\varphi \mu_d.
	\end{equation}
	As already noticed in the proof of Lemma \ref{stronghsing}, an application of the Fatou Lemma in \eqref{l1locddo} gives that $h(u) \in L^1_{\rm loc}(\Omega,\mu_d)$ which also implies
	\begin{equation}\label{u0}
	 \mu_d(\{u=0\}) = 0.
	\end{equation}
	This allows once again an application of the Lebesgue Theorem in \eqref{lebesgueddo}, which takes to
		\begin{equation*}\label{lebesgueddo2}
	\lim_{\delta \to 0^+}\lim_{n\to\infty}\int_{\{u_n> \delta\}} h_n(u_n)\varphi \mu_d = \int_{\{u> 0\}} h(u)\varphi \mu_d.
	\end{equation*}	
	Hence this is sufficient to deduce that 
	\begin{equation*}
	\lim_{n\to\infty}\int_\Omega h_n(u_n)\varphi \mu_d = \int_{\{u> 0\}} h(u)\varphi \mu_d \overset{\eqref{u0}}{=} \int_{\Omega} h(u)\varphi \mu_d.
	\end{equation*}	
	Therefore we have shown that $u$ satisfies \eqref{distrdef}, namely that $u$ is a distributional solution to \eqref{pbmain}.
	
	\medskip
	
	Now are ready to prove that $u$ is a solution to the problem even in a renormalized sense.

	Firstly we want to pass to the limit as $m,n\to\infty$ in the following 
	\begin{align}\label{new1}
	&\int_\Omega a(x,\nabla u_{n,m})\cdot\nabla\varphi S(u_{n,m}) + \int_\Omega a(x,\nabla u_{n,m})\cdot \nabla u_{n,m}S'(u_{n,m})\varphi
	\\
	&=\int_\Omega h_n(u_{n,m})S(u_{n,m})\varphi \mu_d+\int_\Omega h_n(u_{n,m})S(u_{n,m})\varphi\mu_m,\nonumber
	\end{align}	
	where $S\in W^{1,\infty}(\mathbb{R})$ with supp$(S)$ $\subset [-M,M]$ ($M>0$) and $\varphi\in W^{1,p}_0(\Omega)\cap L^\infty(\Omega)$.
	
	Regarding  the left-hand of \eqref{new1},  one  has
	\begin{align*}
	\lim_{n\to\infty}\lim_{m\to\infty}&\left(\int_\Omega a(x,\nabla u_{n,m})\cdot\nabla\varphi S(u_{n,m})+\int_\Omega a(x,\nabla u_{n,m})\cdot \nabla u_{n,m}S'(u_{n,m})\varphi\right)
	\\
	&=\int_\Omega a(x,\nabla u)\cdot \nabla\varphi S(u)+\int_\Omega a(x,\nabla u)\cdot \nabla u S'(u)\varphi,
	\end{align*}
	 since Lemma \ref{stronghsing} gives that $T_M(u_{n,m})$ converges to $T_M(u)$ in $W^{1,p}_{0}(\Omega)$ as $n,m\to\infty$ and thanks to \eqref{cara2} and to the generalized dominated convergence Lebesgue  Theorem. 
	
	For the first term on the right-hand of \eqref{new1} we can reason as in the first part of the proof in order to get 
	$$\lim_{n\to\infty}\lim_{m\to\infty}\int_\Omega h_n(u_{n,m})S(u_{n,m})\varphi \mu_d=\int_\Omega h(u)S(u)\varphi \mu_d.$$
	Let observe that here $\varphi \in W^{1,p}_{0}(\Omega)\cap L^\infty(\Omega)$ but this does not affect  the proof of the above equality. 
	
	For the second term on the right-hand of \eqref{new1} we have, once again, that there exist $k>0$ and $c_k>0$ such that
	\begin{equation}\label{app1}
	\begin{aligned}
	\int_\Omega h_n(u_{n,m})S(u_{n,m})\varphi\mu_m &\leq c_k||\varphi||_{L^{\infty}(\Omega)}\int_\Omega h_n(u_{n,m})(k-u_{n,m})^+\Psi_\nu\mu_m \\
	&+||S||_{L^{\infty}(\mathbb{R})}||\varphi||_{L^{\infty}(\Omega)}\int_\Omega h_n(u_{n,m})(1-\Psi_\nu)\mu_m. 
	\end{aligned}
	\end{equation}
	Using $S(s)=(k-|s|)^+$ and $\varphi=\Psi_\nu$ in the renormalized formulation of \eqref{pbapproxeps} and dropping positive terms we obtain  
	\begin{equation}\label{app2}
	\begin{aligned}
	\int_\Omega h_n(u_{n,m})(k-u_{n,m})^+\Psi_\nu\mu_m &\leq \int_\Omega a(x,\nabla T_k(u_{n,m}))\cdot \nabla\Psi_\nu (k-u_{n,m})^+ \\
	&\leq k ||T_k(u_{n,m})||_{W^{1,p}_0(\Omega)} ||\Psi_\nu||_{W^{1,p}_0(\Omega)}. 
	\end{aligned}
	\end{equation}
	Then, from \eqref{app1} and \eqref{app2}, we deduce, applying Lemma \ref{dalmaso} and Lemma \ref{lemmastimehsing}, that
	$$\lim_{n\to\infty}\lim_{m\to\infty}\int_\Omega h_n(u_{n,m})S(u_{n,m})\varphi\mu_m=0.$$
	Hence we have shown that
	\begin{equation}
	\label{new2}
	\int_\Omega a(x,\nabla u)\cdot\nabla\varphi S(u) + \int_\Omega a(x,\nabla u)\cdot \nabla u S'(u)\varphi=\int_\Omega h(u)S(u)\varphi \mu_d, \\
	\end{equation}
	for any $S \in W^{1,\infty}(\mathbb{R})$  with compact support and for any $\varphi\in W^{1,p}_0(\Omega)\cap L^\infty(\Omega)$, namely \eqref{ren1}. 
	
	\medskip
	
	Now we are going to show \eqref{ren2}; we take $S=V_t$ in \eqref{new2} and $\varphi \in C^1_c(\Omega)$, yielding to
	\begin{align*}
	\frac{1}{t}\int_{\{t< u< 2t\}}a(x,\nabla u)\cdot\nabla u \varphi=-\int_\Omega h(u)V_{t}(u)\varphi \mu_{d}+\int_\Omega a(x,\nabla u)\cdot\nabla\varphi V_{t}(u).
	\end{align*}
	Since $h(u) \in L^1_{\rm loc}(\Omega,\mu_d)$, $|\nabla u|^{p-1} \in L^1(\Omega)$  and \eqref{cara2} is in force, we can take $t\to\infty$ obtaining
	\begin{align*}
	\lim_{t\to\infty}\frac{1}{t}\int_{\{t< u< 2t\}}a(x,\nabla u)\cdot\nabla u \varphi=-\int_\Omega h(u)\varphi \mu_{d}+\int_\Omega a(x,\nabla u)\cdot\nabla\varphi,
	\end{align*}
	which implies, as $u$ satisfies \eqref{distrdef}, that
	\begin{equation}\label{new3}
	\lim_{t\to\infty}\frac{1}{t}\int_{\{t< u< 2t\}}a(x,\nabla u)\cdot\nabla u \varphi=h(\infty)\int_\Omega \varphi \mu_{c} \qquad \forall \varphi\in C^1_c(\Omega).
	\end{equation}
	By the density of $C^1_c(\Omega)$ in $C_c(\Omega)$, \eqref{new3} holds even if  $\varphi\in C_c(\Omega)$. Now, if $\varphi\in C_b(\Omega)$, we have $\varphi\Psi_\nu\in C_c(\Omega)$ and then
	\begin{equation}\label{new4}
	\lim_{t\to\infty}\frac{1}{t}\int_{\{t< u< 2t\}}a(x,\nabla u)\cdot\nabla u \Psi_\nu\varphi=h(\infty)\int_\Omega \varphi\Psi_\nu \mu_{c}\qquad \forall \varphi\in C_b(\Omega).
	\end{equation}
	We want to prove that 
	\begin{equation}\label{new5}
	\lim_{\nu \to 0^+}\lim_{t\to\infty}\frac{1}{t}\int_{\{t< u< 2t\}}a(x,\nabla u)\cdot\nabla u(1-\Psi_\nu)\varphi = 0\qquad \forall \varphi\in C_b(\Omega).
	\end{equation}
	To this aim we take  $\pi_t(u_{n,m})(1-\Psi_{\nu})$ as a test function  ($\pi_t$ is defined into \eqref{not:pi})  and  $S=V_r$  in the renormalized formulation of \eqref{pbapproxeps}. This takes to
	\begin{equation}
	\begin{aligned}\label{newabcd}
	&\frac{1}{t}\int_{\{t<u_{n,m}<2t\}}a(x,\nabla u_{n,m})\cdot\nabla u_{n,m}V_r(u_{n,m})(1-\Psi_{\nu})\\
	=&\frac{1}{r}\int_{\{r<u_{n,m}<2r\}}a(x,\nabla u_{n,m})\cdot\nabla u_{n,m}\pi_t(u_{n,m})(1-\Psi_{\nu}) \\
	&+\int_\Omega h_n(u_{n,m})\pi_t(u_{n,m})V_r(u_{n,m})(1-\Psi_{\nu})\mu_d \\
	&+\int_\Omega h_n(u_{n,m})\pi_t(u_{n,m})V_r(u_{n,m})(1-\Psi_{\nu})\mu_m  \\
	&+\int_\Omega a(x,\nabla u_{n,m})\cdot\nabla\Psi_{\nu}\pi_t(u_{n,m})V_r(u_{n,m}) = (A) +(B) +(C) +(D).  
	\end{aligned}
	\end{equation}
	The Lebesgue Theorem implies that 
	\begin{equation*}
	\lim_{r\to\infty}\int_\Omega a(x,\nabla u_{n,m})\cdot\nabla\Psi_\nu\pi_t(u_{n,m})V_r(u_{n,m})=\int_\Omega a(x,\nabla u_{n,m})\cdot\nabla\Psi_\nu\pi_t(u_{n,m}).
	\end{equation*}
	Since $u$ is almost everywhere finite, $|\nabla u_{n,m}|^{p-1}$ is bounded in $L^q(\Omega)$ for every $q<\frac{N}{N-1}$, one can use \eqref{cara2} and the H\"older inequality with exponents $q$ and $q'$, with $1<q<\frac{N}{N-1}$ fixed, in order to have 
	\begin{align*}
	\left|\int_\Omega a(x,\nabla u_{n,m})\cdot\nabla\Psi_\nu\pi_t(u_{n,m})\right|&\leq ||\nabla\Psi_\nu||_{L^\infty(\Omega)}\left(\int_\Omega|\nabla u_{n,m}|^{(p-1)q}\right)^{\frac{1}{q}}\left|\{x\in \Omega : u_{n,m}(x)\geq t\}\right|^\frac{1}{q'}\\
	&\leq C\,\left|\{x\in \Omega : u_{n,m}(x)\geq t\}\right|^\frac{1}{q'}.
	\end{align*}
	This shows that
	\begin{equation}
	\label{newd}
	\lim_{t\to\infty}\limsup_{n\to\infty}\limsup_{m\to\infty} \ (D) = 0.
	\end{equation}
	For (B)
	\begin{align*}
	(B)\le\sup_{s\in [t,\infty)}h(s)\int_\Omega \pi_t(u_{n,m})(1-\Psi_{\nu})\mu_d,
	\end{align*}
	that implies 
	\begin{equation}
	\label{newb}
	\lim_{t\to\infty}\lim_{n\to\infty}\lim_{m\to\infty} \ (B) = 0.
	\end{equation}
	Moreover
	$$  (C)\leq \int_\Omega h_n(u_{n,m})\pi_t(u_{n,m})(1-\Psi_{\nu})\mu_m\le  \sup_{s\in [t,\infty)}h(s) \int_\Omega(1-\Psi_{\nu})\mu_{m}.$$
	By the narrow convergence of $\mu_{m}$ and Lemma \ref{dalmaso} one has
	\begin{equation}
\label{newc}
\lim_{\nu\to 0^+}\limsup_{t\to\infty}\limsup_{n\to\infty}\limsup_{m\to\infty} \ (C) \le 0.
\end{equation}

	Finally,  as $u_{n,m}$ is a renormalized solution to \eqref{pbapproxeps}, one deduces 
	\begin{equation}
	\begin{aligned}\label{newa}
	&\lim_{r\to\infty} (A) \le  \lim_{r\to\infty} \frac{1}{r}\int_{\{r<u_{n,m}<2r\}}a(x,\nabla u_{n,m})\cdot\nabla u_{n,m} = 0.
	\end{aligned}	
	\end{equation}
	Letting $r,m,n,t \to \infty$  and $\nu\to 0^+$ in \eqref{newabcd} and using \eqref{newd}, \eqref{newb},\eqref{newc} and \eqref{newa}, we get
 \eqref{new5}. As a consequence of \eqref{new4} and \eqref{new5}, letting $\nu$ tend to zero, by Lemma \ref{dalmaso} we have 
	\begin{equation*}\label{ener}
	\lim_{t\to\infty}\frac{1}{t}\int_{\{t< u< 2t\}}a(x,\nabla u)\cdot\nabla u \varphi=h(\infty)\int_\Omega \varphi \mu_{c},
	\end{equation*}
	for all $\varphi\in C_b(\Omega)$ which is \eqref{ren2}. 
This concludes the proof.
\end{proof}

\subsubsection{\underline{Uniqueness of the renormalized solution}}\label{um1}

Here we show that uniqueness holds for renormalized solutions when $h$ is non-increasing and  $\mu$ is diffuse with respect to the $p$-capacity.

\begin{proof}[Proof of Theorem \ref{teounique}]
	Let us suppose that $u_1$ and $u_2$ are renormalized solutions to \eqref{pbmain}.  Let $S= V_n$ ($V_n$ is defined in \eqref{not:Vdelta}) and let us take $V_n(u_2) T_k(u_1-u_2)$ as a test function in the renormalized formulation of $u_1$ and $V_n(u_1) T_k(u_1-u_2)$ as a test function in the renormalized formulation of $u_2$.  We obtain 
	\begin{equation}\label{uniqueren0}
		\begin{aligned}
		&\int_\Omega (a(x,\nabla u_1)-a(x,\nabla u_2)\cdot \nabla T_k(u_1-u_2) V_n(u_1)V_n(u_2) - \frac{1}{n}\int_{\{n<u_1<2n\}} a(x,\nabla u_1)\cdot \nabla u_1 V_n(u_2)T_k(u_1-u_2)
		\\
		&+ \frac{1}{n}\int_{\{n<u_2<2n\}} a(x,\nabla u_2)\cdot \nabla u_2 V_n(u_1)T_k(u_1-u_2)- \frac{1}{n}\int_{\{n<u_2<2n\}} a(x,\nabla u_1)\cdot \nabla u_2 V_n(u_1)T_k(u_1-u_2)
		\\
		&+ \frac{1}{n}\int_{\{n<u_1<2n\}} a(x,\nabla u_2)\cdot \nabla u_1 V_n(u_2)T_k(u_1-u_2)   = \int_\Omega (h(u_1)-h(u_2)) V_n(u_1)V_n(u_2) T_k(u_1-u_2) \le 0. 
		\end{aligned}
	\end{equation} 
	 As $\mu_c\equiv 0$,  it follows from \eqref{ren2} that the second and the third term in \eqref{uniqueren0} vanish as $n\to \infty$. Moreover, recalling \eqref{cara1} and \eqref{cara2}, one obtains after an application of the H\"older inequality that
	\begin{equation}\label{uniqueren}
	\begin{aligned}
	&\frac{1}{n}\int_{\{n<u_2<2n\}} a(x,\nabla u_1)\cdot \nabla u_2 V_n(u_1)T_k(u_1-u_2)
	\\
	&\le \beta k\left(\frac{1}{n}\int_{\{n<u_1<2n\}}|\nabla u_1|^p\right)^{\frac{p-1}{p}} \left(\frac{1}{n}\int_{\{u_2<2n\}}|\nabla u_2|^p\right)^{\frac{1}{p}}  
	\\
	&\le \beta k\left(\frac{1}{\alpha n}\int_{\{n<u_1<2n\}}a(x,\nabla u_1)\cdot \nabla u_1\right)^{\frac{p-1}{p}} \left(\frac{1}{n}\int_{\{u_2<2n\}}|\nabla u_2|^p\right)^{\frac{1}{p}}, 
	\end{aligned}
	\end{equation}
	which goes to zero as $n\to \infty$ by \eqref{ren2} and by the fact that the last term in \eqref{uniqueren} is bounded in $n$. Indeed, this can be easily shown by taking $T_{2n}(u_2)$ as a test function and fixing $S=V_k$; after letting $k\to\infty$ one gains the estimate as $\gamma\le 1$.
	 
	Analogously, the last integral at the left-hand of \eqref{uniqueren0} degenerates as $n\to\infty$. Therefore, an application of the Fatou Lemma gives
	$$\int_\Omega (a(x,\nabla u_1)-a(x,\nabla u_2)\cdot \nabla T_k(u_1-u_2)\le 0,$$
	and taking $k\to \infty$ one has that $\nabla u_1= \nabla u_2$ almost everywhere in $\Omega$, which implies that $u_1=u_2$ almost everywhere in $\Omega$. This concludes the proof.
\end{proof}

\subsection{Some remarks on the assumptions on $h$}

In Section \ref{sec:measureex} we have required some restrictive assumptions on $h$: $h$ satisfied \eqref{h1ren} with $\gamma\le 1$ and it was needed strictly positive. In the remaining of the current section we just spend few words on extending the existence result relaxing these assumptions.

\subsubsection{\underline{Distributional solutions in case $\gamma>1$}}

\label{sec:distrmagg1}

As already observed, if $h$ blows up too fast in the origin (i.e. $\gamma>1$ in \eqref{h1ren}),  in general  the solution loses the weak trace in the classical Sobolev sense and this does not allow to find a renormalized solution anymore. In any case one is still able to prove that a distributional solution exists; indeed one can prove a  result  analogous to Lemma \ref{lemmastimehsing} showing that, at least locally, suitable a priori estimates hold for the approximation sequence given by the solution to \eqref{pbapproxDDO}. Here we only state the existence Theorem whose proof is detailed in \cite{ddo}.

\begin{theorem}\label{teoexistence}
	Let $a$ satisfy \eqref{cara1}-\eqref{cara3}, let $0\le \mu\in \mathcal{M}(\Omega)$ satisfy \eqref{hmu} and let $h$ satisfy \eqref{h1ren} and \eqref{h2ren}. Then there exists a distributional solution $u$ to problem \eqref{pbmain} in the sense of Definition \ref{distributional} such that: 
	$$u^{p-1}\in L^q_{\mathrm{ loc}}(\Omega)\ \;\forall\,q<\frac{N}{N-p}\quad\text{and}\quad|\nabla u|^{p-1} \in L^q_{\rm{loc}}(\Omega)\ \;\forall\,q<\frac{N}{N-1}.$$
\end{theorem}

\subsubsection{\underline{Some remarks when $h$ degenerates}}
\label{H=0}

We have required $h$ to be strictly positive in order to deal with a datum not purely diffuse, i.e. the positivity was needed to face up to the concentrated part of the measure.  Let here discuss the case in which $h$ degenerates somewhere in presence of a diffuse measure as a datum; namely we study problem \eqref{pbmain} when $h$ satisfies \eqref{h1ren} and \eqref{h2ren} and also 
\begin{equation}\label{hdegenere}
	\exists \overline{s}>0: \ h(\overline{s})=0.
\end{equation} 
Under the above assumption, we will show that there exists a solution to \eqref{pbmain} which lies in $[0,\overline{s}]$.

\medskip
As we have already observed, roughly speaking one has that solutions tend to  blow up at the support of $\mu_c$; this means that it may  lose  sense to look for bounded solutions in presence of a concentrated measure.  From a dual point of view  one can also think that, in presence of a bounded solution, one can discard the behaviour at infinity governed by the term 
$$
h(\infty)\int_{\Omega} \varphi \mu_c; 
$$ 
this fact can be view as a  nonexistence result in the same spirit of    Section \ref{sec:nonex}.  
\begin{theorem}\label{teoexistencedegenere}
	Let $a$ satisfy \eqref{cara1}-\eqref{cara3}, let $0\le \mu\in \mathcal{M}(\Omega)$ such that $\mu_c\equiv 0$ and let $h$ satisfy \eqref{h1ren}, \eqref{h2ren} and \eqref{hdegenere}. Then there exists a distributional solution $u\in  W^{1,p}_{\rm loc}(\Omega)\cap L^\infty(\Omega)$ to \eqref{pbmain} such that $||u||_{L^\infty(\Omega)}\le \overline{s}$.  If $\gamma\le 1$ then $u \in W^{1,p}_0(\Omega)$.
\end{theorem}
\begin{proof}
	Here we modify the scheme of approximation \eqref{pbapproxDDO} used so far, which is
	\begin{equation}\begin{cases}
			\displaystyle -\operatorname{div}(a(x,\nabla u_n)) = h_n(u_n)\mu_d &  \text{in}	\ \, \Omega, \\
			u_n=0 & \text{on} \ \partial \Omega.
			\label{pbapproxlimitatafinal}
	\end{cases}\end{equation}
	Let us define the function $\tilde{h}$ as follows
	\begin{equation}\label{H*}
		\tilde{h}(s)=\begin{cases}
			h(s) \quad&\text{if }\;s<\overline{s},\\
			0\quad&\text{if }\;s\geq \overline{s},
		\end{cases}
	\end{equation}
	and we consider the following problem
	\begin{equation}\begin{cases}
			\displaystyle -\operatorname{div}(a(x,\nabla \overline{u}_n)) =\tilde{h}(\overline{u}_n)\mu_{d} &  \text{in} \ \, \Omega, \\
			\overline{u}_n=0 & \text{on} \ \partial \Omega.
			\label{pbapprox*1}
		\end{cases}
	\end{equation}
	Let us highlight that $\overline{u}_n$ is a renormalized solution to \eqref{pbapproxlimitatafinal}. Hence one can take $T_k(G_{\overline{s}}(\overline{u}_n))$  ($k>0$) as a test function and $S=V_r \ (r>k)$. One gets
	$$\alpha\io|\nabla T_k(G_{\overline{s}}(\overline{u}_n))|^p \le \int_{\{r<\overline{u}_n<2r\}} |\nabla \overline{u}_n|^p T_k(G_{\overline{s}}(\overline{u}_n)),$$
	and, taking $r\to\infty$, it follows from \eqref{ren2} that 
	$$\alpha\io|\nabla T_k(G_{\overline{s}}(\overline{u}_n))|^p \le 0,$$	 
	which implies $\overline{u}_n\leq \overline{s}$ almost everywhere in $\Omega$. Hence, recalling \eqref{H*}, we conclude that $\overline{u}_n$ solves also \eqref{pbapproxlimitatafinal}. Now, since $\overline{u}_n$ is bounded in $L^\infty(\Omega)$ then one can take $(\overline{u}_n-\overline{s})\varphi$ ($0\le\varphi\in C^1_c(\Omega)$) as a test function and $S= V_r$ in the renormalized formulation of \eqref{pbapprox*1} deducing that $\overline{u}_n$ is locally bounded in $W^{1,p}(\Omega)$ with respect to $n$. Moreover, when $\gamma\le 1$, one can directly take $\overline{u}_n$ as a test function and $S= V_r$ in order to deduce that $\overline{u}_n$ is bounded in $W^{1,p}_0(\Omega)$ with respect to $n$. The passage to the limit can be carried on exactly as in the proof of Theorem \ref{teoexistence} or as in the proof of Theorem \ref{teoexrinuniqueness} for the case $\gamma\le 1$.	
\end{proof}

\section{Regularity of the solution in presence of a nonlinear operator}\label{regup}

In this section we briefly summarize some  \underline{regularity} results for the solution found in Sections \ref{sec:measureex} and \ref{sec:distrmagg1} in presence of a nonlinear operator and when the source datum is a Lebesgue function. 
\subsection{Regularity in presence of a  smooth datum}
 We start extending  part of the results of Section \ref{sec:lm} to the following quasilinear problem
\begin{equation}
\begin{cases}
\displaystyle - \Delta_p u= \frac{f}{u^\gamma} &  \text{in} \ \, \Omega, \\
u=0 & \text{on}\ \partial \Omega,
\label{pbplm}
\end{cases}
\end{equation}
where $\partial\Omega$ is smooth, $1<p<N$ and $0<m\le f\in L^\infty(\Omega)$.

\medskip 
Before stating the first regularity theorem a comment is needed: if $u$ is a weak solution to \eqref{pbplm}, i.e. $u\in W^{1,p}_0(\Omega)$ is a distributional solution to \eqref{pbplm}, then one can reason, with straightforward modifications,  as in  the proof of Lemma \ref{lemextest} (see also Remark \ref{remarkgeneralizzazione}) in order to extend the set of the test functions. Then one gets that a weak solution $u$ to \eqref{pbplm} satisfies 
\begin{equation}\label{extend_pi}
	\int_{\Omega} |\nabla u|^{p-2}\nabla u \cdot \nabla v = \int_{\Omega}fu^{-\gamma}v, \ \forall v\in W^{1,p}_0(\Omega).
\end{equation}

Next theorem provides a necessary and sufficient condition to have existence of finite energy solution to \eqref{pbplm};  this can be considered as the natural extension of Theorem \ref{teoreg} to the case of equations driven by the $p$-Laplacian.
\begin{theorem}
	Let $0<m\le f\in L^\infty(\Omega)$ then there exists a  distributional solution $u\in W^{1,p}_0(\Omega)$ to \eqref{pbplm} if and only if \begin{equation}\label{cond_gammap}\gamma<2+\frac{1}{p-1}\,.\end{equation}	 
\end{theorem}
\begin{proof}
We start proving that if $\gamma<2+\frac{1}{p-1}$ then there exists a distributional solution $u\in W^{1,p}_0(\Omega)$ to \eqref{pbplm}.

The proof will be concluded once that Sobolev estimates are obtained for the following approximating solutions 
\begin{equation}
\begin{cases}
\displaystyle - \Delta_p u_n= \frac{f}{\left(u_n+\frac{1}{n}\right)^\gamma} &  \text{in} \ \, \Omega, \\
u_n=0 & \text{on}\ \partial \Omega.
\label{pbplmapprox}
\end{cases}
\end{equation}
Indeed, once that one has $u_n$ bounded in $W^{1,p}_0(\Omega)$ with respect to $n$ the existence  of a weak solution $u$  can be carried on almost similar way to the proof of Theorem \ref{boesistenza} or Theorem \ref{esistenzah}. The only difference being the nonlinear operator which can be treat classically and also reasoning similarly to  Lemma \ref{convgradienti}  deducing that $\nabla u_n$ converges  to $\nabla u$ almost everywhere in $\Omega$ as $n\to\infty$.  

\smallskip

Firstly, if $\gamma\le 1$, we take $u_n$ as a test function in \eqref{pbplmapprox} obtaining
$$\int_\Omega |\nabla u_n|^p \le \int_\Omega  f  u_n^{1-\gamma} \le \|f\|_{L^\infty(\Omega)} \int_\Omega u_n^{1-\gamma},$$
which, after applications of the Young inequality and the Sobolev inequality, implies that $u_n$ is bounded in $W^{1,p}_0(\Omega)$ with respect to $n$. 

\medskip

Now let $\varphi_{1,p}$ be the solution to \eqref{not:phi1p} and  let $\gamma>1$.   A step by step re-adaptation of the proof  of Theorem \ref{teoexlm} (see also Remark \ref{remexlm}) gives  that
\begin{equation}a\varphi_{1,p}^t\le u_n\le b \varphi_{1,p}^t \ \ \text{in }\Omega,\label{distplap}\end{equation}
when  $t= \frac{p}{p-1+\gamma}$ and  for a suitable choice of $a<b$.
Hence, let us take $u_n$ as a test function in \eqref{pbplm}, one has
\begin{equation*}
\displaystyle \displaystyle \int_{\Omega}|\nabla u_n|^p \le \|f\|_{L^\infty(\Omega)}\int_{\Omega} a^{1-\gamma}\varphi_{1,p}^{t(1-\gamma)},
\end{equation*} 
and the right-hand of the previous is finite since $\gamma<2+\frac{1}{p-1}$ (see Remark \ref{remarksuautofunzioneplap} below). This shows that $u\in W^{1,p}_0(\Omega)$ is a solution to \eqref{pbplm}. 

\medskip

 Now we show that the assumption \eqref{cond_gammap} on $\gamma$ is necessary; let us assume by contradiction that there exists a  distributional  solution $u\in W^{1,p}_0(\Omega)$ of \eqref{pbplm} for some   $\gamma\ge 2+\frac{1}{p-1}$.
Then, as for \eqref{distplap}, one can show that $u\ge \varphi_{1,p}^{t}$ almost everywhere in $\Omega$.
Therefore, taking $u$ as test in \eqref{extend_pi}, one obtains 
\begin{equation*}
\displaystyle \int_{\Omega}|\nabla u|^p=  \int_{\Omega} \frac{f}{u^{\gamma-1}}\ge \int_{\Omega} \frac{m}{b^{\gamma-1}\varphi_{1,p}^{t(\gamma-1)}}=\infty.
\label{stimalm4plap}
\end{equation*} 
This gives the contradiction as $u\in W^{1,p}_0(\Omega)$ and the proof is concluded.
\end{proof}

One could wonder what happens when $f$ belongs to $L^m(\Omega)$ for $m>1$. In particular, as we have seen in Theorem \ref{teosharp}, the threshold to find finite energy solutions in case $p=2$
is given by
\begin{equation*}
	\label{eq:threshold_2}
	\gamma < 3-\frac{2}{m}.
\end{equation*}
One could ask how this result reads in the general case $p>1$. The next example (see \cite{maop}) suggests that the threshold to have finite energy solutions in this latter case is: 
\begin{equation}
	\label{eq:threshold_3}
	\gamma < \frac{2p-1}{p-1} - \left(\frac{p}{p-1}\right) \frac{1}{m} = \frac{\left(2-\frac{1}{m}\right)p -1}{p-1}.
\end{equation}
\begin{example}
	Let $p,\gamma>1$ and observe that, for $\eta>0$, $u = (1-|x|^{2})^{\eta}$ satisfies
	\begin{equation*}
		\begin{cases}
			\displaystyle -\Delta_p u= \frac{f(x)}{u^{\gamma}}   &\text{in }\ B_{1}(0), \\
			u=0 & \text{on }\ \partial B_{1}(0), 
		\end{cases}
	\end{equation*}
	where
	$$
	f(x)\sim \frac{1}{(1-|x|^{2})^{p-\eta(\gamma+p-1)}},
	$$ 
	which belongs $f\in L^m(\Omega)$ ($m\geq 1$) once one requires
	$$
	\eta >\frac{p-\frac{1}{m}}{\gamma+p-1}.
	$$
	Let us stress that  $u^{\frac{\gamma-1+p}{p}}\in W^{1,p}_0(\Omega)$ but $u\in W^{1,p}_0(\Omega)$ only for $\eta>\frac{p-1}{p}$ and, in particular, if
	$$
	\frac{p-\frac{1}{m}}{\gamma+p-1} > \frac{p-1}{p},
	$$
	then \eqref{eq:threshold_3} holds.
\end{example}

To conclude this section, we deal with the regularity of the solution found in Theorem \ref{teoexrinuniqueness} when $f$ is a possibly  {\it unbounded} function.  Let explicitly observe that here it is not necessary the regularity of the domain. Hence we deal with 
\begin{equation}
\label{pbapp}
\begin{cases}
\dis -\diver (a(x,\nabla u))= h(u)f & \text{in}\;\Omega,\\
u\ge 0 & \text{in}\;\Omega, \\
u=0 & \text{on}\;\partial\Omega,
\end{cases}
\end{equation}
where $a$ satisfies \eqref{cara1}-\eqref{cara3} and $f\in L^m(\Omega)$ is nonnegative. We are also concerned to the regularity of solution in connection with the rate to which $h$ possibly degenerates. Hence, $h$ satisfies once again \eqref{h1ren} and \eqref{h2bis} which we recall for the sake of completeness:

\begin{equation}
	\exists\;\gamma \ge 0, \ c_1,s_1>0: \  h(s) \le \frac{c_1}{s^\gamma}\;\text{ if }\;s<s_1,
	\label{h1reg}
\end{equation} 
and 
\begin{equation}\label{h2reg}
	\exists\;\theta \ge 0, \ c_2,s_2>0: \ h(s) \le \frac{c_2}{s^\theta}\; \ \text{if}\ \;s>s_2.
\end{equation}

\medskip

The strict connection between the datum $f\in L^m(\Omega)$ and the function $h$ is specified  by the following value: 
$$\overline{m}:= 
\begin{cases}
\left(\frac{p^*}{1-\theta}\right)' \ \ &\text{if }\theta<1,
\\
1 \ \ \ &\text{if }\theta\ge1,
\end{cases}
$$
 which, in the next theorem, provides the threshold to have finite energy solutions, at least far away from zero. 
Now recalling that $\sigma:=\max(1,\gamma)$, we state and prove the  following theorem:

\begin{theorem}
	Let $a$ satisfy \eqref{cara1}-\eqref{cara3}, let $0\le f\in L^m(\Omega)$ and let $h$ satisfy \eqref{h1reg} and \eqref{h2reg}. Then there exists a solution $u$ to \eqref{pbapp} such that for any $\varepsilon>0$:
	\begin{itemize}
		\item[i)] if $m>\frac{N}{p}$ then $G_\varepsilon(u) \in W^{1,p}_0(\Omega)\cap L^\infty(\Omega);$
		\item[ii)] if $ \overline{m}\le m<\frac{N}{p}$ then $G_\varepsilon(u) \in W^{1,p}_0(\Omega)\cap L^{\frac{Nm(\theta -1 + p)}{N-mp}}(\Omega);$
		\item[iii)] if $ 1\le m <\overline{m}$ then $|\nabla G_\varepsilon(u)|^{\frac{Nm(\theta-1+p)}{N-m(1-\theta)}} \in L^1(\Omega)$ and $G^{\frac{Nm(\theta-1+p)}{N-mp}}_\varepsilon(u)\in L^1(\Omega)$.
	\end{itemize} 
	Moreover $T_k^{\frac{\sigma-1+p}{p}}(u) \in W^{1,p}_0(\Omega)$ and $T_k(u) \in  W^{1,p}_{\rm loc}(\Omega)$ for any $k>0$.
\end{theorem}
\begin{proof}
	Let us consider the following scheme of approximation:
	\begin{equation}
	\label{pbapp_tronc}
	\begin{cases}
	\dis -\diver (a(x,\nabla u_n))= h_n(u_n)f_n & \text{in}\;\Omega,\\
	u_n=0 & \text{on}\;\partial\Omega,
	\end{cases}
	\end{equation}
	where $h_n(s):= T_n(h(s))$ and $f_n:= T_n(f)$. The existence of such nonnegative $u_n\in W^{1,p}_0(\Omega) \cap L^\infty(\Omega)$ follows from \cite{ll}. With some simplifications,  reasoning as in  Section \ref{sec:measureex}, it follows that $u_n$ converges, as $n\to\infty$, to a function $u$ which is a solution to \eqref{pbapp}. Therefore here we will just need to show suitable a priori estimates on $u_n$ and this will be sufficient to conclude the proof.
	
	\medskip

	\textbf{Proof of i).} We start taking $G_\varepsilon(u_n)$ as a test function in \eqref{pbapp_tronc} obtaining 
	\begin{equation}\label{appi}
	\alpha \int_\Omega |\nabla G_\varepsilon(u_n)|^p \le \int_\Omega a(x,\nabla u_n)\cdot \nabla G_\varepsilon(u_n) = \int_\Omega h_n(u_n) f_n G_\varepsilon(u_n) \le \sup_{s\in (\varepsilon,\infty)} h(s) \int_\Omega f_n G_\varepsilon(u_n).
	\end{equation}
	The previous estimate gives the possibility to reason as in Th\'eor\`eme $4.2$ of \cite{st} in order to the deduce that $G_\varepsilon(u_n)$ is bounded in $L^\infty(\Omega)$ with respect to $n$. Gathering this information with estimate \eqref{appi} then one has that $G_\varepsilon(u_n)$ is also bounded in $W^{1,p}_0(\Omega)$ with respect to $n$. This concludes the proof of i).

	\medskip

	\textbf{Proof of ii).} Once again we take $G_\varepsilon(u_n)$ as a test function in \eqref{pbapp_tronc} deducing
	\begin{equation}\label{appii}
	\begin{aligned}
	\alpha \int_\Omega |\nabla G_\varepsilon(u_n)|^p \le \int_\Omega a(x,\nabla u_n)\cdot \nabla G_\varepsilon(u_n) &= \int_\Omega h_n(u_n) f_n G_\varepsilon(u_n) 
	\\
	&\le \sup_{s\in (\varepsilon,s_2]} [h(s)s] ||f||_{L^1(\Omega)} + c_2\int_{\{u_n> s_2\}}  f_n G_\varepsilon^{1-\theta}(u_n).
	\end{aligned}
	\end{equation}
	Let firstly observe that if $\theta \ge 1$ then one has that $G_\varepsilon(u_n)$ is bounded in $W^{1,p}_0(\Omega)$ with respect to $n$ by simply requiring $f\in L^1(\Omega)$. 
	Otherwise, if $\theta<1$, from the H\"older inequality one has
	\begin{equation}\label{appii2}
	\int_{\{u_n> s_2\}}  f_n  G_\varepsilon^{1-\theta}(u_n) \le ||f||_{L^{\left(\frac{p^*}{1-\theta}\right)'}} \left(\int_\Omega  G_\varepsilon^{p^*}(u_n)\right)^{\frac{1-\theta}{p^*}}. 
	\end{equation}
	Hence, gathering \eqref{appii2} with \eqref{appii}, applying the Sobolev inequality, one obtains
	\begin{equation*}\label{appii3}
	\int_\Omega |\nabla G_\varepsilon(u_n)|^p \le \frac{1}{\alpha}\sup_{s\in (\varepsilon,s_2]} [h(s)s] ||f||_{L^1(\Omega)}  + \frac{c_2}{\alpha} ||f||_{L^{\left(\frac{p^*}{1-\theta}\right)'}} \mathcal{S}_p \left(\int_\Omega |\nabla  G_\varepsilon(u_n)|^p \right)^{\frac{1-\theta}{p}}.
	\end{equation*}
	It follows from an application of the Young inequality that the previous  implies that $G_\varepsilon(u_n)$ is bounded in $W^{1,p}_0(\Omega)$ with respect to $n$.

	Let us focus now on the Lebesgue regularity. If   $m=\overline{m}$ then the result follows by Sobolev embedding. Now assume   $m>\overline{m}$. In this case one takes $G^{\eta}_\varepsilon(u_n)$ for some $\eta>\max(1,\theta)$. Hence, applying the Sobolev inequality, one obtains 
	\begin{equation}\label{appii4}
	\begin{aligned}
	\frac{\alpha\eta}{\mathcal{S}_p^p} \left(\frac{p}{\eta-1+p}\right)^p &\left(\int_\Omega G^{\frac{(\eta-1+p)p^*}{p}}_\varepsilon(u_n) \right)^{\frac{p}{p^*}} \le \alpha\eta \left(\frac{p}{\eta-1+p}\right)^p \int_\Omega |\nabla G^{\frac{\eta-1+p}{p}}_\varepsilon(u_n)|^p 
	\\
	&\le  \sup_{s\in (\varepsilon,s_2]} [h(s)s^\eta] ||f||_{L^1(\Omega)} + c_2\int_{\{u_n> s_2\}}  f_n G_\varepsilon^{\eta-\theta}(u_n)
	\\
	&\le  \sup_{s\in (\varepsilon,s_2]} [h(s)s^\eta] ||f||_{L^1(\Omega)} + c_2 ||f||_{L^m(\Omega)} \left(\int_{\{u_n> s_2\}}  G_\varepsilon^{(\eta-\theta)m'}(u_n)\right)^{\frac{1}{m'}}.
	\end{aligned}
	\end{equation}
	If we require 
	$$\eta = \frac{\theta m(N-p) + (m-1)(p-1)N}{N-mp}$$
	then 
	$$\frac{(\eta-1+p)p^*}{p} = (\eta-\theta)m' = \frac{Nm(\theta -1 + p)}{N-mp}.$$
	Since it follows from $m<\frac{N}{p}$ that $\frac{p}{p^*}>\frac{1}{m'}$, then a simple application of the Young inequality in \eqref{appii4} concludes the proof for this case.

	\medskip

	\textbf{Proof of iii).}  Let us first observe that, in this case, one necessarily has $\theta<1$.

	Let $0<\delta<\frac{1}{n}$ and $\theta\le\eta<1$ and let us  take $(G_\varepsilon(u_n)+\delta)^{\eta} - \delta^{\eta}$ as a test function in \eqref{pbapp_tronc}, yielding to
	\begin{equation}\label{appiii1}
	\displaystyle \eta\alpha\int_{\Omega} |\nabla G_\varepsilon(u_n)|^p(G_\varepsilon(u_n)+\delta)^{\eta-1} \le \sup_{s\in (\varepsilon,s_2]} [h(s)(s+\delta)^\eta] ||f||_{L^1(\Omega)} + \int_{\{u_n>s_2\}} f_n (G_\varepsilon(u_n)+\delta)^{\eta-\theta}.
	\end{equation}
	For the left-hand of \eqref{appiii1}, one has
	\begin{equation*}
	\begin{aligned}
	\displaystyle \int_{\Omega} |\nabla G_\varepsilon(u_n)|^p(G_\varepsilon(u_n)+\delta)^{\eta-1} &= \left(\frac{p}{\eta-1+p}\right)^p\int_{\Omega} |\nabla G_\varepsilon(u_n)+\delta)^\frac{\eta-1+p}{p}-\delta^\frac{\eta-1+p}{p}|^p
	\\
	&\ge \left(\frac{p}{\mathcal{S}_p(\eta-1+p)}\right)^p\left(\int_{\Omega} |(G_\varepsilon(u_n)+\delta)^\frac{\eta-1+p}{p}-\delta^\frac{\eta-1+p}{p}|^{p^*}\right)^{\frac{p}{p^*}}.
	\end{aligned}
	\end{equation*}
	Thus, gathering the above inequalities and letting $\delta \to 0^+$  one gets
	\begin{equation}\label{appiii2}
	\displaystyle \eta \alpha\left(\frac{p}{\mathcal{S}_p(\eta-1+p)}\right)^p\left(\int_{\Omega} G_\varepsilon^\frac{p^*(\eta-1+p)}{p}(u_n)\right)^{\frac{p}{p^*}} \le \sup_{s\in (\varepsilon,s_2]} [h(s)s^\eta] ||f||_{L^1(\Omega)} + \int_{\Omega} f_n G_\varepsilon^{\eta-\theta}(u_n).
	\end{equation}	
	If $m=1$ then we fix $\eta=\theta$ obtaining the previous that $G_\varepsilon^\frac{N(\theta-1 +p)}{N-p}(u_n)$ is bounded in $L^{1}(\Omega)$.
	\\
	Otherwise, if $m>1$, we apply the H\"older  inequality on the second term of the right-hand of \eqref{appiii2}, yielding to 
	\begin{equation*}\label{appiii3}
	\int_{\Omega} f_n G_\varepsilon^{\eta-\theta}(u_n) \le ||f||_{L^m(\Omega)} \left(\int_{\Omega} G_\varepsilon^{(\eta-\theta)m'}(u_n)\right)^{\frac{1}{m'}}. 
	\end{equation*}
	Reasoning as in case ii), that is one can require $\frac{p^*(\eta-1+p)}{p} = (\eta-\theta)m'$, which gives that $G_{\varepsilon}^{\frac{Nm(\theta-1+p)}{N-mp}}(u_n)$ is bounded in $L^1(\Omega)$. 
	\\ Thus, for $q<p$, one has
	\begin{align*}
	\int_{\Omega} |\nabla G_\varepsilon(u_n)|^q = \int_{\Omega} \frac{|\nabla 
		G_\varepsilon(u_n)|^q}{(G_\varepsilon(u_n)+\delta)^{\frac{q(1-\eta)}{p}}}(G_\varepsilon(u_n)+\delta)^{\frac{q(1-\eta)}{p}} &\le \left(\int_{\Omega} \frac{|\nabla G_\varepsilon(u_n)|^p}{(G_\varepsilon(u_n)+\delta)^{(1-\eta)}}\right)^\frac{q}{p} \left( \int_{\Omega} (G_\varepsilon(u_n)+\delta)^{\frac{(1-\eta)q}{p-q}}\right)^{\frac{p-q}{p}} 
	\\
	& \le C\left( \int_{\Omega} (G_\varepsilon(u_n)+\delta)^{\frac{(1-\eta)q}{p-q}}\right)^{\frac{p-q}{p}},
	\end{align*}
	which is bounded with respect to $n$ if $q=\frac{Nm(\theta-1+p)}{N-m(1-\theta)}$.

	\medskip

	In order to conclude the proof we need estimates on the truncations of $u_n$. We take $T_k^\sigma(u_n)$ as a test function in \eqref{pbapp_tronc} deducing that
	\begin{equation*}
	\alpha\sigma \left(\frac{p}{\sigma-1+p}\right)^p \int_\Omega |\nabla T^{\frac{\sigma-1+p}{p}}_k(u_n)|^p 
	\le  c_1k^{\sigma-\gamma}||f||_{L^1(\Omega)} + k\sup_{s\in (s_1,\infty)} h(s) ||f||_{L^1(\Omega)},
	\end{equation*}
	which guarantees the global estimate in $T^{\frac{\sigma-1+p}{p}}_k(u_n)$. Concerning the local estimate we take $(T_k(u_n)-k)\varphi^p$ as a test function where $0\le \varphi \in C^1_c(\Omega)$, and after an application of the Young inequality, one  yields to
	\begin{equation*}
	\int_\Omega |\nabla T_k(u_n)|^p\varphi^p \le C,
	\end{equation*}
	for some constant independent of $n$. This concludes the proof.
\end{proof}

\section{Uniqueness of the distributional solution in the semilinear case}\label{unilinear}
 
In Theorem \ref{boca} we have proven that there is only one weak solution to \eqref{pbgeneralh} provided $h$ is non-increasing. The core of the  proof consisted in  showing that distributional solutions to \eqref{pbgeneralh} are unique in the class $H^1_0(\Omega)$;  in Theorem \ref{teounique}, we have also seen that  uniqueness   holds  for  renormalized solutions if $\gamma \le 1$ (i.e. $h(s) \approx s^{-{\gamma}}$ near zero),  again in presence of a non-increasing $h$ and for a purely diffuse measure as datum.

\medskip

Then, a natural question is whether previous  \underline{uniqueness} results can be  extended to the case of  solutions having  only local finite energy or in the case   \underline{$\gamma>1$}, even in presence of a general measure as datum. 

\medskip

A second interest relies on the possibility of presenting an unified discussion independent on the value of $\gamma$; as we will see, we fix a notion of solution which is fairly consistent with the others presented so far and which allows to deduce uniqueness.

\medskip

As some of the results strongly relies on classical linear elliptic regularity theory,  we restrict ourselves to the case of a  linear principal operator in a smooth domain $\Omega$. Most of the content of this section can be found in \cite{v} and \cite{OP}.

\medskip

 We deal with the following problem  
\begin{equation}
	\begin{cases}
		\displaystyle -\operatorname{div}(A(x)\nabla u) = h(u)\mu &  \text{in}\ \Omega, \\
		u=0 & \text{on}\ \partial \Omega,
		\label{op2_pbgeneral}
	\end{cases}
\end{equation}
where $\Omega$  is an open bounded subset of $\mathbb{R}^N$ ($N\ge 2$) with a smooth boundary and $A$ is such that:
\begin{equation}
A\in C^{0,1}(\overline{\Omega}): \ \exists \alpha,\beta >0  \ \ A(x)\xi \cdot \xi \ge \alpha|\xi|^2, \ |A(x)|\le \beta.\label{lipA}
\end{equation}   
for every $\xi$ in $\mathbb{R}^N$,  for every $x$ in $\Omega$,  and for $\alpha,\beta>0$.  The function $h:[0,\infty) \mapsto [0,\infty]$ is continuous, finite outside the origin and satisfying \eqref{h2}.

\medskip

  Let us set  the notion of  distributional solution to \eqref{op2_pbgeneral} in this context:
\begin{defin}\label{defunique}
	A nonnegative function $u\in L^1(\Omega)\cap W^{1,1}_{\textrm{loc}}(\Omega)$ is a {\it distributional solution} to \eqref{op2_pbgeneral} provided $h(u)\in L^{1}_{\textrm{loc}}(\Omega, \mu_d)$ and if
		\begin{equation}
		\lim_{\varepsilon\to 0^+} \frac{1}{\varepsilon} \int_{\Omega_\varepsilon}u=0,
		\label{op2_weakdef2}
	\end{equation}	
	and	
	\begin{equation} \displaystyle \int_{\Omega}A(x)\nabla u \cdot \nabla \varphi =\int_{\Omega} h(u)\varphi \mu_d + h(\infty)\int_{\Omega} \varphi \mu_c,\ \ \ \forall \varphi \in C^1_c(\Omega).\label{op2_weakdef3}\end{equation}	
	\label{op2_weakdef}
\end{defin}
\begin{remark}
	The main novelty in Definition \ref{defunique} with respect, for instance, to Definition \ref{distributional} is how the boundary datum is intended. 
	
 	Condition \eqref{troncate} is replaced by \eqref{op2_weakdef2}. This allows to uncouple the definition from the value of $\gamma$ and to assume, in general, no trace Sobolev assumption on the solution. 	Indeed, as we have widely seen so far, our solutions should not belong in general to $W^{1,1}_{0}(\Omega)$.

	However let finally underline that \eqref{op2_weakdef2}	is weaker than  the classical sense of traces for functions in  $W^{1,1}_{0}(\Omega)$ (see for instance \cite{sel,afp}). 
\triang \end{remark}
\begin{remark}
	Let us explicitly stress that, if $h$ satisfies \eqref{h1ren} and \eqref{h2ren} and $\mu \in \mathcal{M}(\Omega)$ is nonnegative, then the existence of a distributional solution $u$ in the sense of Definition \ref{distributional} follows from Theorem \ref{teoexistence}.

	Let us stress that $u$ is also a solution to \eqref{op2_pbgeneral} in the sense of Definition \ref{op2_weakdef}. Indeed, reasoning as in Theorem \ref{teoexistence}, one can show that $G_k(u)\in W^{1,q}_0(\Omega)$ for  any $q<\frac{N}{N-1}$ and for any $k>0$. One can be convinced by testing with $T_{r}(G_{k}(u_{n}))$, $r>0$,  the weak formulation of the  approximation scheme leading to 
	$$\displaystyle \int_\Omega |\nabla T_r(G_k(u_n))|^2\le Cr,\ \ \ \text{for any}\ r>0\,,$$
	which implies that  $G_k(u)\in W^{1,q}_0(\Omega)$ for $q<\frac{N}{N-1}$ for any $k>0$ and that also gives that $u\in L^1(\Omega)$.
	Therefore, in order to show that $u$ is a solution in the sense Definition \ref{op2_weakdef}, we are left to show that \eqref{op2_weakdef2} holds.
	Observe first that, if $\gamma\leq 1$, we have that $u\in W^{1,1}_{0}(\Omega)$ and the proof is complete. 
	Otherwise, if $\gamma>1$ and as $u = T_1(u) +G_1(u)$, one can apply the H\"older inequality yielding to 
	$$
	\begin{array}{l}
	\displaystyle \frac{1}{\varepsilon}\int_{\Omega_{\varepsilon}} u \leq \varepsilon^{\frac{1-\gamma}{\gamma+1}}\left( \frac{1}{\varepsilon}\int_{\Omega_{\varepsilon}} T_{1}^{\frac{\gamma+1}{2}}(u)\right)^{\frac{2}{\gamma+1}}|\Omega_{\varepsilon}|^{\frac{\gamma-1}{\gamma+1}} + \frac{1}{\varepsilon} \int_{\Omega_{\varepsilon}} G_{1}(u) \\ \\
	\displaystyle \leq C\left(\frac{1}{\varepsilon}\int_{\Omega_{\varepsilon}} T_{1}^{\frac{\gamma+1}{2}}(u)\right)^{\frac{2}{\gamma+1}}+  \frac{1}{\varepsilon}\int_{\Omega_{\varepsilon}} G_{1}(u)\stackrel{{\varepsilon}\to0}{\longrightarrow} 0
	\end{array}
	$$
	which gives \eqref{op2_weakdef2} since both $T_{1}^{\frac{\gamma+1}{2}}(u),G_1(u)\in W^{1,1}_0(\Omega)$. 
\triang \end{remark}
	  The main theorem of this section is the following:
\begin{theorem}\label{uniquenessmain}
	Let $A$ satisfy \eqref{lipA}, let $h$ be a non-increasing function satisfying  \eqref{h2ren},  and let $0\le \mu\in \mathcal{M}(\Omega)$. Then there is at most one distributional solution to \eqref{op2_pbgeneral} in the sense of Definition \ref{defunique}.
\end{theorem}
\begin{remark}
	  Let us stress that, without loss of generality, we can assume $\mu_d\not\equiv 0$. Otherwise, one can deduce uniqueness as for the linear case (see Section \ref{veron} below). 
\triang \end{remark}

\subsection{The linear case}
\label{veron}

  In order to show Theorem \ref{uniquenessmain}, we need to present some insights for the linear case,  i.e. $h\equiv 1$. 

Let us consider the following problem   
\begin{equation}
	\begin{cases}
		\displaystyle -\operatorname{div} (A(x)\nabla u)= \mu &  \text{in}\ \Omega, \\
		u=0 & \text{on}\ \partial \Omega,
	\end{cases}\label{pbVeron}
\end{equation}
where $\mu$ belongs to $\mathcal{M}(\Omega,d)$. The matrix $A$ satisfies assumption \eqref{lipA}  and, from here on, $A^*$ denotes its adjoint matrix. 

  Now let us introduce the concept of \underline{very weak solution} to \eqref{pbVeron}: 
\begin{defin}\label{verondefin}
	A function $u\in L^1(\Omega)$ is a {\it very weak solution} to \eqref{pbVeron} if	
	\begin{equation*} \displaystyle -\int_{\Omega}u \operatorname{div}(A^*(x)\nabla \varphi) = \int_{\Omega} \varphi \mu,\label{veryweakdef1}\end{equation*}
	
	\noindent for every  $\varphi \in C^1_0(\overline{\Omega})$ such that  $\operatorname{div}(A^*(x)\nabla \varphi)\in L^\infty(\Omega)$.
	\label{veryweakdef}
\end{defin}

  Next existence and uniqueness result can be retrieved in  \cite[Theorem $2.9$]{v}. 

\begin{theorem} 
	Let $A$ satisfy \eqref{lipA} and let $\mu\in\mathcal{M}(\Omega,d)$. Then there exists a unique very weak solution to problem \eqref{pbVeron}. 
\end{theorem}

Moreover, \cite[Corollary $2.8$]{v} gives the following local estimate:

\begin{lemma}\label{veroncomplocale}
	Let $A$ satisfy \eqref{lipA} and let $u, f\in L^1_{\rm{loc}}(\Omega)$ such that 
	\begin{equation*}
		\displaystyle  -\int_{\Omega}u \operatorname{div}(A^*(x)\nabla \varphi) = \int_{\Omega} f\varphi,
	\end{equation*}
	\noindent for every  $\varphi \in C^1_c(\Omega)$ such that $\operatorname{div}(A^*(x)\nabla \varphi)\in L^\infty(\Omega)$. Then  it holds 
	\begin{equation*}
		||u||_{W^{1,q}(G)} \le C (||f||_{L^1(G')} + ||u||_{L^1(G')}),
	\end{equation*}	
	for every $q< \frac{N}{N-1}$ and for every open subsets $G \subset \subset G' \subset \subset \Omega$.
\end{lemma}

To show uniqueness of solutions to \eqref{op2_pbgeneral}, we also need the following regularity result:
\begin{lemma}\label{TroncataH1loc}
Let $u$ be a very weak solution to \eqref{pbVeron} then $T_k(u)\in H^1_{\rm{loc}}(\Omega)$ for any $k>0$. Moreover, for any $\omega\subset\subset \Omega$,  one has 
	\begin{gather}\label{dopo}
		\int_{\omega}|\nabla T_k(u)|^2 \le Ck,\ \ \text{ for any $k>0$}. 
	\end{gather}
\end{lemma}
\begin{proof} 
	Let us consider the following problem
	\begin{equation}
		\begin{cases}
			\displaystyle -\operatorname{div} (A(x)\nabla u_n)= f_n &  \text{in}\ \Omega, \\
			u_n=0 & \text{on}\ \partial \Omega,
		\end{cases}\label{pbnVeron}
	\end{equation}
	which is the approximation scheme used in Theorem $2.9$ of \cite{v}. Here $f_n$ are smooth functions which are bounded in $L^1(\Omega,d)$ and which converge to $\mu$ in the following sense:
	$$\displaystyle \lim_{n\to \infty}\int_{\Omega} f_n \phi = \int_{\Omega} \phi \mu, \ \ \ \forall \phi: \frac{\phi}{d} \in C(\overline{\Omega}).$$ 
	The author also shows that $u_n$ is bounded in $L^1(\Omega)$ and it converges almost everywhere, as $n\to\infty$, to a solution of problem \eqref{pbVeron}.
	
	\medskip
	
	We consider $\varphi\in C^1_c(\Omega)$ such that $0\le\varphi\le 1$ and $\varphi = 1$ on  a set $\omega \subset\subset\Omega$; testing weak formulation of \eqref{pbnVeron} with $T_k(u_n)\varphi$ ($k>0$) and applying \eqref{lipA}, one has 
	\begin{equation}
	\begin{aligned}\label{AppendiceTroncateinH1loc1}
		\alpha\int_{\Omega} |\nabla T_k(u_n)|^2\varphi \le \int_{\Omega}A(x)\nabla u_n\cdot  \nabla T_k(u_n)   \varphi
		=  \int_{\Omega} T_k(u_n) f_n  \varphi - \int_{\Omega} T_k(u_n) A(x)\nabla u_n\cdot \nabla \varphi. 
	\end{aligned}
	\end{equation}
	For the second term on the right-hand of \eqref{AppendiceTroncateinH1loc1} we have
	\begin{align*}
		\int_{\Omega} T_k(u_n) A(x)\nabla u_n\cdot \nabla \varphi &= - \int_{\Omega} T_k(u_n) u_n\operatorname{div} (A^*(x)\nabla \varphi) - \int_{\Omega} T_k(u_n) A^*(x)\nabla \varphi\cdot \nabla T_k(u_n) 
		\\ \nonumber
		&=- \int_{\Omega} T_k(u_n) u_n \operatorname{div} (A^*(x)\nabla \varphi) - \frac{1}{2}\int_{\Omega} A^*(x)\nabla \varphi\cdot \nabla [T_k(u_n)]^2 
		\\ \nonumber
		&= - \int_{\Omega} T_k(u_n) u_n\operatorname{div} (A^*(x)\nabla \varphi) + \frac{1}{2}\int_{\Omega} \operatorname{div}(A^*(x)\nabla \varphi) [T_k(u_n)]^2   
		\\ \nonumber
		&= - \int_{\Omega} T_k(u_n) (u_n- \frac{1}{2}T_k(u_n))\operatorname{div} (A^*(x)\nabla \varphi) 
		\\ \nonumber
		&\ge - k \int_{\Omega} |u_n| |\operatorname{div} (A^*(x)\nabla \varphi)|.
	\end{align*}
	Then the previous estimate and \eqref{AppendiceTroncateinH1loc1} imply that  \eqref{dopo} holds
thanks also to the weak lower semicontinuity of the norm and to the fact that $f_n$ is  locally bounded in $L^1 (\Omega)$.
\end{proof}

We also need a  Kato local  type of inequality in the same spirit of \cite{bp}.
\begin{lemma}\label{Katoinequalitylocal}
	 Let $\mu \in \mathcal{M}(\Omega,d)$ be diffuse with respect to the $2$-capacity and let $u$ be the very weak solution to \eqref{pbVeron}.  Then
	\begin{equation*}
		\displaystyle - \int_{\Omega} u^+ \operatorname{div}(A^*(x)\nabla \varphi) \le \int_{\{u\ge 0\}} \varphi \mu,
	\end{equation*}	
	for every $\varphi\in C^1_c(\Omega)$ such that $\operatorname{div}(A^*(x)\nabla \varphi) \in L^\infty(\Omega)$.
\end{lemma}
\begin{proof}
	Let consider   again \eqref{pbnVeron};   observe that, as $d\mu$ is diffuse with respect to the $2$-capacity and thanks to Theorem \ref{diffuse}, one has $d\mu=g-\operatorname{div}(G)$ where $g\in L^1(\Omega)$ and $G\in L^2(\Omega)^N$. Then it follows from Proposition \ref{approssimazionediffuse} that there exists $f_n$ such that $d f_n = g_n - \operatorname{div}(G_n)$ where $g_n$ weakly converges in $L^1(\Omega)$ to $g$ and $G_n$ strongly converges in $L^2(\Omega)^N$ to $G$.  
	
	\medskip
	
	Now let $\Phi :\mathbb{R} \to \mathbb{R}$ be a $C^2$-convex function with $0\le \Phi' \le 1$ and $\Phi''$ with compact support such that $\Phi(0)=0$, and let $0\le \varphi\in C^1_c(\Omega)$ such that $\operatorname{div}(A^*(x)\nabla \varphi) \in L^\infty(\Omega)$.
	
	We have 
	\begin{gather*}
		\displaystyle - \int_{\Omega} \Phi(u_n)\operatorname{div}(A^*(x)\nabla \varphi) =  \int_{\Omega} \nabla u_n \cdot A^*(x)\nabla \varphi \Phi'(u_n) \le \int_{\Omega} \varphi \Phi'(u_n) f_n.
	\end{gather*}
	Now we want to  take the limit as  $n\to\infty$ in  
	\begin{gather}\label{katolocale1}
		\displaystyle - \int_{\Omega} \Phi(u_n)\operatorname{div}(A^*(x)\nabla \varphi) \le \int_{\Omega} \varphi \Phi'(u_n) f_n.
	\end{gather}
	It follows from Lemma \ref{veroncomplocale} that  $u_n$ converges locally, at least in $L^1(\Omega)$; as
	$\Phi''$ has compact support, this is sufficient to take $n\to\infty$  by generalized dominated convergence Lebesgue Theorem on the left-hand of \eqref{katolocale1}. 
	
	For the right-hand it is sufficient to note that $\Phi'(u_n)$ converges to $\Phi'(u)$ in $L^\infty(\Omega)$ $*$-weak and almost everywhere; then, thanks to the structure of
	$d f_n$, to pass to the limit in the right-hand of \eqref{katolocale1} one only needs to check that $\Phi'(u_n)$
	is bounded in $H^1_{\textrm{loc}}(\Omega)$. 
	
	Hence observe that
	$$\nabla \Phi'(u_n)= \Phi''(u_n) \nabla u_n = \Phi''(u_n)\nabla T_Q(u_n),$$
	for some $Q>0$ since $\Phi''$ has compact support. 
	Then, using \eqref{dopo}, one deduces
	\begin{gather*}\label{katolocale2}
		\displaystyle - \int_{\Omega} \Phi(u)\operatorname{div}(A^*(x)\nabla \varphi) \le \int_{\Omega} \varphi \Phi'(u)\mu.
	\end{gather*}
	
	\medskip
	
	Finally one can apply the previous to a sequence of regular convex functions $\Phi_\varepsilon(t)$ with $\Phi_\varepsilon(t)=t$ on $t\ge 0$, $|\Phi_\varepsilon(t)|\le \varepsilon$ on $t<0$  and $\Phi''_\varepsilon$ with compact support.   Taking the limit as $\varepsilon\to 0^+$, it holds 
	\begin{gather*}
		\displaystyle - \int_{\Omega} u^+\operatorname{div}(A^*(x)\nabla \varphi) \le \int_{\{u\ge 0\}} \varphi \mu,
	\end{gather*}
	that concludes the proof.
\end{proof}

We conclude this section with a regularity result which will be useful in the sequel.

\begin{lemma}\label{M1respectdelta}
	Let $A$ satisfy \eqref{lipA} and let $u\in L^1(\Omega) \cap W^{1,1}_{\rm{loc}}(\Omega)$ be such that $-\operatorname{div}(A(x)\nabla u) = \mu$ in the sense of distributions for some nonnegative local measure $\mu$. Then $\mu\in \mathcal{M}(\Omega,d)$.	
\end{lemma}
\begin{proof}
	
Let $\xi$ be the smooth solution to 
	\begin{equation}\label{xi}
		\begin{cases}
			\displaystyle -\operatorname{div}(A^*(x)\nabla \xi) = 1 &  \text{in}\ \Omega, \\
			\xi=0 & \text{on}\ \partial \Omega.
		\end{cases}
	\end{equation}
	
	Then, in order to conclude the proof, one needs to show that 
	$$
	\int_\Omega \xi \mu \leq C,
	$$
	as, as a consequence of Hopf's Lemma \ref{hopf}, one has $\xi\ge Cd$ on $\Omega$. 
	
	\medskip
	
	 Let $\Phi$ be a convex smooth function with $\Phi'$ bounded and which vanishes in a neighborhood of $0$.  As an example,  for a   $k>0$, a possible choice is to consider $\Phi$ as  a convex smooth  function that agrees with  $|G_{k}(s)|$ for every $s$ but $|s|\in [k,k+1]$. 
	
	Now let us  consider $\varphi_{n}:= \frac{1}{n}{\Phi (n\xi(x))}$. Then it is easy to check that $\varphi_{n}$ has compact support in $\Omega$ and that $\varphi_{n}$ converges to $\xi$ almost everywhere in $\Omega$. 
	
	One also has that, in the sense of distributions, it holds
	\begin{align*}
		-\operatorname{div}(A^*(x)\nabla \varphi_{n}) &= -\operatorname{div}(A^*(x)\nabla \xi)\Phi'(n\xi) - A^*(x)\nabla \xi\cdot \nabla \xi \Phi''(n\xi) \\
		&\leq -\operatorname{div}(A^*(x)\nabla \xi)\Phi'(n\xi)=\Phi'(n\xi),
	\end{align*}		
	thanks to using the convexity of $\Phi$.
	Then 
	$$
	\int_\Omega \varphi_{n}\mu \le -\int_\Omega  u \operatorname{div}(A^*(x)\nabla\varphi_{n})\ \leq \|\Phi'\|_{L^{\infty}(\R)}\int_\Omega u\leq C,
	$$
	where we can apply the Fatou Lemma  as $C$ does not depend on $n$. This concludes the proof.  
\end{proof}

\subsection{Uniqueness of the distributional solution}	
\label{op2_uniqueness_distributional}

We are ready to prove the uniqueness theorem.
\begin{proof}[Proof of Theorem \ref{uniquenessmain}]
	Let $u$ be a solution to \eqref{op2_pbgeneral} in the sense of Definition \ref{defunique}. Then observe that an application of Lemma \ref{M1respectdelta} gives that $h(u)\mu_d + h(\infty)\mu_c\in \mathcal{M}(\Omega,d)$.  As $\mu_d$ is diffuse with respect to the $2$-capacity and $h(u)d$ is measurable with respect to $\mu_d$ then one can deduce that $h(u)\mu_d$ is also diffuse with respect to the $2$-capacity and it belongs to $\mathcal{M}(\Omega,d)$.
	
	\medskip
	
	The function $u$ satisfies
	\begin{equation*} 
		\displaystyle \int_{\Omega}A(x)\nabla u \cdot \nabla \varphi =\int_{\Omega} h(u)\varphi \mu_d + h(\infty)\int_{\Omega} \varphi \mu_c, \ \ \ \forall \varphi \in C^1_c(\Omega),
	\end{equation*}
	 where we fix  $\varphi= \eta_k \phi$ where $\phi \in C^1_0(\overline{\Omega})$ such that $\operatorname{div}(A(x)\nabla \phi)\in L^\infty(\Omega)$ and $0\le\eta_k\le1$ is such that $\eta_k\in C^1_c(\Omega)$,  and $\eta_k=1$ when $d(x) > \frac{1}{k}$, $||\nabla \eta_k||_{L^\infty(\Omega)}\le k$ and $||\operatorname{div}(A(x)\nabla \eta_k)||_{L^\infty(\Omega)}\le Ck^2$. This yields to
	\begin{equation}\label{uniquenessteo1} 
		\displaystyle \int_{\Omega}A(x)\nabla u \cdot \nabla (\eta_k \phi) =\int_{\Omega} h(u) \eta_k \phi \mu_d + h(\infty)\int_{\Omega}  \eta_k \phi \mu_c, 
	\end{equation}
	and the aim becomes passing the previous to the limit as $k\to\infty$. An application of the Lebesgue Theorem allows to pass to the limit the right-hand of \eqref{uniquenessteo1} since $h(u)\mu_d + h(\infty)\mu_c$ belongs to $\mathcal{M}(\Omega,d)$.

	The left-hand of \eqref{uniquenessteo1} can be written as
	\begin{equation}
	\begin{aligned}\label{uniquenessteo1lhs} 
		\displaystyle \int_{\Omega}A(x)\nabla u \cdot \nabla (\eta_k \phi) = &- \int_{\Omega}A^*(x) \nabla \eta_k\cdot \nabla \phi u - \int_{\Omega}u\eta_k \operatorname{div}(A^*(x)\nabla \phi)
		\\ &- \int_{\Omega}A^*(x) \nabla \phi\cdot \nabla \eta_k u - \int_{\Omega}u\phi \operatorname{div}(A^*(x)\nabla \eta_k) . 
	\end{aligned}
	\end{equation}
	For the first term on the right-hand of \eqref{uniquenessteo1lhs} one can observe that it follows from \eqref{op2_weakdef2} that
	\begin{equation*}
		\displaystyle \lim_{k\to\infty}\int_{\Omega}|A^*(x) \nabla \eta_k\cdot \nabla \phi u| \le  \lim_{k\to\infty}\beta Ck\int_{\{x \in \Omega : d(x)< \frac{1}{k}\}}u=0.
	\end{equation*}
	The same reasoning applies also for the third term on the right-hand of \eqref{uniquenessteo1lhs}.
	
	An application of the Lebesgue Theorem allows to pass to the limit 
	in the second term on the right-hand of \eqref{uniquenessteo1lhs}. For the fourth term one has
	\begin{equation*}
		\displaystyle\lim_{k\to\infty}\int_{\Omega}|u\phi \operatorname{div}(A(x)\nabla \eta_k)| \le \lim_{k\to\infty}\beta Ck^2\int_{\{x \in \Omega : d(x)< \frac{1}{k}\}}u|\phi|\le \lim_{k\to\infty}\beta Ck\int_{\{x \in \Omega : d(x)< \frac{1}{k}\}}u = 0.
	\end{equation*}
	Then the above argument allows to take $k\to\infty$ into \eqref{uniquenessteo1lhs}, yielding to
	\begin{equation}\label{formulazioneconC2} 
		\displaystyle - \int_{\Omega}u \operatorname{div}(A^*(x)\nabla \phi) =\int_{\Omega} h(u) \phi \mu_d + h(\infty)\int_{\Omega} \phi \mu_c,
	\end{equation}
	for every $\phi \in C^1_0(\overline{\Omega})$ such that $\operatorname{div}(A^*(x)\nabla \phi)\in L^\infty(\Omega)$. 
	
	\medskip
	
	Now we are in position to apply \eqref{formulazioneconC2} to the difference of two solutions $v$ and $w$ to \eqref{op2_pbgeneral}, obtaining
	\begin{equation*} 
		\displaystyle - \int_{\Omega}(v-w) \operatorname{div}(A^*(x)\nabla \phi) =\int_{\Omega} (h(v)-h(w)) \phi \mu_d.
	\end{equation*}
		  As $(h(v)-h(w))\mu_d \in \mathcal{M}(\Omega,d)$ is diffuse with respect to the $2$-capacity,  it follows from Lemma \ref{Katoinequalitylocal} that 
	\begin{equation*} 
		\displaystyle - \int_{\Omega}(v-w)^+ \operatorname{div}(A^*(x)\nabla \varphi) =\int_{\{v-w\ge 0\}} (h(v)-h(w)) \varphi \mu_d\le 0.
	\end{equation*}		
	for every $\varphi\in C^1_c(\Omega)$ such that $\operatorname{div}(A^*(x)\nabla \varphi)\in L^\infty(\Omega)$.
	Now a reasoning analogous to what done above in proving \eqref{formulazioneconC2}, it allows to deduce that
	\begin{equation*} 
		\displaystyle - \int_{\Omega}(v-w)^+ \operatorname{div}(A^*(x)\nabla \xi) \le 0,
	\end{equation*}		
	for every $0<\xi\in C^1_0(\overline{\Omega})$ such that $\operatorname{div}(A^*(x)\nabla \xi)\in L^\infty(\Omega)$.	Finally, fixing $\xi$ as defined in \eqref{xi}, one gets that $v\le w$. The proof concludes by switching $v$ and $w$ in order to obtain that $v=w$ almost everywhere in $\Omega$.
\end{proof}

\appendix

\section{Radon measures and capacities}
\label{app:radon}

 In this appendix we provide  some useful results regarding Radon measures and capacities for which we mainly refer to \cite{hkm} and \cite{dms}.

\medskip

 Let recall that, if $\Omega$ is an open bounded subset of $\rn$, $N\geq 1$, then  $\mathcal{M}(\Omega)$ denotes the space of the real valued Borel measures $\mu$ with bounded total variation $|\mu|(\Omega)$. 

\medskip

Let us briefly  list some notions   widely used along the paper: 
\begin{defin}\label{defconcentrated}
	The measure $\mu \in \mathcal{M}(\Omega)$ is said to be {\itshape concentrated} on a Borel subset $E$ of $\Omega$ if, for every Borel set $B\subseteq \Omega$, then 
	$$\mu(B)=\mu(B \cap E).$$ 
	\label{concdef}
\end{defin}
\begin{defin}
	Let $\mu, \lambda \in \mathcal{M}(\Omega)$ then $\mu$ is said to be {\it diffuse} with respect to $\lambda$ if
	$$\lambda (E)=0 \text{ implies }\mu (E)=0.$$
We denote this property by $\mu \ll \lambda$.
	\label{diffdef}
\end{defin}
\begin{defin}
	The measure $\mu, \lambda \in \mathcal{M}(\Omega)$ are said to be {\itshape orthogonal} if there exists a set $E\subset \Omega$ such that
	$$\mu(E)=0 \text{ and } \lambda = \lambda \mathcal{b}_E.$$
	\label{ortdef}
	We denote this property as $\mu \bot \lambda$.
\end{defin}
\begin{defin} 
	A sequence $\mu_n\in \mathcal{M}(\Omega)$ converges in the  {\it narrow topology of measures} to $\mu$ if
	$$\lim_{n\to\infty}\io \varphi \mu_n=\io \varphi \mu, \quad\forall\varphi\in C_b(\Omega).$$
\end{defin}
\begin{remark}
	\noindent  It is possible to prove, for nonnegative measures,  that $\mu_n$ narrow converges to $\mu$ if and only if $\mu_n$ converges to $\mu$  $*$-weakly in $\mathcal{M} (\Omega)$ and $|\mu_n|(\Omega)$ converges to $|\mu|(\Omega)$.
\triang \end{remark}
\noindent We present the following decomposition theorem:
\begin{theorem}	\label{secdec}
	Let $\mu, \lambda \in \mathcal{M}(\Omega)$, then there exists a unique pair $(\mu_0,\mu_1) \in [\mathcal{M}(\Omega)]^2$ such that 
	$$\mu=\mu_0+\mu_1, \ \ \text{where }\mu_0 \ll \lambda \ \text{and}\ \ \mu_1 \bot \lambda.$$
\end{theorem}
\noindent We also recall a well known approximation scheme for Radon measures.

\begin{lemma}\label{lem_approssimazione}
	Let $0\leq\mu\in\mathcal{M}(\Omega)$ then there exists a sequence $0\leq g_n\in C^{\infty}(\Omega)$ such that 
	$\|g_n\|_{L^1(\Omega)}\leq\|\mu\|_{\mathcal{M}(\Omega)}$ and $g_n\to\mu$ in in the narrow topology of measures.
\end{lemma}

Let us introduce definition and properties of the $p$-capacity of a subset of $\Omega$. 

\begin{defin}\label{a8}   
	Let $1\le p< \infty$, and let $K \subset\subset \Omega$ be a compact set. The {\it $p$-capacity} is defined as  
	$$\dis \operatorname{cap}_p(K)=\inf\left\{\int_{\Omega}|\nabla \varphi |^p, \text{ with } \varphi \in C^{\infty}_c(\Omega) \text{ and } \varphi \ge 1 \text{ in } K\right\}.$$
	It can be considered the capacity of any subset in the following way:	if $A \subset \Omega$ is open then
	$$\operatorname{cap}_p(A)=\sup \{ \operatorname{cap}_p(K): K \subset\subset  A\},$$
	and,  more generally, if $A \subset \Omega$ is any Borel set, then
	$$\operatorname{cap}_p(A)=\sup \{ \operatorname{cap}_p(U): U \subset \Omega \text{ open and such that } A \subset U\},$$
	where is assumed $\inf \emptyset = \infty$ by convention.
	\label{capacity}
\end{defin}

By a truncation argument it is also possible to characterize the $p$-capacity of any Borel set $A\subseteq \Omega$, as
\begin{equation*}
\operatorname{cap}_p (A)=\inf\left\{\int_\Omega |\nabla v|^{p}: 0\leq v\in W^{1,p}_0 (\Omega),\ v=1\  \text{a.e. on}\ A \right\}.
\end{equation*}

  It can be shown that the capacity is an outer measure and that the following properties hold: 
\begin{theorem}
	The map $\operatorname{cap}_p:E\subset \Omega \to [0,\infty]$ satisfies:
	\begin{itemize}
		\item[i)] $\operatorname{cap}_p(\emptyset)=0$;
		\item[ii)] if $E_{1}\subseteq E_{2}$, then $\operatorname{cap}_p(E_1) \le \operatorname{cap}_p(E_2)$;
		\item[iii)] if $E=\displaystyle\bigcup_{n=1}^{\infty} E_{n}$, then $\operatorname{cap}_p(E) \le \displaystyle\sum_{n=1}^{\infty}\operatorname{cap}_p(E_n)$.
	\end{itemize}
\end{theorem}

 Let us specify what we mean by a  measure to be  diffuse with respect to $p$-capacity. 
\begin{defin}\label{defdiffuse}
	$\mu \in \mathcal{M}(\Omega)$ is a {\it diffuse measure} with respect to the $p$-capacity if for every Borel set $A \subset \Omega$ such that $\operatorname{cap}_p(A)=0$, then $\mu(A)=0$. 
	 This property is denoted by $\mu \ll \operatorname{cap}_p$ and the set of diffuse measures with respect to $p$-capacity is denoted by $\mathcal{M}^p_0(\Omega)$.
\end{defin}
Hence, in the above definition, by diffuse we mean that the measure $\mu$ does not charge set of zero $p$-capacity.

\noindent We need to define the concept of  cap$_p$-quasi continuous functions.
\begin{defin} 
	A function $u:\Omega\to\R$ is said to be {\it cap$_p$-quasi continuous} if for every $\varepsilon>0$ there exists an open set $E\subset\Omega$ such that $\operatorname{cap}_p(E)<\varepsilon$ and $u$ is continuous in $\Omega\setminus \overline{E}$.
\end{defin}

Broadly speaking the $p$-capacity plays, for functions in $W^{1,p}(\Omega)$, the same role played by the Lebesgue measure for measurable functions. In fact, it can be shown that any function $u$ in $W^{1,p}(\Omega)$ admits a cap$_p$-quasi continuous representative $\widehat{u}$ defined cap$_p$-almost everywhere in $\Omega$, i.e. outside a set of zero $p$-capacity. When dealing with a function $u \in W^{1,p}(\Omega)$, we will always consider its cap$_p$-quasi continuous representative. 

\medskip

The following two  decomposition theorems hold:

\begin{theorem}\label{decmeas}
	Let $\mu\in \mathcal{M}(\Omega)$. Then $\mu$ can be uniquely decomposed as
	$$\mu=\mu_d+\mu_c,$$
	\noindent where $\mu_d$ is diffuse with respect to the $p$-capacity and $\mu_c$ is concentrated on a set of zero $p$-capacity. Moreover, if $\mu \ge 0$, then $\mu_d,\mu_c\ge 0$.
\end{theorem}

\begin{proof}
	See Lemma $2.1$ \cite{fuku}.
\end{proof}

\begin{theorem}\label{diffuse}
	Let $\mu \in \mathcal{M}(\Omega)$. Then 
	$\mu\in \mathcal{M}^p_0(\Omega)$   if and only if   $\mu\in L^1(\Omega)+W^{-1,p'}(\Omega)$.
\end{theorem}

\begin{proof}
	See Theorem $2.1$ of \cite{bgo}.
\end{proof}

Observe that, as $L^1(\Omega)\cap W^{-1,p'}(\Omega)\neq\{0\}$, the decomposition given by Theorem \ref{diffuse} is not unique.

\medskip

In view of the previous theorem, measures which are diffuse with respect to the $p$-capacity can be straightforwardly  approximated as follows:
\begin{proposition}\label{approssimazionediffuse}
	Let $\mu \in\mathcal{M}^p_0(\Omega)$ such that $\mu= f - \operatorname{div}(G)$.
	Then there exists a sequence of nonnegative functions $\mu_n$ in $L^2(\Omega)$ with
	$$\mu_n= f_n - \operatorname{div}(G_n)$$
	where $f_n\in L^2(\Omega)$ converges to $f$ weakly $L^1(\Omega)$ while $G_n$ converges strongly to $G$ in $(L^{p'}(\Omega))^N$.
\end{proposition}

A monotone  approximation result for nonnegative measures also holds: 
\begin{proposition}\label{approssimazionediffusecrescente}
	Let $0\le\mu\in\mathcal{M}^p_0(\Omega)$ then there exists an increasing sequence $0\le\mu_n\in W^{-1,{p'}}(\Omega)$
	such that $\mu_n$ converges to $\mu$ strongly in $\mathcal{M}(\Omega)$.
\end{proposition}
\begin{proof} 
	See  \cite[Lemme $4.2$]{bapi}.
\end{proof}

Here we collect some useful technical results that can be found  in \cite{dmop}:
\begin{proposition}\label{linfmud}
	Let $\mu_d\in\mathcal{M}^p_0(\Omega)$ and let $v\in W^{1,p}_0(\Omega)$. Then the cap$_p$-quasi continuous representative of $v$ is measurable with respect to $\mu_d$. If moreover $v$ belongs to $L^{\infty}(\Omega)$, then the cap$_p$-quasi continuous representative of $v$ belongs to $L^{\infty}(\Omega;\mu_d)$ (and hence to $L^1(\Omega;\mu_d)$).
\end{proposition}

\begin{lemma}\label{dalmaso}
	Let $\lambda\in\mathcal{M}(\Omega)$ be nonnegative and concentrated on a set $E$ such that $\operatorname{cap}_p(E)=0$. Then, for every $\eta>0$, there exists a compact subset $K_{\eta}\subset E$ and a function $\Psi_{\eta}\in C^{\infty}_c(\Omega)$ such that the following hold
	$$\lambda(E\setminus K_{\eta})<\eta,\;\text{in}\;\;\Omega \ \ \text{with}\ \ \;\;\Psi_{\eta}\equiv1\;\text{in}\;K_{\eta}, \ \ \ \displaystyle 0\le \Psi_\eta\le 1,$$
	$$
	\lim_{\eta\to 0^+}||\Psi_{\eta}||_{W^{1,p}_0(\Omega)}=0\,,
	$$
	and, in particular
\begin{equation*}\label{serve}
	 \ \ \ 0\le \int_{\Omega} (1-\Psi_\eta)\ d \lambda \le \eta, \ \text{and}\ \ \ \int_{\Omega} |\nabla \Psi_\eta|^p\le \eta.
\end{equation*} 
\end{lemma}

\begin{lemma}\label{dalmaso3}
	 Let $\mu\in \mathcal{M}_0^p(\Omega)$ and let $u\in W^{1,p}_0(\Omega)\cap L^\infty(\Omega)$. Then, up to the choice of its cap$_p$-quasi continuous representative, $u\in L^\infty(\Omega,\mu)$ and
$$\io u \mu\leq ||u||_{L^\infty(\Omega)}|\mu|(\Omega).$$
\end{lemma}

As in some cases we deal with measurable functions whose truncations have finite energy we state the following result:

\begin{lemma}\label{dalmaso2}
	Let $u:\Omega\to\mathbb{R}$ be a measurable function  almost everywhere finite on $\Omega$ such that $T_k(u)\in W^{1,p}_0(\Omega)$ for every $k>0$. Then there exists a measurable function $v:\Omega\to\mathbb{R}^N$ such that
	$$\nabla T_k(u)=v\chi_{\{|u|\leq k\}},$$ 
	and we define the gradient of $u$ as $\nabla u=v$.  
\end{lemma}

\section{The Hopf Lemma}

In order to be self-contained we give the proof of the Hopf Lemma for an operator in divergence form which is widely used in the manuscript. 
\begin{lemma}[\textbf{Hopf's Lemma}]\label{hopf}
	Let $A$ be a bounded elliptic matrix with coercivity $\alpha>0$ and coefficients $a_{ij}\in C^{0,1}(\overline{\Omega})$. Let $u\in C^1(\overline{\Omega})$ such that
	\begin{equation}\label{ipohopf0}
		-\operatorname{div}(A(x)\nabla u) \le 0 \text{   in   } \Omega,
	\end{equation}
	and assume that there exists a point $x_0\in \partial\Omega$ such that
	\begin{equation}\label{ipohopf} u(x_0)> u(x) \text{  for  all  $x$ in neighbourhood of $x_0$ }.\end{equation}
	Moreover let assume that $\Omega$ satisfies the interior ball condition at $x_0$ (i.e. there exists an open ball $B$ contained in $\Omega$ with $x_0\in \partial B$).
	\\Then
	$$\frac{\partial u}{\partial \nu}(x_0)>0.$$
\end{lemma}
\begin{proof}
	 Without loss of generality, let us assume that, for $r>0$, $B_r(0)\subset \Omega$ is such that $x_0\in \partial B_r(0)$ and \eqref{ipohopf} holds for any  $x\in B_r(0)$.  
	Let consider the following function
	$$v(x)=e^{-\lambda|x|^2} - e^{-\lambda r^2},$$
	for $x\in B_r(0)$. Also, using  the regularity of $A$, let $L>0$ such that $|(a_{i j} (x))_{x_i}|\leq L$, a. e. in $\Omega$ for any $i,j=1, ..., N$. 
	Let $\delta^{ij}$ be the usual Kronecker symbol, then we have 
	\begin{equation}\begin{aligned}\label{hopf2}
		-\operatorname{div}(A(x)\nabla v)
		&= \sum_{i,j = 1}^N - \frac{\partial}{\partial x_i} \left (a_{ij}(x)\frac{\partial v}{\partial x_j}\right ) =  \sum_{i,j = 1}^N \left(- \frac{\partial a_{ij}(x)}{\partial x_i} \frac{\partial v}{\partial x_j} - a_{ij}(x)\frac{\partial v}{ \partial x_i\partial x_j}\right )\\  
		&=e^{-\lambda|x|^2} \sum_{i,j = 1}^N \left( \frac{\partial a_{ij}(x)}{\partial x_i} 2\lambda x_j + a_{ij}(x)(-4\lambda^2 x_i x_j + 2\lambda\delta^{ij})\right ) \\  
		&\le e^{-\lambda|x|^2}  \left(  2LN^2 \lambda |x| - 4 \alpha\lambda^2 |x|^2 + 2\lambda\textbf{Tr}A\right )\le 0,
	\end{aligned}
	\end{equation}
	where the last inequality in the previous holds taking $\lambda$ sufficiently large and for almost every $x\in B_r(0)\setminus B_{\frac{r}{2}}(0)$.
	 
	Now observe that it follows from \eqref{ipohopf} that one can pick $\varepsilon$ small enough so that  
	$$u(x_0) \ge u(x) + \varepsilon v(x) \ \ \text{   for   a.e. }x \in \partial B_{\frac{r}{2}}(0).$$
	Let us stress that the previous inequality still holds for a.e. $x \in \partial B_r(0)$ as here $v(x)=0$.
	
	From \eqref{ipohopf0} and  \eqref{hopf2},     one has
	$$-\operatorname{div}(A(x)\nabla (u+\varepsilon v - u(x_0))) \le 0$$
	and $u+\varepsilon v - u(x_0)\le 0$ on $\partial B_r(0)$ and on $\partial B_{\frac{r}{2}}(0)$. Thus we can apply the maximum principle in order to deduce 
	$$u+\varepsilon v - u(x_0)\le 0 \ \ x \in B_r(0)\setminus B_{\frac{r}{2}}(0).$$
	Having $u(x_0)+\varepsilon v(x_0) - u(x_0)= 0$ then this implies
	$$\frac{\partial u}{\partial \nu}(x_0) + \varepsilon\frac{\partial v}{\partial \nu}(x_0) \ge 0,$$
	and we conclude by observing
	$$\frac{\partial u}{\partial \nu}(x_0) \ge - \varepsilon\frac{\partial v}{\partial \nu}(x_0) =-\frac{\varepsilon}{r} \nabla v(x_0) \cdot x_0 = 2\lambda r e^{-\lambda r^2}>0.$$
\end{proof}
\begin{remark}
\noindent Observe that the previous applied to $-u$ instead of $u$ gives that the analogous  result holds for $A$-superharmonic  functions $u$, i.e. $$\frac{\partial u}{\partial \nu}(x_0)<0,$$
provided $$ u(x_0)< u(x)  \ \ \text{in a neighbourhood of $x_0$}.   $$
\triang\end{remark}

 As an application of Hopf's Lemma one has the following (see for instance  \cite[Lemma 2]{diazjfa}): 
\begin{lemma}\label{hopfdiaz}
	Let $A$ be a symmetric, elliptic and bounded matrix with coefficients $a_{ij}\in C^{0,1}(\overline{\Omega})$. Then there exist $\lambda_1>0$ and a function $\varphi_{1,A}\in W^{2,p}(\Omega) \cap H^1_0(\Omega)$ for every $p<\infty$ such that
	\begin{equation*}
		\begin{cases}
			\displaystyle -\operatorname{div}(A(x) \nabla \varphi_{1,A})= \lambda_1 \varphi_{1,A} &  \text{in}\ \Omega, \\
			\varphi_{1,A}=0 & \text{on}\ \partial \Omega.
		\end{cases}\label{evpb}
	\end{equation*}
	Moreover it holds
	\begin{equation}
		cd(x) \le \varphi_{1,A}(x) \le Cd(x), \ \ \ \forall x \in \Omega.
		\label{hopfvarphi2}
	\end{equation}
	for two constants $c,C>0$. 
\end{lemma} 

	\begin{remark}\label{remarksuautofunzioneplap}
	 Let us stress that, thanks to the previous lemma and since $\int_{\Omega} d^r <\infty$ if and only if $r>-1$, one deduces that
	\begin{equation}
		\displaystyle \int_{\Omega}\varphi_{1,A}^r < \infty \quad \text{if and only if} \quad r>-1.
		\label{hopfvarphi}
	\end{equation}
	Let us also mention that, throughout this paper, we use that \eqref{hopfvarphi2} also holds for $\varphi_{1,p}$, i.e. for solutions to the nonlinear $p$-Laplace eigenvalue problem \eqref{not:phi1p}.  This result is classical (see for instance \cite[Proposition 3.1]{sak} or  \cite[Lemma $3.2$]{loc}). Then \eqref{hopfvarphi} holds even with  $\varphi_{1,p}$ in place of  $\varphi_{1,A}$. \triang
	\end{remark}

\section*{Acknowledgments} 
The authors  are partially supported by the Gruppo Nazionale per l’Analisi Matematica, la Probabilità e le loro Applicazioni (GNAMPA) of the Istituto Nazionale di Alta Matematica (INdAM).

\section*{Conflict of interest declaration}
The authors declare no competing interests.

\section*{Data availability statement}
 We do not analyse or generate any datasets, because our work
proceeds within a theoretical and mathematical approach. One
can obtain the relevant materials from the references below.

\section*{Ethical Approval} not applicable


\begin{thebibliography}{100}
\bibitem[AcShPe]{asp} A. Acrivos, M.J. Shah and E.E. Petersen, On the flow of non-Newtonian liquid on a rotating disk, J. Appl. Phys. 31 (1960), 963-968 

\bibitem[AdGiSa]{ags}  A. Adimurthi, J. Giacomoni and S. Santra, Positive solutions to a fractional equation with singular nonlinearity, J. Differential Equations 265 (4) (2018), 1191-1226

\bibitem[AmFuPa]{afp} L. Ambrosio, N. Fusco and D. Pallara, Functions of Bounded Variation and Free Discontinuity Problems, Oxford Mathematical Monographs, 2000

\bibitem[ArCaLeMaOrPe]{a6} D. Arcoya, J. Carmona, T. Leonori, P.J. Mart\'inez-Aparicio, L. Orsina and F. Petitta,
Existence and nonexistence of solutions for singular quadratic quasilinear equations, J. Differential Equations  246 (2009), 4006-4042 

\bibitem[ArBoLePo]{ABLP} D. Arcoya, L. Boccardo, T. Leonori and A. Porretta, Some elliptic problems with singular natural growth lower order terms, J. Differential Equations 249 (11) (2010), 2771-2795

\bibitem[ArMo]{am} D. Arcoya and L. Moreno-M\'erida, Multiplicity of solutions for a Dirichlet problem with a strongly singular nonlinearity, Nonlinear Anal. 95 (2014), 281-291	

\bibitem[BaPaZe]{bpz}Y. Bai, N.S. Papageorgiou and  S. Zeng, Parametric singular double phase Dirichlet problems,
Adv. Nonlinear Anal. 12 (1)  (2023) Paper No. 20230122, 20 pp 


\bibitem[BaPi]{bapi}  P. Baras and M. Pierre, Singularit\'es \'eliminables pour des \'equations semi-lin\'eaires, Ann. Inst. Fourier (Grenoble) 34 (1984), 185-206

\bibitem[BaDeMePe]{bdmp} B. Barrios, I. De Bonis,  M. Medina and I. Peral,  Semilinear problems for the fractional laplacian with a singular nonlinearity, Open Math. 13 (1) (2015), 390-407

\bibitem[BeBoGaGaPiVa]{b6} P. Benilan, L. Boccardo, T. Gallou\"{e}t, R.Gariepy, M. Pierre and J.L. Vazquez, An $L^1$ theory of existence and uniqueness of nonlinear elliptic equations,  Ann. Sc. Norm. Super. Pisa Cl. Sci. (4) 22 (1995), 240-273

\bibitem[Bl]{bl} D. Blanchard, Truncations and monotonicity methods for parabolic equations, Nonlinear Anal. 21 (10)  (1993), 725-743

\bibitem[BlMu]{blmu} D. Blanchard and F. Murat, Renormalised solutions of nonlinear parabolic problems with $L^1$ data: existence and uniqueness, Proc. Roy. Soc. Edinburgh Sect. A 127 (1997), 1137-1152

\bibitem[Bla]{Bla} H. Blasius, Grenzschichten in Fliissigkeiten mit Kleiner Reibung, Z. Math. Phys. 56 (1908), 1-37 

\bibitem[Bo]{B} L. Boccardo, Dirichlet problems with singular and gradient quadratic lower order terms, ESAIM Control Optim. Calc. Var. 14 (2008), 411-426

\bibitem[BoCa]{boca} L. Boccardo and J. Casado-D\'iaz, Some properties of solutions of some semilinear elliptic singular problems and applications to the $G$-convergence, Asymptot. Anal. 86 (2014), 1-15

\bibitem[BoDiGiMu]{bdgm1} L. Boccardo, J.I. D\'iaz, D. Giachetti and F. Murat, Existence of a solution for a weaker form of a nonlinear elliptic equation. Recent advances in nonlinear elliptic and parabolic problems, Pitman Res. Notes Math. Ser. 208 (1988), 229-246

\bibitem[BoDiGiMu2]{bdgm2} L. Boccardo, J.I. D\'iaz, D. Giachetti and F. Murat, Existence and regularity of renormalized solutions for some elliptic problems involving derivatives of nonlinear terms, J. Differential Equations 106 (2) (1993) 215-237

\bibitem[BoGa]{bg} L. Boccardo and T. Gallou\"et, Nonlinear elliptic and parabolic equations involving measure data, J. Funct. Anal. 87 (1989), 149-169

\bibitem[BoGaOr]{bgo} L. Boccardo, T. Gallou\"et and L. Orsina, Existence and uniqueness of entropy solutions for nonlinear elliptic equations with measure data, Ann. Inst. H. Poincar\'e Anal. Non Lin\'eaire 13 (1996), 539-551

\bibitem[BoMu]{bm} L. Boccardo and F. Murat, Almost everywhere convergence of the gradients of solutions to elliptic and parabolic equations, Nonlinear Anal. 19 (1992), 581-597

\bibitem[BoMuPu]{bmp} L.Boccardo, F. Murat and J.P. Puel, Existence of bounded solutions for nonlinear unilateral problems, Ann. Mat. Pura Appl. 152 (1988), 183-196

\bibitem[BoOr]{bo} L. Boccardo and L. Orsina, Semilinear elliptic equations with singular nonlinearities, Calc. Var. Partial Differential Equations 37 (2009), 363-380

\bibitem[BoGiHe]{bougia} B. Bougherara, J. Giacomoni and J. Hern\'andez, Existence and regularity of weak solutions for singular elliptic problems, Proceedings of the 2014 Madrid Conference on Applied Mathematics in honor of Alfonso Casal,  Electron. J.Differ. Equ. Conf. 22 (2015), 19-30

\bibitem[BrChTr]{bct} B. Brandolini,   F. Chiacchio and C. Trombetti, Symmetrization for singular semilinear elliptic equations, Ann. Mat. Pura Appl. 193 (2014), 389-404

\bibitem[Br]{Br} H. Brezis, Nonlinear elliptic equations involving measures, Contributions to nonlinear partial differential equations (Madrid, 1981), Pitman Res. Notes in Math. 89 (1983), 82-89

\bibitem[BrPo]{bp} H. Brezis and A.C. Ponce, Kato's inequality when $\Delta u$ is a measure, C. R. Math. Acad. Sci. Paris 338 (8) (2004), 599-604

\bibitem[Ca]{c} A. Canino, Minimax methods for singular elliptic equations with an application to a jumping problem, J. Differential Equations 221 (1) (2006), 210-223

\bibitem[CaDe]{cade} A. Canino and M. Degiovanni, A variational approach to a class of singular semilinear elliptic equations, J. Convex Anal. 11 (1) (2004), 147-162

\bibitem[CaMoScSq]{CaMoScSq} A. Canino, L.  Montoro, B. Sciunzi and M.  Squassina,  Nonlocal problems with singular nonlinearity, 
Bull. Sci. Math. 141 (3) (2017),  223-250
 
\bibitem[CaScTr]{CST} A. Canino, B. Sciunzi and A. Trombetta, Existence and uniqueness for {$p$}-Laplace equations
involving singular nonlinearities, NoDEA Nonlinear Differential Equations Appl. 23 (2) (2016), 8-18

\bibitem[CaMa]{cma} J. Carmona and P.J. Mart\'inez-Aparicio,  A singular semilinear elliptic equation with a variable exponent, Adv. Nonlinear Stud. 16 (3) (2016), 491-498 

\bibitem[CaMaMaMa]{cmmt} J.  Carmona, A.J.  Martínez Aparicio,  P.J.  Martínez-Aparicio and  M. Martínez-Teruel, Regularizing effect in singular semilinear problems,
Math. Model. Anal. 28 (4) (2023),  561-580

\bibitem[CiSkVe]{csv} S. Ciani, I. Skrypnik and V. Vespri, On the local behavior of local weak solutions to some
singular anisotropic elliptic equations, Adv. Nonlinear Anal. 12 (1) (2023),   237-265

\bibitem[CoCo]{CoCo} 	 G.M. Coclite and M.M. Coclite, 
Positive solutions for an integro-differential equation with singular nonlinear term, Differential Integral Equations 18 (9) (2005), 1055-1080

\bibitem[CoCo2]{CoCo2} G.M. Coclite and M.M. Coclite, On a Dirichlet problem in bounded domains with singular nonlinearity, Discrete Contin. Dyn. Syst. 33 (2013), 4923-4944

\bibitem[CoPa]{CoPa} M.M. Coclite and G. Palmieri, On a singular nonlinear Dirichlet problem, Comm. Partial Differential Equations 14 (10) (1989), 1315-1327

\bibitem[CrRaTa]{crt} M.G. Crandall, P.H. Rabinowitz and L. Tartar,  On a dirichlet problem with a singular nonlinearity, Comm. Partial Differential Equations 2 (1977), 193-222

\bibitem[Cr]{crocco} L.  Crocco, The Laminar Boundary Layer in Gases,  Monografie Scientifiche di Aeronautica 3 (1946), translated by I. Hodes and J. Castelfranco. Rep. CF 1038, North American Aviation, (1948)

\bibitem[DaMuOrPr]{dmop} G. Dal Maso, F. Murat, L. Orsina and A. Prignet, Renormalized solutions of elliptic equations with general measure data, Ann. Sc. Norm. Sup. Pisa Cl. Sci. (4) (1999), 741-808

\bibitem[DaSk]{dms} G. Dal Maso and I.V. Skrypnik,
Capacity theory for monotone operators, Potential Anal. 7 (4) (1997), 765-803

\bibitem[Dal]{dall} A. Dall'Aglio, Approximated solutions of equations with $L^1$ data. Application to the H-convergence of quasi-linear parabolic equations, Ann. Mat. Pura Appl. (4) 170 (1996), 207-240

\bibitem[De]{DCA} L.M. De Cave, Nonlinear elliptic equations with singular nonlinearities, Asymptot. Anal. 84 (3-4) (2013), 181-195

\bibitem[DeDuOl]{ddo} L.M. De Cave, R. Durastanti and F. Oliva, Existence and uniqueness results for possibly singular nonlinear elliptic equations with measure data, NoDEA Nonlinear Differential Equations Appl. (2018) 25:18    

\bibitem[DeOl]{do} L.M. De Cave and F. Oliva, Elliptic equations with general singular lower order terms and measure data, Nonlinear Anal. 128 (2015), 391-411

\bibitem[DeGiOlPe]{DGOP} V. De Cicco, D. Giachetti, F. Oliva and  F. Petitta, The Dirichlet problem for singular elliptic equations with general nonlinearities, Calc. Var. Partial Differential Equations 58 (4) (2019), Paper No. 129

\bibitem[DeOlSe]{DeOlSe}F. Della Pietra, F. Oliva and S. Segura de León, On a nonlinear Robin problem with an absorption term on the boundary and $L^1$ data, Advances in Nonlinear Analysis 13 (1) (2024), 20230118  

\bibitem[dP]{dp}  M.A. del Pino,   
A global estimate for the gradient in a singular elliptic boundary value problem, Proc. Roy. Soc. Edinburgh Sect. A 122 (3-4) (1992), 341-352

\bibitem[DiLi]{dpl} R.J. Di Perna and P.L. Lions, On the Cauchy problem for Boltzmann equations: global  existence and weak stability, Ann. of Math. 130 (1989), 321-366

\bibitem[DiHeRa]{dhr} J.I. D\'iaz, J. Hern\'andez and J.M. Rakotoson, On very weak positive solutions to some semilinear elliptic problems with simultaneous singular nonlinear and spatial dependence terms, Milan J. Math. 79 (2011), 233-245
 
\bibitem[DiMoOs]{DMO} J.I. D\'iaz, J.M. Morel and L. Oswald, An elliptic equation with singular nonlinearity, Comm. Partial Differential Equations 12 (1987), 1333-1344

\bibitem[DiRa]{diazjfa} J.I. D\'iaz and J.M. Rakotoson, On the differentiability of very weak solutions with right-hand side data integrable with respect to the distance to the boundary, J. Funct. Anal. 257 (2009), 807-831

\bibitem[Du]{dura} R. Durastanti, Asymptotic behavior and existence of solutions for singular elliptic equations, Ann. Mat. Pura Appl. (4) 199 (3) (2020), 925-954

\bibitem[DuGi]{dugi} R. Durastanti and L. Giacomelli, Spreading equilibria under mildly singular potentials: Pancakes Versus Droplets,  J. Nonlinear Sci. 32 (2022), Paper No. 71

\bibitem[ElDiRa]{edr} N. El Berdan, J.I. D\'iaz and J.M. Rakotoson, The Uniform Hopf Inequality for discontinuous coefficients and optimal regularity in bmo for singular problems, J. Math. Anal. Appl. 437 (2016), 350-379

\bibitem[EsSc]{EsSc} F. Esposito and B.  Sciunzi, 
On the Höpf boundary lemma for quasilinear problems involving singular nonlinearities and applications, J. Funct. Anal. 278 (4) (2020),  108346, 25 pp
 
\bibitem[FeRo]{fernandezros} X. Fern\'andez-Real and X. Ros-Oton,
Regularity theory for elliptic PDE, Zur. Lect. Adv. Math., 28, EMS Press, Berlin, (2022),  viii+228 pp

\bibitem[FeMeSe]{FeMeSe} A. Ferone, A.  Mercaldo, S.  Segura de León,  
A singular elliptic equation and a related functional, 
ESAIM Control Optim. Calc. Var. 27 (39) (2021), Paper No. 39, 17 pp

\bibitem[FeMeSe2]{fms2}A. Ferone, A.  Mercaldo and S.  Segura de León,  Singular elliptic equations having a gradient term with natural growth,  arXiv:2401.06237, 2024

\bibitem[Fo]{fo} R.H. Fowler, The solution of Emden’s and similar differential equations, Monthly Notices Roy. Astro. Soc.
91 (1930), 63-91

\bibitem[FuMa]{fulks} W. Fulks and J.S. Maybee, A singular nonlinear equation, Osaka J. Math. 12 (1960), 1-19

\bibitem[FuSaTa]{fuku} M. Fukushima, K. Sato and S. Taniguchi, On the closable part of pre-Dirichlet forms and the fine support of the underlying measures, Osaka J. Math. 28 (1991), 517-535

\bibitem[Ga]{Ga} P. Garain,  On a degenerate singular elliptic problem, 
Math. Nachr. 295 (7) (2022),   1354-1377 

\bibitem[GaPa]{gapa} L. Gasiński and N.S. Papageorgiou, 
Singular equations with variable exponents and concave-convex nonlinearities, Discrete Contin. Dyn. Syst. Ser. S 16 (6) (2023), 1414-1434

\bibitem[GaOlWa]{gow} J.A. Gatica, V. Oliker and P. Waltman, Singular nonlinear boundary-value problems for  second-order ordinary differential equations, J. Differential Equations 79 (1989), 62-78
 
\bibitem[GhRa]{gr} M. Ghergu and V. Radulescu, Subliner singular elliptic problems with two parameters, J. Differential Equations
195 (2) (2003), 520-536 

 \bibitem[GhRa2]{ghra}  M. Ghergu and V. Radulescu, Singular Elliptic Equations, Bifurcation and
Asymptotic Analysis, Oxford Lecture Ser. Math. Appl., vol. 37, The Clarendon Press, Oxford University
Press, Oxford, 2008
 
\bibitem[GiPeSe]{gps2} D. Giachetti, F. Petitta and S. Segura de Le\'on, Elliptic equations having a singular quadratic gradient term and a changing sign datum,  Comm. Pure Appl. Anal. 11 (5) (2012), 1875-1895
 
\bibitem[GiPeSe2]{gps1} D. Giachetti, F. Petitta and S. Segura de Le\'on, A priori estimates for elliptic problems with a strongly singular gradient term and a general datum, Differential Integral Equations 26 (9/10) (2013), 913-948

\bibitem[GiMaMu]{gmm} Giachetti, P.J. Mart\'inez-Aparicio and F. Murat, A semilinear elliptic equation with a mild singularity
at u = 0: Existence and homogenization, J. Math. Pures Appl. 107 (2017), 41-77

\bibitem[GiMaMu2]{GMM} D. Giachetti, P.J. Mart\'inez-Aparicio and F. Murat, Definition, existence, stability and uniqueness of the solution to a semilinear elliptic problem with a strong singularity at $u = 0$, Ann. Sc. Norm. Super. Pisa Cl. Sci. (5) 18 (4) (2018), 1395-1442 

\bibitem[GidSSa]{gss} J. Giacomoni, L.M. dos Santos and  C.A. Santos, Multiplicity for a strongly singular quasilinear problem via bifurcation theory, Bull. Math. Sci. 13  (1) (2023),   Paper No. 2250013, 25 pp

\bibitem[GiSa]{gisa} J. Giacomoni and K. Saoudi, Multiplicity of positive solutions for a singular and critical problem, Nonlinear Anal. 71 (2009), 4060-4077

\bibitem[GiTr]{gt} D. Gilbarg and N.S. Trudinger, Elliptic partial differential equations of second order, Springer-Verlag, Berlin, (1983)

\bibitem[GoGu]{GoGu} T. Godoy and A. Guerin, Existence of nonnegative solutions to singular elliptic problems, a variational approach, 
Discrete Contin. Dyn. Syst. 38 (3) (2018),   1505-1525

\bibitem[GuMa]{GuMa} U. Guarnotta and  S.A. Marano,  Infinitely many solutions to singular convective Neumann systems with arbitrarily growing reactions, 
J. Differential Equations 271 (2021), 849-863

\bibitem[GuMaMo]{GuMaMo} 
U. Guarnotta,  S.A. Marano   and  D.  Motreanu,  On a singular Robin problem with convection terms, 
Adv. Nonlinear Stud. 20 (4) (2020), 895-909

\bibitem[GuLi]{ghl} C. Gui and F. Lin, Regularity of an elliptic problem with a singular nonlinearity, Proc. Roy. Soc. Edinburgh Sect. A 123 (6) (1993), 1021-1029

\bibitem[Ha]{Ha} T. Hara,  
Trace inequalities of the Sobolev type and nonlinear Dirichlet problems, 
Calc. Var. Partial Differential Equations 61 (6) (2022), Paper No. 216, 21 pp

\bibitem[HeKiMa]{hkm} J. Heinonen, T. Kilpel\"ainen and O. Martio, Nonlinear potential theory of degenerate elliptic equations, Oxford University Press, Oxford, (1993)

\bibitem[HeMa]{hema}  J. Hern\'andez and F.J. Mancebo, Singular Elliptic and Parabolic Equations,
Handbook of Differential Equations, Stationary Partial Differential Equations,
volume 3 Edited by M. Chipot and R Quittner, 2006

\bibitem[HoSc]{loc} N. Hoang Loc and K. Schmitt, Boundary value problems for singular elliptic equations, Rocky
Mountain J. Math. 41 (2011), 555-572

\bibitem[Hu]{hunt} R. Hunt, On $L^{(p,q)}$ spaces, Enseign. Math. (2) 12 (1966), 249-276 

\bibitem[KiKiMa]{kkm} T. Kilpel\"ainen, J. Kinnunen and O. Martio, Sobolev Spaces with Zero Boundary Values on Metric Spaces, Potential Anal. 12 (2000), 233-247

\bibitem[LaSh]{ls} A.V. Lair and A.W. Shaker,  Classical and weak solutions of a singular semilinear elliptic problem, J. Math. Anal. Appl. 211 (2) (1997), 371-385

\bibitem[LaOlPeSe]{lops} M. Latorre, F. Oliva, F. Petitta and S. Segura de León, The Dirichlet problem for the 1-Laplacian with a general singular term and $L^1$-data, Nonlinearity 34 (2021), 1791-1816

\bibitem[LaMc]{lm} A.C. Lazer and P.J. McKenna, On a singular nonlinear elliptic boundary-value problem, Proc. Amer. Math. Soc. 111 (1991), 721-730

\bibitem[LeSa]{LeSa} N.Q. Le and O. Savin, Global $C^{1,\beta}$ and $W^{2,p}$ regularity for some singular Monge-Amp\`ere equations,  arXiv:2407.04586, 2024

\bibitem[LeLi]{ll} J. Leray and J.L. Lions, Quelques r\'esultats de V\v{i}sik sur les probl\'emes elliptiques semi-lin\'eaires par les m\'ethodes de Minty et Browder, Bull. Soc. Math. France 93 (1965), 97-107

\bibitem[LiMu]{LiMu} P. L. Lions, F. Murat, Solutions renormalisées d’équations elliptiques non linéaires, unpublished

\bibitem[MaOlPe]{maop} A.J. Martínez Aparicio, F. Oliva and F. Petitta, Optimal global BV regularity for 1-Laplace type BVP's with singular lower order terms, arXiv:2405.13793, 2024

\bibitem[Me]{Met}  A.B. Metzner,  in "Advances in Chemical Engineering," ed. by T. B. Drew and J. W. Hoopes, Jr., Vol. I, 79-150, Academic Press, New York (1956)
 
\bibitem[Mi]{min} G. Mingione, The Calder\'on-Zygmund theory for elliptic problems with measure data, Ann. Sc. Norm. Super. Pisa Cl. Sci. (5) 6 (2) (2007), 195-261

\bibitem[Mo]{mo} A. Mohammed, Existence and estimates of solutions to a singular Dirichlet problem for the Monge-Amp\'ere equation, J. Math. Anal. Appl. 340 (2) (2008), 1226-1234

\bibitem[MoMuSc]{mms}  L. Montoro, L. Muglia and B. Sciunzi,  Classification of solutions to $-\Delta u =u^{-\gamma} $ in the half-space,  Math. Ann. 389 (3) (2024),  3163-3179

\bibitem[MoMuSc2]{mms2}  L. Montoro, L. Muglia and B. Sciunzi, The Classification of all weak solutions to  $-\Delta u =u^{-\gamma} $ in the half-space,  arXiv:2404.03343, 2024 

\bibitem[Mu]{mu} F. Murat, Soluciones renormalizadas de EDP elipticas no lineales, published by 
C.N.R.S., Laboratoire d'Analyse Num\'erique, Univerit\'e P. and M. Curie (Paris VI), (1993)

\bibitem[MuPo]{mupo} F. Murat and A. Porretta, Stability properties, existence and nonexistence of renormalized solutions for elliptic equations with measure data, Comm. Partial Differential Equations 27 (2002), 2267-2310

\bibitem[NaCa]{nc} A. Nachman and A. Callegari, A nonlinear singular boundary value problem in the theory of pseudoplastic fluids, SIAM J. Appl. Math. 38 (2) (1980), 275-281

\bibitem[NaTa]{nt} A. Nachman  and S. Taliaferro, Mass transfer into boundary layers for power law fluids, Proc. R. Soc. Lond. A 365 (1979), 313-326

\bibitem[No]{nowo} P. Nowosad, On the integral equation $\kappa f=1/f$ arising in a problem in communication, J. Math. Anal. Appl. 14 (1966), 484-492

\bibitem[Ol]{Ol} F. Oliva, Existence and uniqueness of solutions to some singular equations with natural growth, Ann. Mat. Pura Appl. 200 (4) (2021), 287-314  

\bibitem[Ol2]{ol2}  F. Oliva, Regularizing effect of absorption terms in singular problems,  J. Math. Anal. Appl. 472 (1) (2019),  1136-1166

\bibitem[OlPe]{OP} F. Oliva and F. Petitta, Finite and infinite energy solutions of singular elliptic problems: existence and uniqueness,  J. Differential Equations 264 (1) (2018), 311-340 

\bibitem[OlPe2]{OP2} F. Oliva and F. Petitta, A nonlinear parabolic problem with singular terms and nonregular data, 
Nonlinear Anal. 194 (2020), 111472, 13 pp
 
\bibitem[OlPeSt]{OPS} F. Oliva and F. Petitta, M.F. Stapenhorst, Existence and non-existence phenomena for nonlinear elliptic equations with $L^1$ data and singular reactions, preprint

\bibitem[OrPe]{orpe} L. Orsina and F. Petitta, A Lazer-McKenna type problem with measures, Differential Integral Equations 29 (1-2) (2016), 19-36

\bibitem[PaRaRe]{prr} N.S. Papageorgiou, V.D. Radulescu and D. Repovs, Positive solutions for nonlinear Neumann
problems with singular terms and convection, J. Math. Pures Appl. (9) 136 (2020), 1-21

\bibitem[PaRaZh]{parazh}  N.S. Papageorgiou, V.  Radulescu, and Y. Zhang, Anisotropic singular double phase Dirichlet problems, Discrete Contin. Dyn. Syst. Ser. S 14 (12) (2021), 4465-4502

\bibitem[Po]{sel} A. C. Ponce, Elliptic PDEs, Measures and Capacities, Tracts in Mathematics 23, European Mathematical Society (EMS), Zürich, 2016, 463 pp

\bibitem[Pr]{prignet} A. Prignet, Remarks on existence and uniqueness of solutions of elliptic problems with right-hand side measures, Rendiconti di Matematica 15 (1995), 321-337

\bibitem[Sa]{sak} S. Sakaguchi, Concavity properties of solutions to some degenerate quasilinear elliptic Dirichlet problems, Ann. Scuola Norm. Sup. Pisa Cl. Sci. (4) 14 (3) (1987), 403-421 


\bibitem[SaGhCh]{SaGhCh} K. Saoudi, S.  Ghosh and  D. Choudhuri,   Multiplicity and Hölder regularity of solutions for a nonlocal elliptic PDE involving singularity, 
J. Math. Phys. 60 (10) (2019), 101509, 28 pp

\bibitem[Sat]{sat} D.H. Satinger, Monotone methods in nonlinear elliptic and parabolic boundary value problems, Indiana Univ. Math. J. 21 (1972), 979-1000

\bibitem[Sc]{Sch} W.R. Schowalter,  The application of boundary‐layer theory to power‐law pseudoplastic fluids: Similar solutions, AIChE Journal 6 (1) (1960), 24-28

\bibitem[Se]{serrin} J. Serrin, Pathological solutions of elliptic differential equations, Ann. Sc. Norm. Sup. Pisa Cl. Sci. (3) 18 (1964), 385-387

\bibitem[Si]{S} G. Singh, Weak solutions for singular quasilinear elliptic systems, Complex Var. Elliptic Equ. 61 (10) (2016), 1389-1408

\bibitem[SoTe]{SoTe} N. Soave  and S.  Terracini, 
The nodal set of solutions to some elliptic problems: singular nonlinearities, J. Math. Pures Appl. (9) 128 (2019), 264-296

\bibitem[St]{st} G. Stampacchia, Le probl\'eme de Dirichlet pour les \'equations elliptiques du seconde ordre \`a coefficientes discontinus, Ann. Inst. Fourier (Grenoble) 15 (1965), 189-258
 
\bibitem[Stu]{stuart} C.A. Stuart, Existence and approximation of solutions of nonlinear elliptic problems, Mathematics Report, Battelle Advanced Studies Center, Geneva, Switzerland, 86 (1976)

\bibitem[SuZh]{sz} Y. Sun and D. Zhang, The role of the power 3 for elliptic equations with negative exponents, Calc. Var. Partial Differential Equations, 49 (3-4) (2014), 909-922

\bibitem[Ta]{t} S. Taliaferro, A nonlinear singular boundary value problem, Nonlinear Anal. 3 (1979), 897-904
 
\bibitem[Tr]{tro} G.M. Troianiello, Elliptic Differential Equations and Obstacle Problems, Plenum Press, New York, (1987)

\bibitem[Ty]{Ty}  J. Tyagi,  
An existence of positive solutions to singular elliptic equations, 
Boll. Unione Mat. Ital. 7 (1) (2014),  45-53

\bibitem[VaSoMo]{vsm} K. Vajravelu, E. Soewono, and R.N. Mohapatra, On solutions of some singular, nonlinear differential equations arising in boundary layer theory, J. Math. Anal. Appl. 155 (2) (1991), 499-512

\bibitem[Ve]{v}  L. Veron, Elliptic Equations Involving Measures, M. Chipot, P. Quittner. Stationary Partial Differential equations, Vol. 1, Elsevier, Handbook of Differential Equations (2004), 593-712 	

\bibitem[Wo]{wo} J. Wong, On the generalized Emden–Fowler equation, SIAM Rev. 17 (1975), 339-360

\bibitem[ZhCh]{zach} Z. Zhang and J. Cheng, Existence and optimal estimates of solutions for singular nonlinear Dirichlet problems, Nonlinear Anal. 57 (3) (2004), 473-484	
\end{thebibliography}
\end{document}